\documentclass[a4paper,reqno]{amsart}
\usepackage[british]{babel}
\usepackage[T1]{fontenc}
\usepackage{amssymb,amsthm,bm}
\usepackage{mathtools}
\usepackage{mathrsfs}
\usepackage{nicefrac}
\mathtoolsset{mathic}
\allowdisplaybreaks
\usepackage{courier,mathptmx,stmaryrd}
\usepackage{longtable}
\usepackage[numbers,sort&compress]{natbib}
\usepackage{hyperref,url}
\hypersetup{
bookmarksnumbered
}
\hypersetup{
breaklinks=false,
pdfborderstyle={/S/U/W 0},
citebordercolor=.235 .702 .443,
urlbordercolor=.255 .412 .882,
linkbordercolor=.804 .149 .19,
}
\hypersetup{
pdfauthor={Gert de Cooman, Arthur Van Camp and Jasper De Bock},
pdftitle={The logic behind desirable sets of things, and its filter representation},
pdfkeywords={desirability, desirable sets of things, conservative inference, propositional logic, filter, prime filter, principal filter}
}
\usepackage[all]{hypcap}
\usepackage{xcolor}
\usepackage{enumitem}
\usepackage{tikz}
\usetikzlibrary{decorations.text}

\newif\ifbibincluded
\bibincludedtrue

\DeclarePairedDelimiter{\group}{(}{)}

\DeclarePairedDelimiter{\set}{\{}{\}}

\DeclarePairedDelimiter{\structure}{\langle}{\rangle}

\DeclareMathOperator{\cl}{Cl}
\DeclareMathOperator{\fin}{fin}
\DeclareMathOperator{\finty}{fty}
\DeclareMathOperator{\compl}{co}
\DeclareMathOperator{\up}{up}
\DeclareMathOperator{\Up}{Up}

\DeclareMathOperator{\posi}{posi}
\DeclareMathOperator{\sds}{\cohsdssymbol}

\newcommand{\sdfs}{\smash{\hat{\sds}}}

\newcommand{\lattice}{L}
\newcommand{\meet}{\frown}
\newcommand{\join}{\smile}
\newcommand{\smallest}[1][\lattice]{\smash{0_{#1}}}
\newcommand{\largest}[1][\lattice]{\smash{1_{#1}}}
\newcommand{\latticetop}{\largest}
\newcommand{\latticebottom}{\smallest}
\newcommand{\cohsdttop}{\largest[\toppedcohsdts]}
\newcommand{\cohsdtbottom}{\smallest[\toppedcohsdts]}
\newcommand{\cohsdstop}{\largest[\toppedcohsdss]}
\newcommand{\cohsdsbottom}{\smallest[\toppedcohsdss]}
\newcommand{\fincohsdstop}{\largest[\toppedfincohsdss]}
\newcommand{\fincohsdsbottom}{\smallest[\toppedfincohsdss]}
\newcommand{\cohsdfstop}{\largest[\toppedcohsdfss]}
\newcommand{\cohsdfsbottom}{\smallest[\toppedcohsdfss]}
\newcommand{\eventprincipalfiltersbottom}{\smallest[\eventprincipalfilters]}
\newcommand{\eventprincipalfilterstop}{\largest[\eventprincipalfilters]}
\newcommand{\fineventfiltersbottom}{\smallest[\fineventfilters]}
\newcommand{\fineventfilterstop}{\largest[\fineventfilters]}
\newcommand{\finfineventfiltersbottom}{\smallest[\finfineventfilters]}
\newcommand{\finfineventfilterstop}{\largest[\finfineventfilters]}

\newcommand{\upset}[3][]{\smash{\up_{{#1}}}\group[#3]{#2}}
\newcommand{\Upset}[3][]{\smash{\Up_{{#1}}}\group[#3]{#2}}
\newcommand{\posetupset}[2][]{\upset[\lattice]{#2}{#1}}
\newcommand{\posetUpset}[2][]{\Upset[\lattice]{#2}{#1}}
\newcommand{\sotupset}[2][]{\upset[\setsofthings]{#2}{#1}}
\newcommand{\sotUpset}[2][]{\Upset[\setsofthings]{#2}{#1}}
\newcommand{\cohsdtupset}[2][]{\upset[\cohsdts]{#2}{#1}}
\newcommand{\cohsdtUpset}[2][]{\Upset[\cohsdts]{#2}{#1}}
\newcommand{\eventupset}[2][]{\upset[\events]{#2}{#1}}
\newcommand{\fineventUpset}[2][]{\Upset[\finevents]{#2}{#1}}
\newcommand{\finfineventUpset}[2][]{\Upset[\finfinevents]{#2}{#1}}
\newcommand{\closure}[1][]{\smash{\cl_{#1}}}
\newcommand{\initialclosure}{\closure[\things]}

\newcommand{\filterclosure}{\closure[\latticefilters]}

\newcommand{\fineventfilterclosure}{\closure[\fineventfilters]}
\newcommand{\finfineventfilterclosure}{\closure[\finfineventfilters]}
\newcommand{\cohsdsclosure}{\closure[\toppedcohsdss]}
\newcommand{\fincohsdsclosure}{\closure[\toppedfincohsdss]}
\newcommand{\cohsdfsclosure}{\closure[\toppedcohsdfss]}
\newcommand{\filter}[1][]{\mathscr{F}_{#1}}
\newcommand{\filterbase}[1][]{\mathscr{B}_{#1}}
\newcommand{\primefilter}[1][]{\mathscr{G}_{#1}}
\newcommand{\filters}{\mathbb{F}}
\newcommand{\primefilters}{\smash{\overline{\mathbb{F}}_{\mathrm{p}}}}
\newcommand{\principalfilters}{\mathbb{P}}
\newcommand{\properfilters}{\smash{\overline{\mathbb{F}}}}
\newcommand{\properprincipalfilters}{\smash{\overline{\mathbb{P}}}}
\newcommand{\latticefilters}{\filters(\lattice)}
\newcommand{\latticeprimefilters}{\primefilters(\lattice)}
\newcommand{\properlatticefilters}{\properfilters(\lattice)}

\newcommand{\eventprincipalfilters}{\principalfilters(\events)}
\newcommand{\propereventprincipalfilters}{\properprincipalfilters(\events)}
\newcommand{\fineventfilters}{\filters(\finevents)}
\newcommand{\fineventprimefilters}{\primefilters(\finevents)}
\newcommand{\properfineventfilters}{\properfilters(\finevents)}
\newcommand{\finfineventfilters}{\filters(\finfinevents)}
\newcommand{\finfineventprimefilters}{\primefilters(\finfinevents)}
\newcommand{\properfinfineventfilters}{\properfilters(\finfinevents)}
\newcommand{\setofpropositions}[1][]{H_{#1}}

\newcommand{\lindenbaumproperfilters}{\properfilters(\lindenbaum)}

\newcommand{\true}{\mathrm{T}}
\newcommand{\thing}[1][]{t_{#1}}
\newcommand{\things}[1][]{\smash{T_{#1}}}
\newcommand{\beautifulthings}{\things[{+}]}
\newcommand{\uglythings}{\things[{-}]}
\newcommand{\setsofthings}{\powerset(\things)}
\newcommand{\finites}{\mathscr{Q}}
\newcommand{\finitesetsofthings}{\finites(\things)}
\newcommand{\setsofsetsofthings}{\powerset(\setsofthings)}
\newcommand{\setsoffinitesetsofthings}{\powerset(\finitesetsofthings)}
\newcommand{\cohsdtsymbol}{D}
\newcommand{\cohsdssymbol}{K}
\newcommand{\sot}[1][]{S_{#1}}
\newcommand{\finsot}[1][]{\smash{\hat{S}_{#1}}}
\newcommand{\cohsdt}[1][]{\cohsdtsymbol_{#1}}
\newcommand{\toppedcohsdts}{\mathbf{\cohsdtsymbol}}
\newcommand{\cohsdts}[1][]{\smash{\overline{\toppedcohsdts}_{#1}}}
\newcommand{\idealcohsdt}{\cohsdt[\true]}
\newcommand{\sosot}[1][]{W_{#1}}
\newcommand{\sofinsot}[1][]{\hat{W}_{#1}}
\newcommand{\sop}[1][]{V_{#1}}
\newcommand{\cohsds}[1][]{K_{#1}}
\newcommand{\cohsdfs}[1][]{\smash{\hat{K}}_{#1}}
\newcommand{\cohsdss}{\smash{\overline{\toppedcohsdss}}}
\newcommand{\cohsdfss}{\smash{\overline{\mathbf{F}}}_{\mathrm{fin}}}
\newcommand{\completefincohsdss}{\smash{\overline{\toppedcohsdss}}_{\mathrm{fin,c}}}
\newcommand{\completesdss}{\smash{\overline{\toppedcohsdss}}_{\mathrm{c}}}
\newcommand{\completecohsdfss}{\smash{\overline{\mathbf{F}}}_{\mathrm{fin,c}}}
\newcommand{\fincohsdss}{\smash{\overline{\toppedcohsdss}}_{\mathrm{fin}}}
\newcommand{\toppedcohsdss}{\mathbf{\cohsds}}
\newcommand{\toppedcohsdfss}{\smash{\mathbf{F}_{\mathrm{fin}}}}
\newcommand{\toppedfincohsdss}{\mathbf{\cohsds}_{\mathrm{fin}}}
\newcommand{\indexedselections}[2][\sosot]{\smash{\Phi_{#1}^{#2}}}
\newcommand{\selections}[1][\sosot]{\smash{\Phi_{#1}}}
\newcommand{\finselections}[1][\sofinsot]{\smash{\Phi_{#1}}}
\newcommand{\selection}[1][]{\smash{\sigma_{#1}}}
\newcommand{\finpart}[2][]{\fin\group[#1]{#2}}
\newcommand{\fintypart}[2][]{\finty\group[#1]{#2}}
\newcommand{\sdsify}[1]{\cohsdssymbol_{#1}}
\newcommand{\sdfsify}[1]{\smash{\hat{\cohsdssymbol}}_{#1}}
\newcommand{\sdtify}[1]{\cohsdtsymbol_{#1}}

\newcommand{\filterise}[1][\toppedcohsdts]{\smash{\varphi_{#1}}}
\newcommand{\desirify}[1][\toppedcohsdts]{\smash{\kappa_{#1}}}
\newcommand{\finfilterise}[1][\toppedcohsdts]{\smash{\varphi^{\mathrm{fin}}_{#1}}}
\newcommand{\findesirify}[1][\toppedcohsdts]{\smash{\kappa^{\mathrm{fin}}_{#1}}}
\newcommand{\finfinfilterise}[1][\toppedcohsdts]{\smash{\hat{\varphi}}^{\mathrm{fin}}_{#1}}
\newcommand{\finfindesirify}[1][\toppedcohsdts]{\smash{\hat{\kappa}}^{\mathrm{fin}}_{#1}}

\newcommand{\entails}{\vdash}
\newcommand{\toclass}[2][\equiv]{#2/_{#1}}
\newcommand{\thingclass}[1][]{\toclass{\thing[#1]}}
\newcommand{\lindenbaum}{\mathbb{L}}

\newcommand{\opt}[1][]{h_{#1}}
\newcommand{\altopt}[1][]{g_{#1}}
\newcommand{\opts}{\gbls}
\newcommand{\posopts}{\opts_{\optgt0}}
\newcommand{\negopts}{\opts_{\optlt0}}
\newcommand{\nonposopts}{\opts_{\optlte0}}

\newcommand{\finiteoptsets}{\finites(\opts)}
\newcommand{\optlt}[1][]{\mathrel{\prec_{#1}}}
\newcommand{\optlte}[1][]{\mathrel{\preceq_{#1}}}
\newcommand{\optgt}[1][]{\mathrel{\succ_{#1}}}

\newcommand{\optset}[1][]{A_{#1}}

\newcommand{\states}{\mathscr{X}}
\newcommand{\gbls}{\mathscr{G}}
\newcommand{\posgbls}{\gbls_{\gblgt0}}
\newcommand{\neggbls}{\gbls_{\gbllt0}}

\newcommand{\gbllt}{\optlt}
\newcommand{\gblgt}{\optgt}
\newcommand{\prev}{P}

\newcommand{\simplex}[1][\states]{\Sigma_{#1}}
\newcommand{\solp}{\mathscr{M}}
\newcommand{\rf}{R}
\newcommand{\cf}{C}

\newcommand{\naturals}{\mathbb{N}}
\newcommand{\naturalswithzero}{\mathbb{N}_0}
\newcommand{\reals}{\mathbb{R}}
\newcommand{\posreals}{\mathbb{R}_{>0}}
\newcommand{\nonnegreals}{\mathbb{R}_{\geq0}}

\newcommand{\cset}[3][]{\set[#1]{#2\colon#3}}
\newcommand{\then}{\Rightarrow}
\newcommand{\ifandonlyif}{\Leftrightarrow}
\newcommand{\bolleke}{\vcenter{\hbox{\scalebox{0.7}{\(\bullet\)}}}}
\newcommand{\powerset}{\mathscr{P}}
\newcommand{\finite}[1]{\smash{\hat{#1}}}

\newcommand{\events}{\mathbf{E}}
\newcommand{\finevents}{\mathbf{E}_{\mathrm{fin}}}
\newcommand{\finfinevents}{\smash{\hat{\mathbf{E}}}_{\mathrm{fin}}}
\newcommand{\event}[2][]{\mathscr{E}\group[#1]{#2}}
\newcommand{\basicevent}[2][]{\mathscr{D}\group[#1]{#2}}
\newcommand{\eventwithindex}[2][]{\mathscr{E}\group[#1]{\sosot[#2]}}
\newcommand{\fineventwithindex}[2][]{\mathscr{E}\group[#1]{\sofinsot[#2]}}

\newcommand{\SDT}{set of desirable things}

\newcommand{\SDS}{set of desirable sets of things}

\newcommand{\SDFS}{set of desirable finite sets of things}

\newtheorem{theorem}{Theorem}
\newtheorem{proposition}[theorem]{Proposition}
\newtheorem{lemma}[theorem]{Lemma}
\newtheorem{corollary}[theorem]{Corollary}
\theoremstyle{definition}
\newtheorem{definition}{Definition}
\theoremstyle{remark}
\newtheorem*{counterexample}{Counterexample}

\author{Gert de Cooman \and Arthur Van Camp \and Jasper De Bock}
\address{Ghent University, Foundations Lab, Technologiepark--Zwijnaarde 125, Ghent, Belgium}
\title{The logic behind desirable sets of things, and its filter representation}
\begin{document}
\begin{abstract}
We identify the (filter representation of the) logic behind the recent theory of coherent sets of desirable (sets of) things, which generalise coherent sets of desirable (sets of) gambles as well as coherent choice functions, and show that this identification allows us to establish various representation results for such coherent models in terms of simpler ones.
\end{abstract}
\keywords{desirability, desirable sets of things, conservative inference, propositional logic, filter, prime filter, principal filter}
\maketitle

\section{Introduction}\label{sec:introduction}
A very important aspect of the theory of \emph{imprecise probabilities} \cite{levi1980a,walley1991,augustin2013:itip,troffaes2013:lp} is that it allows for partial specification of probability models (such as probability measures) and, equally importantly, that it enables us to do \emph{conservative inference}.
To give an example, if we specify bounds on the probabilities of a number of events, then the theory is concerned with, amongst other things, inferring the implied bounds on the probabilities of other events.

Such conservative probabilistic inference can be represented quite intuitively and effectively by considering simple desirability statements \cite{walley2000,seidenfeld1999}: if some uncertain rewards---also called \emph{gambles}---are considered desirable to a subject, what does that imply about the desirability (or otherwise) of other gambles?
That a subject considers a given gamble to be desirable is then considered as a simple statement, very much like her asserting a proposition in a propositional logic context.
Inferring from a collection of such desirability statements which other gambles are desirable, is then effectively a matter of \emph{deductive inference} based on a number of so-called \emph{coherence} rules, very much like logical inference is based on the conjunction and modus ponens rules.
This observation has led to a theory of \emph{coherent sets of desirable gambles} \cite{quaeghebeur2012:itip,couso2011:desirable,debock2015:phdthesis,cooman2010,decooman2015:coherent:predictive:inference,miranda2022:nonlinear:desirability,walley2000}: sets of gambles that are deductively closed under the inference based on the coherence rules.

There's, however, a further complication to be dealt with.
Indeed, a desirability statement for a gamble is tantamount to a pairwise comparison, and more specifically a \emph{strict preference}, between this gamble and the zero gamble---the status quo.
It was recognised by Isaac Levi \cite{levi1980a} quite early on in the development of the theory that certain aspects of conservative probabilistic inference demand looking \emph{further than merely pairwise} preferences between gambles.
This has led to the introduction of choice functions, a tool from social choice theory \cite{sen1986:social:choice:theory}, into the field of imprecise probabilities \cite{seidenfeld2010,Rubin1987,Kadane2004,2017vancamp:phdthesis}.
In a number of recent papers \cite{cooman2021:archimedean:choice,debock2018,debock2020:axiomatic:archimedean,debock2019b}, we showed that working with the resulting so-called \emph{coherent choice functions} is mathematically equivalent to doing inferences with \emph{desirable sets of gambles}, rather than with desirable gambles, where a set of gambles is judged to be desirable as soon as at least one of its elements is: coherent choice functions can be seen as a special case of coherent---deductively closed---sets of desirable sets of gambles.
Choosing between gambles can also be usefully generalised to choosing between options that live in a linear space, as was done in Refs.~\cite{2017vancamp:phdthesis,cooman2021:archimedean:choice}.

In very recent work \cite{debock2023:things:arxiv,debock2023:things:isipta}, Jasper De Bock has taken the idea of moving from desirable \emph{gambles} to desirable \emph{sets of gambles} a significant step further, by recognising that it can be applied to any context involving conservative inference based on a closure operator.
In his abstract generalisation step, gambles are replaced by abstract objects, called \emph{things}, and it's assumed that some abstract property of things, called their \emph{desirability}, can be inferred from the desirability of other things through inference rules that are summarised by the action of some closure operator.
This leads to a theory of coherent---deductively closed---sets of desirable sets of things.

At about the same time, Catrin Campbell\textendash Moore \cite{campbellmoore2021:probability:filters:isipta2021} showed that statements about, and the inference behind, the desirability of gambles, and to some extent also the desirability of sets of gambles,\footnote{More specifically, the desirability of sets of gambles then needs to satisfy a so-called \emph{mixingness} requirement; see also the discussion in Section~\ref{sec:choice:functions}.} can be represented by filters of (sets of) probability measures.
This established that the conservative inference mechanism behind these desirability models can also be interpreted as that of propositional logic involving statements about some `ideal unknown' probability measure.\footnote{She has since \cite{campbellmoore2024:desirable:gamble:sets} extended this idea to representations by filters of sets of coherent sets of desirable gambles, along the lines of, but independently from, what we'll achieve for the more general coherent sets of things in Section~\ref{sec:representation:finitary}.}

In the present paper, we combine ideas from both these recent developments, by showing how the conservative inference mechanism behind coherent---deductively closed---sets of \emph{desirable sets of things} is related to that of a propositional logic involving statements about some `ideal unknown' coherent set of \emph{desirable things}.
We do this by showing that the collection of all coherent sets of desirable sets of things is an intersection structure that's order isomorphic to the set of all proper filters on a specific distributive lattice of events, where the events are appropriately chosen collections of coherent sets of desirable things.

Besides identifying the nature of the inference mechanism behind coherent sets of desirable sets of things (and in particular, coherent sets of desirable sets of gambles and therefore also coherent choice functions), our results also allow us to prove (or provide simple alternative proofs for) powerful, interesting as well as useful representation theorems for such models in terms of simpler, so-called conjunctive, models.

We'll freely use basic concepts and results from order theory, and we'll assume the reader to be familiar with most of them, so especially in the context of proofs, we'll mostly limit ourselves to pointers to the literature.
For a nice introduction, we refer to Davey and Priestley's book \cite{davey2002}.

Our argumentation is structured as follows.

In Section~\ref{sec:desirability}, we summarise the basic ideas behind coherent sets of desirable sets of things, and identify the order-theoretic underpinnings of the inference mechanism behind them.
We distinguish between the standard coherence notion, which connects the desirability of arbitrary sets of things, and the less restrictive \emph{finite} coherence notion, which essentially only focuses on the consequences of the desirability of finite sets of things.
We show that the coherent sets of desirable things can be embedded into the (finitely) coherent sets of desirable sets of things, in the form of \emph{conjunctive} models.

In order to provide a more direct link with our earlier work on desirable gamble and option sets \cite{cooman2021:archimedean:choice,debock2018,debock2020:axiomatic:archimedean,debock2019b,debock2023:things:arxiv}, we introduce and discuss yet another model, finitely coherent sets of desirable \emph{finite} sets of things, in Section~\ref{sec:desirability:sdfs}.

In Section~\ref{sec:towards:representation}, we explain how each desirability statement for a given set of things can be identified with a so-called \emph{event}: the specific subset of the collection of all coherent sets of desirable things it's compatible with.

In Section~\ref{sec:representation:lattices}, we identify the order-theoretic nature of the collection of all finitary events as a bounded distributive lattice, and for the collection of all events as a completely distributive complete lattice.

Section~\ref{sec:filters} is a very short primer on order-theoretic and set-theoretic filters, and the prime filter representation theorem.

After setting up all this background material, we show in Section~\ref{sec:representation:finitary} that the collection of all finitely coherent sets of desirable sets of things is order-isomorphic to the collection of all proper filters of finitary events, and then explain how this order isomorphism leads in a straightforward manner to a representation of finitely coherent sets of desirable sets of things as limits inferior of conjunctive ones, providing a simple alternative proof to, and extending, similar representation results by Catrin Campbell\textendash Moore \cite{campbellmoore2021:probability:filters:isipta2021}.

Similarly,  we show in Section~\ref{sec:representation} that the collection of all coherent sets of desirable sets of things is order-isomorphic to the collection of all proper principal filters of events, and then explain how this order isomorphism can be used to a prove a representation of coherent sets of desirable sets of things as intersections of conjunctive ones, leading to a simple alternative proof for related results by Jasper De Bock \cite{debock2023:things:arxiv}.

We pay extra attention to the finitely coherent models in Section~\ref{sec:finitary:models}, where we use the prime filter representation theorem to show that under some extra conditions, a so-called \emph{finitary} subclass of them can also be represented as intersections of conjunctive ones.

All this work allows us to show in Section~\ref{sec:representation:finitary:finite} that the finitely coherent sets of desirable \emph{finite} sets of things introduced in Section~\ref{sec:desirability:sdfs} are in a one-to-one-relationship with the \emph{finitary} and finitely coherent sets of desirable sets of things that Section~\ref{sec:representation:finitary:finite} is devoted to.
This connection allows us to prove representation results that generalise our earlier results on desirable gamble sets \cite{cooman2021:archimedean:choice,debock2019b,debock2018}, and also leads to a simple alternative proof for a similar result in Ref.~\cite{debock2023:things:arxiv}.

Sections~\ref{sec:propositional:logic} and~\ref{sec:choice:functions} contain discussions of two important special cases, where desirable things take the more concrete form of \emph{asserted propositions} and \emph{desirable gambles}, respectively.

We have moved most of the technical details and arguments that, taken by themselves, wouldn't add to the flow of the narrative, to Appendix~\ref{sec:proofs}, organised by the sections they appear in.
All claims in this paper either find a proof there if they are new, or are given a proper reference to the existing literature.
To help the reader find his way through the many notions and notations we need in this paper, Appendix~\ref{sec:symbols} provides of list of the most common ones, with a short hint at their meaning, and where they are introduced.

\section{Sets of desirable (sets of) things: an overview}\label{sec:desirability}
Let's begin by giving a brief overview of (a version of) Jasper De Bock's theory of desirable things \cite{debock2023:things:arxiv} that's sufficient for the purposes of this paper.
We use somewhat different notations, and our overview differs in a few technical details, but the ideas and the conclusions we draw from them are essentially the same.\footnote{To be more precise, and adopting already here some of the notation introduced further on, De Bock assumes that \(\initialclosure(\emptyset)=\emptyset\) and that the set~\(\beautifulthings\) is added explicitly to the theory, with the requirement that \(\initialclosure(\beautifulthings)\cap\uglythings=\emptyset\), instead of being inferred from the closure operator by letting~\(\beautifulthings\coloneqq\initialclosure(\emptyset)\) as we do here. He also doesn't make use of selection maps to formulate the coherence axioms for the various types of sets of desirable sets of things.}

\subsection{Desirable things}\label{sec:desirable:things}
We begin by considering a non-empty set~\(\things\) of things~\(\thing\) that may or may not have a certain property.
Having this property makes a thing \emph{desirable}.

You, our subject, may entertain ideas about which things are desirable, and You represent these ideas by providing a (not necessarily exhaustive) set of things that You find desirable.
We'll call such a subset \(\sot\subseteq\things\) a \emph{{\SDT}}, or SDT for short (plural: SDTs): a set with the property that You think \emph{each of its elements} desirable.
We denote by \(\setsofthings\) the set of all subsets~\(\sot\) of~\(\things\), or in other words, the collection of all candidate SDTs.

Such SDTs can be ordered by set inclusion.
We interpret \(\sot[1]\subseteq\sot[2]\) to mean that \(\sot[1]\) is \emph{less informative}, or \emph{more conservative}, than~\(\sot[2]\), simply because a subject with an SDT~\(\sot[1]\) finds fewer things desirable than a subject with SDT~\(\sot[2]\).

Our basic assumption is that, mathematically speaking, there are a number of rules that underlie the notion of desirability for things, and that the net effect of these rules can be captured by a closure operator and a set of forbidden things.
Let's explain.

We recall that a \emph{closure operator} on a non-empty set~\(G\) is a map~\(\cl\colon\powerset(G)\to\powerset(G)\) satisfying:
\begin{enumerate}[label={\upshape C}\({}_{\arabic*}\).,ref={\upshape C}\({}_{\arabic*}\),leftmargin=*]
\item\label{axiom:closure:super} \(A\subseteq\cl(A)\) for all~\(A\subseteq G\);
\item\label{axiom:closure:increasing} if \(A\subseteq B\) then \(\cl(A)\subseteq\cl(B)\) for all~\(A,B\subseteq G\);
\item\label{axiom:closure:idempotent} \(\cl(\cl(A))=\cl(A)\) for all~\(A\subseteq G\).
\end{enumerate}
Again, \(\powerset(G)\) denotes the set of all subsets of the set~\(G\).
A closure operator~\(\cl\) is called \emph{finitary}\footnote{Davey and Priestley \cite[Definition~7.12]{davey2002} use the term `algebraic', but `finitary' seems to be the more common name for this concept, so we'll stick with that.} if it's enough to know the closure of finite sets, in the following sense:
\begin{equation*}
\cl(A)=\bigcup\cset{\cl(F)}{F\in\powerset(G)\text{ and }F\Subset A}
\text{ for all~\(A\in\powerset(G)\)},
\end{equation*}
where we use the notation `\(\Subset\)' to mean `is a finite subset of', and agree to call the empty set~\(\emptyset\) finite.

First of all, as already suggested above, we'll assume that there's some inference mechanism that allows us to infer the desirability of a thing from the desirability of other things.
This inference mechanism is represented by a closure operator~\(\initialclosure\colon\setsofthings\to\setsofthings\), in the following sense:
\begin{enumerate}[label={\upshape D}\({}_{\arabic*}\).,ref={\upshape D}\({}_{\arabic*}\),leftmargin=*]
\item\label{axiom:desirable:things:closure} if all things in~\(\sot\) are desirable, then so are all things in~\(\initialclosure(\sot)\).
\end{enumerate}
We collect all sets of things that are \emph{closed} under the inference mechanism in the set~\(\toppedcohsdts\), so
\[
\toppedcohsdts
\coloneqq\cset{\sot\in\setsofthings}{\initialclosure(\sot)=\sot}.
\]
The following result is then a standard conclusion in order theory; see Ref.~\cite[Chapter~7]{davey2002} for the argumentation.

\begin{proposition}\label{prop:coherent:sets:of:desirable:things:consititute:a:complete:lattice}
The partially ordered set~\(\structure{\toppedcohsdts,\subseteq}\) is a complete lattice.
For any non-empty family~\(\sot[i]\), \(i\in I\) of elements of~\(\toppedcohsdts\), we have for its infimum and its supremum that, respectively, \(\inf_{i\in I}\sot[i]=\bigcap_{i\in I}\sot[i]\) and \(\sup_{i\in I}\sot[i]=\initialclosure\group{\bigcup_{i\in I}\sot[i]}\).
\end{proposition}
\noindent The \emph{bottom} (smallest element)~\(\cohsdtbottom\coloneqq\bigcap\toppedcohsdts\) of this complete lattice of closed sets will also be denoted by~\(\beautifulthings\); it is the set of those things that are \emph{always desirable}, regardless of what You may think.
The \emph{top} (largest element)~\(\cohsdttop\coloneqq\initialclosure\group{\bigcup\toppedcohsdts}\) is clearly the set of all things~\(\things\).

Secondly, we assume that there's a set of so-called \emph{forbidden} things~\(\uglythings\), which are \emph{never desirable}:
\begin{enumerate}[label={\upshape D}\({}_{\arabic*}\).,ref={\upshape D}\({}_{\arabic*}\),leftmargin=*,resume]
\item\label{axiom:desirable:things:forbidden} no thing in~\(\uglythings\) is desirable, so if all things in~\(\sot\) are desirable, then \(\sot\cap\uglythings=\emptyset\).
\end{enumerate}
Of course, because we assume that all things in an SDT are desirable, it can never intersect the set~\(\uglythings\), so this leaves us with the collection
\begin{equation}\label{eq:definition:cohsdts}
\cohsdts
\coloneqq\cset{\sot\in\setsofthings}{\initialclosure(\sot)=\sot\text{ and }\sot\cap\uglythings=\emptyset}
=\cset{\sot\in\toppedcohsdts}{\sot\cap\uglythings=\emptyset}
\end{equation}
of the closed SDTs that we'll call \emph{coherent}.\footnote{It's of course perfectly possible that \(\uglythings=\emptyset\), and in that case coherence and closedness coincide: \(\cohsdts=\toppedcohsdts\).}
We'll use the generic notation~\(\cohsdt\) for such coherent SDTs.
It's a standard result in order theory, and easy to check, that they constitute an \emph{intersection structure}: the intersection of any non-empty family \(\cohsdt[i]\), \(i\in I\neq\emptyset\) of them is still coherent:
\begin{equation}\label{eq:closed:under:intersections}
\group[\big]{\group{\forall i\in I}\cohsdt[i]\in\cohsdts}\then\bigcap_{i\in I}\cohsdt[i]\in\cohsdts.
\end{equation}

Clearly, an SDT~\(\sot\subseteq\things\) can be extended to a coherent one if and only if \(\initialclosure(\sot)\in\cohsdts\), or equivalently, if \(\initialclosure(\sot)\cap\uglythings=\emptyset\), and in that case we'll call this~\(\sot\) \emph{consistent}.

For any consistent SDT~\(\sot\), it's easy to see that
\begin{equation}\label{eq:cohsdtclosure:consistent:sodts}
\initialclosure(\sot)
=\bigcap\cset{\cohsdt\in\cohsdts}{\sot\subseteq\cohsdt},
\end{equation}
so \(\initialclosure(\sot)\) is the smallest, or most conservative, or least informative, coherent SDT that the consistent~\(\sot\) can be extended to.\footnote{As one reviewer remarked, there may be things that aren't forbidden but also never desirable: all things in the set~\(\things[o]\coloneqq\things\setminus\group{\uglythings\cup\bigcup\cohsdts}\). Clearly, since no element of~\(\things[o]\) can belong to a consistent or a coherent SDT, we can safely remove all of them from the set of things~\(\things\), or alternatively, add them to the set of forbidden things~\(\uglythings\), without affecting the essence of the inference mechanism.}

Incidentally, the set~\(\beautifulthings\coloneqq\initialclosure(\emptyset)=\bigcap\toppedcohsdts\) is the smallest closed SDT.
If \(\beautifulthings\) is coherent, or in other words if the empty set \(\emptyset\) is consistent, then \(\beautifulthings\) is the smallest, or most conservative, coherent SDT.
This will be the case if and only if
\begin{enumerate}[label={\upshape D}\({}_{\arabic*}\).,ref={\upshape D}\({}_{\arabic*}\),leftmargin=*,resume]
\item\label{axiom:desirable:things:background:consistency} \(\beautifulthings\cap\uglythings=\emptyset\), or equivalently, \(\cohsdts\neq\emptyset\).
\end{enumerate}
We'll from now on also always assume that this `sanitary' condition is verified.
As already mentioned, all things in~\(\beautifulthings\) are always implicitly desirable, regardless of any of the desirability statements You might make.
For this reason, we can also call \(\beautifulthings\) the \emph{vacuous} set of desirable things.

\subsection{Desirable sets of things}\label{sec:desirability:sds}
When You claim that a set of things~\(\sot\subseteq\things\) is a set of desirable things, this is tantamount to a \emph{conjunctive} statement: You state that ``all things in~\(\sot\) are desirable''.
In the formalism described above, there's no way to deal with \emph{disjunctive} statements of the type ``at least one of the things in~\(\sot\) is desirable''.
Let's now look for a way to also allow for dealing with such disjunctive statements.

We'll say that You consider a \emph{set of things}~\(\sot\) to be \emph{desirable} if You consider at least one thing in~\(\sot\) to be.
In other words, in a set of desirable things (an SDT), all things are desirable, whereas in a desirable set of things, at least one thing is.
As with the desirability of things, You can make many desirability statements for sets of things, and we then collect all of these in a \emph{{\SDS}}---or for short SDS, plural SDSes---\(\sosot\subseteq\setsofthings\).
So \(\sosot\) is an SDS for You if all sets of things~\(\sot\in\sosot\) are desirable to You, in the sense You deem each of them to contain at least one desirable thing.

Sets of desirable sets of things can be ordered by set inclusion too.
We take \(\sosot[1]\subseteq\sosot[2]\) to imply that \(\sosot[1]\) is \emph{less informative}, or \emph{more conservative}, than~\(\sosot[2]\), simply because a subject with SDS~\(\sosot[1]\) finds fewer sets of things desirable than a subject with SDS~\(\sosot[2]\).

The inference mechanism for the desirability of things also has its consequences for the desirability of sets of things, as we'll now make clear.
Consider any set of sets of things~\(\sosot\subseteq\setsofthings\), then we denote by~\(\selections\) the set of all so-called \emph{selection maps}
\begin{equation*}
\selection\colon\sosot\to\things\colon\sot\mapsto\selection(\sot)
\text{ such that }
\selection(\sot)\in\sot
\text{ for all~\(\sot\in\sosot\)}.
\end{equation*}
Each such selection map~\(\selection\in\selections\) selects a single thing~\(\selection(\sot)\) from each set of things~\(\sot\) in~\(\sosot\), and we use the notation
\begin{equation*}
\selection(\sosot)\coloneqq\cset{\selection(\sot)}{\sot\in\sosot}\in\setsofthings
\end{equation*}
for the corresponding set of all these selected things.

We now call an SDS~\(\cohsds\subseteq\setsofthings\) \emph{coherent} if it satisfies the following conditions:
\begin{enumerate}[label=\upshape K\({}_{\arabic*}\).,ref=\upshape K\({}_{\arabic*}\),leftmargin=*]
\item\label{axiom:desirable:sets:consistency} \(\emptyset\notin\cohsds\);
\item\label{axiom:desirable:sets:increasing} if \(\sot[1]\in\cohsds\) and \(\sot[1]\subseteq\sot[2]\) then \(\sot[2]\in\cohsds\), for all~\(\sot[1],\sot[2]\in\setsofthings\);
\item\label{axiom:desirable:sets:forbidden} if \(\sot\in\cohsds\) then \(\sot\setminus\uglythings\in\cohsds\), for all~\(\sot\in\setsofthings\);
\item\label{axiom:desirable:sets:background} \(\set{\thing[{+}]}\in\cohsds\) for all~\(\thing[{+}]\in\beautifulthings\);
\item\label{axiom:desirable:sets:production} if \(\thing[\selection]\in\initialclosure(\selection(\sosot))\) for all~\(\selection\in\selections\), then \(\cset{\thing[\selection]}{\selection\in\selections}\in\cohsds\), for all~\(\emptyset\neq\sosot\subseteq\cohsds\).
\end{enumerate}
The first condition~\ref{axiom:desirable:sets:consistency} takes into account that the empty set of things can't be desirable, as it contains no desirable thing.
The second condition~\ref{axiom:desirable:sets:increasing} reminds us that if a set of things contains a desirable thing, then of course so do all its supersets.
The third condition~\ref{axiom:desirable:sets:forbidden} reflects that things in~\(\uglythings\) can never be desirable, by \ref{axiom:desirable:things:forbidden}, and can therefore safely be removed from any set of things without affecting the latter's desirability.
And, to conclude, we'll see further on that the last two conditions~\ref{axiom:desirable:sets:background} and~\ref{axiom:desirable:sets:production} do a very fine job of lifting the effects of inferential closure from the desirability of things to the desirability of sets of things.
They can be justified as follows.
For~\ref{axiom:desirable:sets:background}, recall from the discussion above that any element of~\(\beautifulthings\) is always implicitly desirable, and so therefore will be any set that contains it.
For~\ref{axiom:desirable:sets:production}, recall that \(\emptyset\neq\sosot\subseteq\cohsds\) means that each set of things~\(\sot\in\sosot\) contains at least one desirable thing, and therefore there must be some selection map~\(\selection[o]\in\selections[\sosot]\) such that \(\selection[o](\sot)\) is a desirable thing for all~\(\sot\in\sosot\).
This implies that all things in~\(\selection[o](\sosot)\) are desirable, and therefore so are all things in~\(\initialclosure(\selection[o](\sosot))\), by~\ref{axiom:desirable:things:closure}.
Whatever \(\thing[{\selection[o]}]\) we chose in~\(\initialclosure(\selection[o](\sosot))\) will therefore be desirable, which guarantees that the set of things~\(\cset{\thing[\selection]}{\selection\in\selections[\sosot]}\) must also be desirable, because it contains the necessarily desirable thing~\(\thing[{\selection[o]}]\).

Interestingly, Axioms~\ref{axiom:desirable:sets:background} and~\ref{axiom:desirable:sets:production} can be replaced by a single axiom, obtained from~\ref{axiom:desirable:sets:production} by also allowing \(\sosot\subseteq\cohsds\) to be \emph{empty}:
\begin{enumerate}[label={\upshape K}\({}_{4,5}\).,ref={\upshape K}\({}_{4,5}\),leftmargin=*]
\item\label{axiom:desirable:sets:production:with:background} if \(\thing[\selection]\in\initialclosure(\selection(\sosot))\) for all~\(\selection\in\selections\), then \(\cset{\thing[\selection]}{\selection\in\selections}\in\cohsds\), for all~\(\sosot\subseteq\cohsds\).
\end{enumerate}

We denote the set of all coherent SDSes by~\(\cohsdss\), and we let~\(\toppedcohsdss\coloneqq\cohsdss\cup\set{\setsofthings}\).
Observe that \(\setsofthings\) is never coherent, by~\ref{axiom:desirable:sets:consistency}.
Since each of the axioms~\ref{axiom:desirable:sets:consistency}--\ref{axiom:desirable:sets:production} is preserved under taking arbitrary non-empty intersections, the set~\(\cohsdss\) of all coherent SDSes constitutes an intersection structure: the intersection of any non-empty family of coherent SDSes is still coherent, or in other and more formal words, for any non-empty family~\(\cohsds[i]\), \(i\in I\neq\emptyset\) of elements of~\(\cohsdss\), we see that still \(\bigcap_{i\in I}\cohsds[i]\in\cohsdss\).\footnote{Similarly, \(\toppedcohsdss=\cohsdss\cup\set{\setsofthings}\) is a \emph{topped} intersection structure: closed under arbitrary (also empty) intersections \cite[Chapter~7]{davey2002}.}
As explained in Ref.~\cite[Chapter~7]{davey2002}, this allows us to capture the inferential aspects of desirability at this level using the closure operator~\(\cohsdsclosure\colon\setsofsetsofthings\to\toppedcohsdss\) associated with the collection~\(\toppedcohsdss\) of \emph{closed} SDSes, defined by
\begin{equation*}
\cohsdsclosure(\sosot)
\coloneqq\bigcap\cset{\cohsds\in\toppedcohsdss}{\sosot\subseteq\cohsds}
=\bigcap\cset{\cohsds\in\cohsdss}{\sosot\subseteq\cohsds}
\text{ for all~\(\sosot\subseteq\setsofthings\)}.
\end{equation*}
If we call an SDS~\(\sosot\) \emph{consistent} if it can be extended to some coherent SDS, or equivalently, if \(\cohsdsclosure(\sosot)\neq\setsofthings\), then we find that \(\cohsdsclosure(\sosot)\) is the smallest, or most conservative, coherent SDS that includes~\(\sosot\), for any consistent~\(\sosot\).
Of course,
\begin{equation*}
\sosot=\cohsdsclosure(\sosot)\ifandonlyif\sosot\in\toppedcohsdss,
\text{ for all~\(\sosot\subseteq\setsofthings\)},
\end{equation*}
and therefore also \(\toppedcohsdss=\cohsdsclosure(\setsofsetsofthings)=\cset{\cohsdsclosure(\sosot)}{\sosot\subseteq\setsofthings}\).
The following result is then again a standard conclusion in order theory; see Ref.~\cite[Chapter~7]{davey2002} for the argumentation.

\begin{proposition}\label{prop:coherent:sds:constitute:a:complete:lattice}
The partially ordered set~\(\structure{\toppedcohsdss,\subseteq}\) is a complete lattice.
For any non-empty family~\(\sosot[i]\), \(i\in I\) of elements of~\(\toppedcohsdss\), we have for its infimum and its supremum that, respectively, \(\inf_{i\in I}\sosot[i]=\bigcap_{i\in I}\sosot[i]\) and \(\sup_{i\in I}\sosot[i]=\cohsdsclosure\group{\bigcup_{i\in I}\sosot[i]}\).
\end{proposition}
\noindent The top~\(\cohsdstop=\bigcup\toppedcohsdss\) of this complete lattice is the largest closed SDS~\(\setsofthings\); its bottom~\(\cohsdsbottom=\cohsdsclosure(\emptyset)=\bigcap\toppedcohsdss=\bigcap\cohsdss\) is the smallest coherent SDS \(\cohsdsbottom\), which is also easy to identify.

\begin{proposition}\label{prop:the:smallest:cohsds}
\(\cohsdsbottom=\bigcap\cohsdss=\cset{\sot\in\setsofthings}{\sot\cap\beautifulthings\neq\emptyset}\).
\end{proposition}

\subsection{Desirable sets of things: the finitary case}\label{sec:desirability:sds:finitary}
We call a subset~\(\cohsds\) of~\(\setsofthings\) a \emph{finitely} coherent SDS if it satisfies conditions~\ref{axiom:desirable:sets:consistency}--\ref{axiom:desirable:sets:background}, together with the following finitary version of~\ref{axiom:desirable:sets:production}:
\begin{enumerate}[label={\upshape K}\({}_{5}^{\mathrm{fin}}\).,ref={\upshape K}\({}_{5}^{\mathrm{fin}}\),leftmargin=*]
\item\label{axiom:desirable:sets:finitary:production} if \(\thing[\selection]\in\initialclosure(\selection(\sosot))\) for all~\(\selection\in\selections\), then \(\cset{\thing[\selection]}{\selection\in\selections}\in\cohsds\), for all~\(\emptyset\neq\sosot\Subset\cohsds\);
\end{enumerate}
We can replace \ref{axiom:desirable:sets:background} and \ref{axiom:desirable:sets:finitary:production} by a single axiom, which is the finitary counterpart of~\ref{axiom:desirable:sets:production:with:background}:
\begin{enumerate}[label={\upshape K}\({}_{4,5}^{\mathrm{fin}}\).,ref={\upshape K}\({}_{4,5}^{\mathrm{fin}}\),leftmargin=*]
\item\label{axiom:desirable:sets:finitary:production:with:background} if \(\thing[\selection]\in\initialclosure(\selection(\sosot))\) for all~\(\selection\in\selections\), then \(\cset{\thing[\selection]}{\selection\in\selections}\in\cohsds\), for all~\(\sosot\Subset\cohsds\).
\end{enumerate}
We denote by~\(\fincohsdss\) the set of all finitely coherent SDSes, and we let \(\toppedfincohsdss\coloneqq\fincohsdss\cup\set{\setsofthings}\).

For this finitary version, the discussion, definitions and ensuing results about the intersection structure \(\fincohsdss\), the complete lattice~\(\structure{\toppedfincohsdss,\subseteq}\) with bottom~\(\cset{\sot\in\setsofthings}{\sot\cap\beautifulthings\neq\emptyset}\) and top~\(\setsofthings\), and the associated closure operator~\(\fincohsdsclosure\) (counterparts to Propositions~\ref{prop:coherent:sds:constitute:a:complete:lattice} and~\ref{prop:the:smallest:cohsds}) are completely similar, and we'll refrain from repeating them here.
Observe nevertheless that \(\cohsdss\subseteq\fincohsdss\) and therefore also \(\toppedcohsdss\subseteq\toppedfincohsdss\): since \ref{axiom:desirable:sets:production} clearly implies~\ref{axiom:desirable:sets:finitary:production}, any coherent~\(\cohsds\) is also finitely coherent, so finite coherence is the weaker requirement.
As a consequence, we also find that \(\fincohsdsclosure(\sosot)\subseteq\cohsdsclosure(\sosot)\) for all~\(\sosot\subseteq\setsofthings\).

\subsection{Conjunctive models}\label{sec:desirability:conjunctive:models}
Let's now show that it's possible to order embed the structure~\(\structure{\cohsdts,\subseteq}\) into the structure~\(\structure{\cohsdss,\subseteq}\), and therefore also into the structure~\(\structure{\fincohsdss,\subseteq}\), in a straightforward and natural manner.

If we consider any set of things~\(\sot\) that's an element of the coherent SDS~\(\cohsds\), then we know from the coherence condition~\ref{axiom:desirable:sets:increasing} that all its supersets are also in~\(\cohsds\).
But, of course, not all of its subsets need to be, as is made clear by the coherence condition~\ref{axiom:desirable:sets:consistency}.

This observation brings us to the following idea.
Consider any SDS~\(\sosot\)---not necessarily coherent---and any element \(\sot\in\sosot\).
If there's some finite subset~\(\finsot\) of~\(\sot\) such that \(\finsot\in\sosot\), then we'll call~\(\sot\) \emph{finitary} (in~\(\sosot\)).
If, moreover, all the elements~\(\sot\) of the SDS~\(\sosot\) are finitary, then we'll call~\(\sosot\) \emph{finitary} as well; so any desirable set in a finitary~\(\sosot\) has a desirable finite subset.
The \emph{(finitely) coherent} finitary SDSes will be studied in much more detail in Section~\ref{sec:finitary:models}.
They are special because they are completely determined by their finite elements.

For the present discussion, however, we restrict our attention to an important special case of such finitary SDSes, where each desirable set has a desirable \emph{singleton} subset:

\begin{definition}[Conjunctivity]
We call an SDS~\(\sosot\subseteq\setsofthings\) \emph{conjunctive}\footnote{In earlier papers \cite{debock2018,debock2019b,cooman2021:archimedean:choice} about the desirability of gambles, and the underlying preference models, we used the term `binary' instead of `conjunctive', because the corresponding preference models turn out to be binary relations. But it seems a bit counter-intuitive to call sets based on singletons `binary'. We now use the term `conjunctive' because, as we'll see further on in Proposition~\ref{prop:conjunctive:and:coherent:sds} and Theorem~\ref{thm:sdt:to:sds:order:embedding}, these models are essentially representations of a conjunction of desirability statements for things.} if for all~\(\sot\in\sosot\) there is some~\(\thing\in\sot\) such that \(\set{\thing}\in\sosot\).
\end{definition}
\noindent In the remainder of this section, we'll spend some effort on identifying the conjunctive \emph{coherent} SDSes.

We begin by introducing ways to turn an SDT into an SDS, and vice versa.
Consider, to this end, any~\(\sot[o]\subseteq\things\) and any~\(\sosot[o]\subseteq\setsofthings\), and let
\begin{equation}\label{eq:sos:to:sot:and:sot:to:sos}
\sdtify{\sosot[o]}\coloneqq\cset{\thing\in\things}{\set{\thing}\in\sosot[o]}\subseteq\things
\text{ and }
\sdsify{\sot[o]}\coloneqq\cset{\sot\subseteq\things}{\sot\cap\sot[o]\neq\emptyset}\subseteq\setsofthings.
\end{equation}
So, if \(\sosot[o]\) is Your assessment of desirable sets of things, then \(\sdtify{\sosot[o]}\) is the set of things that You hold desirable, according to that assessment.
And similarly, if \(\sot[o]\) is Your assessment of desirable things, then \(\sdsify{\sot[o]}\) is the set of sets of things that are desirable to You, according to that assessment.
This leads to the introduction of two maps,
\[
\sdsify{\bolleke}\colon\setsofthings\to\setsofsetsofthings\colon\sot\mapsto\sdsify{\sot}
\text{ and }
\sdtify{\bolleke}\colon\setsofsetsofthings\to\setsofthings\colon\sosot\mapsto\sdtify{\sosot},
\]
whose properties we now explore in the next three basic propositions.

The following conclusion is fairly immediate.
It states that \(\sdtify{\bolleke}\) is an order homomorphism (or order-preserving map) and that \(\sdsify{\bolleke}\) is an order embedding.

\begin{proposition}\label{prop:sds:to:sdt:and:back:ordering}
Consider any~\(\sot,\sot[1],\sot[2]\in\setsofthings\) and any~\(\sosot[1],\sosot[2]\subseteq\setsofthings\).
Then \(\sdsify{\sot}\) is conjunctive, and \(\sdtify{\sdsify{\sot}}=\sot\).
Moreover,
\begin{enumerate}[label=\upshape(\roman*),leftmargin=*]
\item\label{it:sds:to:sdt:ordering} if \(\sosot[1]\subseteq\sosot[2]\) then \(\sdtify{\sosot[1]}\subseteq\sdtify{\sosot[2]}\);
\item\label{it:sdt:to:sds:ordering} \(\sot[1]\subseteq\sot[2]\) if and only if \(\sdsify{\sot[1]}\subseteq\sdsify{\sot[2]}\).
\end{enumerate}
\end{proposition}
\noindent In the following two propositions we find out whether the maps~\(\sdtify{\bolleke}\) and~\(\sdsify{\bolleke}\) preserve (finite) coherence.

\begin{proposition}\label{prop:cohsds:to:cohsdt}
Consider any SDS~\(\cohsds\subseteq\setsofthings\).
If \(\cohsds\) is (finitely) coherent, then \(\sdsify{\sdtify{\cohsds}}\subseteq\cohsds\).
Moreover, the following statements hold:
\begin{enumerate}[label=\upshape(\roman*),leftmargin=*]
\item\label{it:cohsds:to:cohsdt:coherent} if \(\cohsds\) is coherent,  then \(\sdtify{\cohsds}\) is coherent;
\item\label{it:cohsds:to:cohsdt:finitely:coherent} if \(\cohsds\) is finitely coherent and the closure operator~\(\initialclosure\) is finitary, then \(\sdtify{\cohsds}\) is coherent.
\end{enumerate}
\end{proposition}

\begin{proposition}\label{prop:cohsdt:to:cohsds}
Consider any set of things~\(\cohsdt\subseteq\things\), then \(\sdtify{\sdsify{\cohsdt}}=\cohsdt\).
Moreover, consider the following statements:
\begin{enumerate}[label=\upshape(\roman*),leftmargin=*]
\item\label{it:cohsdt:to:cohsds:coherent:sdt} \(\cohsdt\) is a coherent SDT;
\item\label{it:cohsdt:to:cohsds:coherent:sds} \(\sdsify{\cohsdt}\) is a coherent SDS;
\item\label{it:cohsdt:to:cohsds:finitely:coherent:sds} \(\sdsify{\cohsdt}\) is a finitely coherent SDS.
\end{enumerate}
Then \ref{it:cohsdt:to:cohsds:coherent:sdt}\(\ifandonlyif\)\ref{it:cohsdt:to:cohsds:coherent:sds} and \ref{it:cohsdt:to:cohsds:coherent:sds}\(\then\)\ref{it:cohsdt:to:cohsds:finitely:coherent:sds}, so \ref{it:cohsdt:to:cohsds:coherent:sdt}\(\then\)\ref{it:cohsdt:to:cohsds:finitely:coherent:sds}.
Moreover, if the closure operator~\(\initialclosure\) is finitary, then also \ref{it:cohsdt:to:cohsds:coherent:sdt}\(\ifandonlyif\)\ref{it:cohsdt:to:cohsds:finitely:coherent:sds}.
\end{proposition}
\noindent So, if we start out with a coherent SDS~\(\cohsds\), then the coherent SDT~\(\sdtify{\cohsds}\) and the corresponding coherent and conjunctive SDS~\(\sdsify{\sdtify{\cohsds}}\) are conservative approximations of~\(\cohsds\): going from a model~\(\cohsds\) to its \emph{conjunctive part}~\(\sdsify{\sdtify{\cohsds}}=\cset{\sot\in\cohsds}{(\exists\thing\in\sot)\set{\thing}\in\cohsds}\) typically results in a loss of information.
Similar results hold for finitely coherent~SDSes~\(\cohsds\), provided that the closure operator~\(\initialclosure\) is finitary, as then its conjunctive part~\(\sdsify{\sdtify{\cohsds}}\) will be coherent too.
Interestingly, going from a coherent SDT~\(\cohsdt\) to the corresponding coherent and conjunctive SDS~\(\sdsify{\cohsdt}\) does \emph{not} result in a loss of information, as we'll soon see below in Theorem~\ref{thm:sdt:to:sds:order:embedding}.

Based on our findings in Propositions~\ref{prop:sds:to:sdt:and:back:ordering}--\ref{prop:cohsdt:to:cohsds}, we're now finally in a position to find out what the conjunctive and (finitely) coherent SDSes look like.

\begin{proposition}[Conjunctivity]\label{prop:conjunctive:and:coherent:sds}
The following statements hold:
\begin{enumerate}[label=\upshape(\roman*),leftmargin=*]
\item\label{it:conjunctive:and:coherent:sds} An SDS~\(\cohsds\subseteq\setsofthings\) is coherent and conjunctive if and only if there's some coherent SDT~\(\cohsdt\in\cohsdts\) such that \(\cohsds=\sdsify{\cohsdt}\).
\item\label{it:conjunctive:and:finitely:coherent:sds} When the closure operator~\(\initialclosure\) is finitary, then an SDS~\(\cohsds\subseteq\setsofthings\) is finitely coherent and conjunctive if and only if there's some coherent SDT~\(\cohsdt\in\cohsdts\) such that \(\cohsds=\sdsify{\cohsdt}\).
\end{enumerate}
In both these cases then necessarily \(\cohsdt=\sdtify{\cohsds}\).
\end{proposition}

Combining the results of Propositions~\ref{prop:sds:to:sdt:and:back:ordering}\ref{it:sdt:to:sds:ordering}, \ref{prop:cohsdt:to:cohsds} and~\ref{prop:conjunctive:and:coherent:sds} leads at once to a formal statement of what we alluded to in the introduction to this subsection.

\begin{theorem}\label{thm:sdt:to:sds:order:embedding}
The map~\(\sdsify{\bolleke}\) is an order embedding of the (intersection) structure~\(\structure{\cohsdts,\subseteq}\) into the (intersection) structure~\(\structure{\cohsdss,\subseteq}\).
Its image is the set of all conjunctive coherent SDSes, and on this set its inverse is the map~\(\sdtify{\bolleke}\).
Moreover, if the closure operator~\(\initialclosure\) is finitary, then the map~\(\sdsify{\bolleke}\) is also an order embedding of the (intersection) structure~\(\structure{\cohsdts,\subseteq}\) into the (intersection) structure~\(\structure{\fincohsdss,\subseteq}\).
Its image is the set of all conjunctive finitely coherent SDSes, and on this set its inverse is the map~\(\sdtify{\bolleke}\).
\end{theorem}

In all the results following Proposition~\ref{prop:cohsds:to:cohsdt}, the finitary character of the closure operator~\(\initialclosure\) appears as a \emph{sufficient} condition to, loosely speaking, bijectively connect the finitely coherent SDSes to the coherent SDTs.
Let's give a simple counterexample to show that this one-to-one correspondence may break down for infinitary~\(\initialclosure\).
It was suggested to us by Kevin Blackwell, and is loosely based on Satan's Apple, a paradox for dominance in infinitary decisions due to Arntzenius, Elga, and Hawthorn \cite{arntzenius2004:satans:apple}.

\begin{counterexample}[Satan's Apple]
We consider as a set of things~\(\things\) the set~\(\naturalswithzero\) of all non-negative integers (with zero), and use the following closure operator
\[
\initialclosure(\sot)
\coloneqq
\begin{cases}
\set{1}&\text{if \(\sot=\emptyset\)}\\
\set{1,\dots,\max\sot}&\text{if \(\sot\) is a finite and non-empty subset of~\(\naturals\)},\\
\naturalswithzero&\text{otherwise}
\end{cases}
\quad\text{for all~\(\sot\subseteq\things\)},
\]
where~\(\naturals\) denotes the set of all natural numbers (without zero) and \(\max\sot\) is the maximum of the finite set of natural numbers~\(\sot\).
To see that this closure operator is \emph{not finitary}, check for instance that
\[
\initialclosure(\naturals)
=\naturalswithzero
\supset\naturals
=\bigcup\cset[\big]{\set{1,\dots,\max\sot}}{\emptyset\neq\sot\Subset\naturals}
=\bigcup\cset[\big]{\initialclosure(\sot)}{\sot\Subset\naturals}
.
\]
If we let \(\uglythings\coloneqq\set{0}\), then an SDT~\(\sot\subseteq\things\) is consistent if and only if it's finite and doesn't contain the forbidden thing~\(0\), and coherent if and only if it takes the form~\(\set{1,\dots,n}\) for some~\(n\in\naturals\).
Observe that, also, \(\beautifulthings=\set{1}\).

Let's now consider the SDS~\(\cohsds\coloneqq\powerset(\naturalswithzero)\setminus\set{\emptyset,\set{0}}\).
It is clearly conjunctive, and equal to its conjunctive part~\(\sdsify{\sdtify{\cohsds}}\).
The SDT~\(\sdtify{\cohsds}=\naturals\) is clearly \emph{not coherent}, so we can infer from the equivalence results in Proposition~\ref{prop:cohsdt:to:cohsds} that \(\cohsds=\sdsify{\sdtify{\cohsds}}\) isn't coherent either.

Since the closure operator~\(\initialclosure\) isn't finitary, we can't use Proposition~\ref{prop:cohsdt:to:cohsds} to also prove the \emph{finite incoherence} of \(\cohsds\).
On the contrary, we'll now prove that~\(\cohsds\) \emph{is} in fact \emph{finitely coherent}.

Since~\(\cohsds\) clearly satisfies~\ref{axiom:desirable:sets:consistency}--\ref{axiom:desirable:sets:background}, we concentrate on showing that it also satisfies~\ref{axiom:desirable:sets:finitary:production}.
Consider, to this effect, any non-empty~\(\sosot\Subset\cohsds\), so \(\sot\notin\set{\emptyset,\set{0}}\) for all~\(\sot\in\sosot\).
This already implies that there is some~\(\selection[o]\in\selections\) such that \(\selection[o](\sot)\neq0\) for all~\(\sot\in\sosot\).
Now, consider for any~\(\selection\in\selections\) any choice of the \(\thing[\selection]\in\initialclosure(\selection(\sosot))\), and assume towards contradiction that \(\cset{\thing[\selection]}{\selection\in\selections}\in\set{\emptyset,\set{0}}\).
Since \(\selections\) is non-empty because \(\sosot\) is, it must then be that \(\cset{\thing[\selection]}{\selection\in\selections}=\set{0}\), so it must follow that \(0\in\initialclosure(\selection(\sosot))\), implying that \(\selection(\sosot)\) is an infinite subset of~\(\naturals\) or contains~\(0\), for all~\(\selection\in\selections\).
That \(\selection(\sosot)\) should be infinite, is impossible for finite~\(\sosot\), so we find that \(0\in\selection(\sosot)=\cset{\selection(\sot)}{\sot\in\sosot}\) for all~\(\selection\in\selections\), a contradiction.

That the SDS~\(\cohsds\) is finitely coherent (but not coherent), whereas the SDT~\(\sdtify{\cohsds}\) isn't coherent, shows that the equivalence between the coherence of~\(\sdtify{\cohsds}\) and the finite coherence of~\(\cohsds\) in Proposition~\ref{prop:cohsdt:to:cohsds} in general can't be guaranteed to hold unless the closure operator~\(\initialclosure\) is finitary.
It also shows that, in the context of Proposition~\ref{prop:conjunctive:and:coherent:sds}, the finite coherence of a conjunctive SDS~\(\cohsds\) needn't guarantee that the SDT~\(\sdtify{\cohsds}\), and therefore also the conjunctive SDS~\(\cohsds=\sdsify{\sdtify{\cohsds}}\), are coherent, unless the closure operator~\(\initialclosure\) is finitary.
We'll come back to these issues in Section~\ref{sec:finitary:models}, where we discuss models that are merely finitary.
\hfill\(\triangleleft\)
\end{counterexample}

\section{Sets of desirable finite sets of things}\label{sec:desirability:sdfs}
In order to connect this discussion more directly with our recent work on coherent choice in earlier papers \cite{cooman2021:archimedean:choice,debock2018,debock2020:axiomatic:archimedean,debock2019b,debock2023:things:arxiv}, we'll also consider the case where You are only allowed to express the desirability of \emph{finite} sets of things.

We'll denote by \(\finitesetsofthings\) the set of all finite sets of things:
\begin{equation*}
\finitesetsofthings
\coloneqq\cset{\finsot\in\setsofthings}{\finsot\Subset\things},
\end{equation*}
and follow the convention that the empty set is finite, so \(\emptyset\in\finitesetsofthings\).

As before, we'll say that You consider a \emph{finite set of things}~\(\finsot\in\finitesetsofthings\) to be \emph{desirable} if You consider at least one thing in~\(\finsot\) to be.
We collect all Your desirability statements for finite sets of things in a \emph{{\SDFS}}---or for short SDFS, plural SDFSes---\(\sofinsot\subseteq\finitesetsofthings\).\footnote{We'll generally use a \(\hat{\bolleke}\) to indicate that the set `\(\bolleke\)' is finite at the level of things, or contains only finite sets at the level of sets of things.}

SDFSes can be ordered by set inclusion too, and, here too, we take \(\sofinsot[1]\subseteq\sofinsot[2]\) to imply that \(\sofinsot[1]\) is \emph{less informative}, or \emph{more conservative}, than~\(\sofinsot[2]\).

We call an SDFS~\(\cohsdfs\subseteq\finitesetsofthings\) \emph{finitely coherent} if it satisfies the following conditions:\footnote{We could also introduce an infinitary notion of coherence for SDFDes (see also Ref.~\cite{debock2023:things:arxiv}), but we'll refrain from doing so here, in order not to unduly complicate things: we haven't come across any such notion in the literature, apart from Ref.~\cite{debock2023:things:arxiv}.}
\begin{enumerate}[label=\upshape F\({}_{\arabic*}\).,ref=\upshape F\({}_{\arabic*}\),leftmargin=*]
\item\label{axiom:desirable:finite:sets:consistency} \(\emptyset\notin\cohsdfs\);
\item\label{axiom:desirable:finite:sets:increasing} if \(\finsot[1]\in\cohsdfs\) and \(\finsot[1]\Subset\finsot[2]\) then \(\finsot[2]\in\cohsdfs\), for all~\(\finsot[1],\finsot[2]\in\finitesetsofthings\);
\item\label{axiom:desirable:finite:sets:forbidden} if \(\finsot\in\cohsdfs\) then \(\finsot\setminus\uglythings\in\cohsdfs\), for all~\(\finsot\in\finitesetsofthings\);
\item\label{axiom:desirable:finite:sets:background} \(\set{\thing[{+}]}\in\cohsdfs\) for all~\(\thing[{+}]\in\beautifulthings\);
\item\label{axiom:desirable:finite:sets:production} if \(\thing[\selection]\in\initialclosure(\selection(\sofinsot))\) for all~\(\selection\in\finselections\), then \(\cset{\thing[\selection]}{\selection\in\finselections}\in\cohsdfs\), for all~\(\emptyset\neq\sofinsot\Subset\cohsdfs\).
\end{enumerate}
As before, Axioms~\ref{axiom:desirable:finite:sets:background} and~\ref{axiom:desirable:finite:sets:production} can be replaced by a single axiom, obtained from~\ref{axiom:desirable:finite:sets:production} by also allowing \(\sofinsot\Subset\cohsdfs\) to be empty:
\begin{enumerate}[label={\upshape F}\({}_{4,5}\).,ref={\upshape F}\({}_{4,5}\),leftmargin=*]
\item\label{axiom:desirable:finite:sets:production:with:background} if \(\thing[\selection]\in\initialclosure(\selection(\sofinsot))\) for all~\(\selection\in\finselections\), then \(\cset{\thing[\selection]}{\selection\in\finselections}\in\cohsdfs\), for all~\(\sofinsot\subseteq\cohsdfs\).
\end{enumerate}

We denote the set of all finitely coherent SDFSes by~\(\cohsdfss\), and let~\(\toppedcohsdfss\coloneqq\cohsdfss\cup\set{\finitesetsofthings}\).
Observe that \(\finitesetsofthings\) is never finitely coherent, by~\ref{axiom:desirable:finite:sets:consistency}.
Since each of the axioms~\ref{axiom:desirable:finite:sets:consistency}--\ref{axiom:desirable:finite:sets:production} is preserved under taking arbitrary non-empty intersections, the set~\(\cohsdfss\) of all finitely coherent SDFSes constitutes an intersection structure: the intersection of any non-empty family of finitely coherent SDFSes is still coherent, or in other and more formal words, for any non-empty family~\(\cohsdfs[i]\), \(i\in I\neq\emptyset\) of elements of~\(\cohsdfss\), we see that still \(\bigcap_{i\in I}\cohsdfs[i]\in\cohsdfss\).\footnote{Similarly, \(\toppedcohsdfss=\cohsdfss\cup
\set{\finitesetsofthings}\) is a \emph{topped} intersection structure: closed under arbitrary (also empty) intersections \cite[Chapter~7]{davey2002}.}
Again, as explained in Ref.~\cite[Chapter~7]{davey2002}, this allows us to capture the inferential aspects of desirability at this level using the closure operator~\(\cohsdfsclosure\colon\setsoffinitesetsofthings\to\toppedcohsdfss\) associated with the collection~\(\toppedcohsdfss\) of \emph{closed} SDFSes, defined by
\begin{equation*}
\cohsdfsclosure(\sofinsot)
\coloneqq\bigcap\cset{\cohsdfs\in\toppedcohsdfss}{\sofinsot\subseteq\cohsdfs}
=\bigcap\cset{\cohsdfs\in\cohsdfss}{\sofinsot\subseteq\cohsdfs}
\text{ for all~\(\sofinsot\subseteq\finitesetsofthings\)}.
\end{equation*}
If we call an SDFS~\(\sofinsot\) \emph{finitely consistent} if it can be extended to some finitely coherent SDFS, or equivalently, if \(\cohsdfsclosure(\sofinsot)\neq\finitesetsofthings\), then we find that \(\cohsdfsclosure(\sofinsot)\) is the smallest, or most conservative, finitely coherent SDFS that includes~\(\sofinsot\), for any finitely consistent~\(\sofinsot\).
Of course,
\begin{equation*}
\sofinsot
=\cohsdfsclosure(\sofinsot)\ifandonlyif\sofinsot\in\toppedcohsdfss,
\text{ for all~\(\sofinsot\subseteq\finitesetsofthings\)},
\end{equation*}
and therefore also \(\toppedcohsdfss=\cohsdfsclosure(\setsoffinitesetsofthings)\).
The following result is then once again a standard conclusion in order theory; see Ref.~\cite[Chapter~7]{davey2002} for the argumentation.

\begin{proposition}\label{prop:coherent:sdfs:constitute:a:complete:lattice}
The partially ordered set~\(\structure{\toppedcohsdfss,\subseteq}\) is a complete lattice.
For any non-empty family~\(\sofinsot[i]\), \(i\in I\) of elements of~\(\toppedcohsdfss\), we have for its infimum and its supremum that, respectively, \(\inf_{i\in I}\sofinsot[i]=\bigcap_{i\in I}\sofinsot[i]\) and \(\sup_{i\in I}\sofinsot[i]=\cohsdfsclosure\group{\bigcup_{i\in I}\sofinsot[i]}\).
\end{proposition}
\noindent The top~\(\cohsdfstop=\bigcup\toppedcohsdfss\) of this complete lattice is the largest closed SDFS~\(\finitesetsofthings\); its bottom~\(\cohsdfsbottom=\cohsdfsclosure(\emptyset)=\bigcap\toppedcohsdfss=\bigcap\cohsdfss\) is the smallest coherent SDFS, which, here too, is easy to identify.

\begin{proposition}\label{prop:the:smallest:cohsdfs}
\(\cohsdfsbottom=\bigcap\cohsdfss=\cset{\finsot\in\finitesetsofthings}{\finsot\cap\beautifulthings\neq\emptyset}\).
\end{proposition}

Let's now investigate which of the results of Section~\ref{sec:desirability:conjunctive:models} we can recover in the context of SDFSes.
Of course, we'll also call an SDFS~\(\sofinsot\subseteq\finitesetsofthings\) \emph{conjunctive} if for all~\(\finsot\in\sofinsot\) there is some~\(\thing\in\finsot\) such that \(\set{\thing}\in\sofinsot\), and we'll try to find out what the conjunctive \emph{finitely coherent} SDFSes are.
In the light of what we found for finite coherence in Section~\ref{sec:desirability:conjunctive:models}, we expect the finitary character of the closure operator~\(\initialclosure\) to play a role in some of these results.

If we introduce the map
\begin{equation}\label{eq:sofinsot:to:finsot:and:finsot:to:sofinsot}
\sdfsify{\bolleke}
\colon\setsofthings\to\setsoffinitesetsofthings
\colon\sot\mapsto\sdfsify{\sot}
\coloneqq\cset{\finsot\Subset\things}{\finsot\cap\sot\neq\emptyset},
\end{equation}
then it's immediately clear that \(\sdfsify{\bolleke}\) is an order embedding: \(\sot[1]\subseteq\sot[2]\) if and only if \(\sdfsify{\sot[1]}\subseteq\sdfsify{\sot[2]}\), for all \(\sot[1],\sot[2]\subseteq\things\).
We'll see further on in Theorem~\ref{thm:sdt:to:sdfs:order:embedding} that it allows us to embed \(\cohsdts\) into~\(\cohsdfss\).

First of all, the maps~\(\sdtify{\bolleke}\) and~\(\sdfsify{\bolleke}\) preserve (finite) coherence.

\begin{proposition}\label{prop:cohsdfs:to:cohsdt}
Consider any finitely coherent SDFS~\(\cohsdfs\subseteq\finitesetsofthings\), then \(\sdfsify{\sdtify{\cohsdfs}}\subseteq\cohsdfs\).
If, moreover, the closure operator~\(\initialclosure\) is finitary, then \(\sdtify{\cohsdfs}\) is a coherent SDT.
\end{proposition}

\begin{proposition}\label{prop:cohsdt:to:cohsdfs}
Consider any set of things~\(\cohsdt\subseteq\things\), then \(\sdtify{\sdfsify{\cohsdt}}=\cohsdt\).
Moreover, the following statements hold:
\begin{enumerate}[label=\upshape(\roman*),leftmargin=*,noitemsep]
\item if \(\cohsdt\) is a coherent SDT then \(\sdfsify{\cohsdt}\) is a finitely coherent SDFS;
\item if the closure operator~\(\initialclosure\) is finitary, and \(\sdfsify{\cohsdt}\) is a finitely coherent SDFS, then \(\cohsdt\) is a coherent SDT.
\end{enumerate}
Consequently, if the closure operator~\(\initialclosure\) is finitary, then \(\cohsdt\) is a coherent SDT if and only if \(\sdfsify{\cohsdt}\) is a finitely coherent SDFS.
\end{proposition}

\noindent So, assuming that the closure operator~\(\initialclosure\) is finitary, if we start out with a finitely coherent SDFS~\(\cohsdfs\), then the coherent SDT~\(\sdtify{\cohsdfs}\) and the corresponding finitely coherent and conjunctive SDFS~\(\sdfsify{\sdtify{\cohsdfs}}\) are conservative approximations of~\(\cohsdfs\): going from a model~\(\cohsdfs\) to its \emph{conjunctive part}~\(\sdfsify{\sdtify{\cohsdfs}}=\cset{\finsot\in\cohsdfs}{(\exists\thing\in\finsot)\set{\thing}\in\cohsdfs}\) typically results in a loss of information.
On the other hand, going from a coherent SDT~\(\cohsdt\) to the corresponding finitely coherent and conjunctive SDFS~\(\sdfsify{\cohsdt}\) \emph{never} results in a loss of information; see Theorem~\ref{thm:sdt:to:sdfs:order:embedding} below.

Combining these results allows us to find out what the conjunctive and finitely coherent SDFSes look like.

\begin{proposition}[Conjunctivity]\label{prop:conjunctive:and:coherent:sdfs}
Assume that the closure operator~\(\initialclosure\) is finitary.
An SDFS~\(\cohsdfs\subseteq\finitesetsofthings\) is finitely coherent and conjunctive if and only if there's some coherent SDT~\(\cohsdt\in\cohsdts\) such that \(\cohsdfs=\sdfsify{\cohsdt}\), and then necessarily \(\cohsdt=\sdtify{\cohsdfs}\).
\end{proposition}

Combining the results of Propositions~\ref{prop:cohsdt:to:cohsdfs} and~\ref{prop:conjunctive:and:coherent:sdfs} leads at once to a formal statement of what we alluded to in the introduction to this section.

\begin{theorem}\label{thm:sdt:to:sdfs:order:embedding}
Assume that the closure operator~\(\initialclosure\) is finitary.
The map~\(\sdfsify{\bolleke}\) is an order embedding of the (intersection) structure~\(\structure{\cohsdts,\subseteq}\) into the (intersection) structure~\(\structure{\cohsdfss,\subseteq}\).
Its image is the set of all conjunctive finitely coherent SDFSes, and on this set its inverse is the map~\(\sdtify{\bolleke}\).
\end{theorem}

\section{Towards a representation with filters}\label{sec:towards:representation}
It's a well-established consequence of Stone's Representation Theorem \cite[Chapters~5, 10 and~11]{davey2002} that filters of subsets of a space constitute abstract ways of dealing with deductively closed sets of propositions about elements of that space, or in other words and very simply put, they allow us to do propositional logic with statements about elements of the space.
In Stone's Theorem, the subsets that the filters are composed of constitute a Boolean lattice.

Further on, we'll also have occasion to work with filters of subsets, but where the subsets no longer constitute a Boolean, but only a bounded distributive lattice.
The role of Stone's Theorem will therefore be taken over by its generalisation to that more general class of lattices, which is the Prime Filter Representation Theorem.
So, where do such filters on distributive lattices come from in the context described above?

In general, we'll denote a space by~\(\mathscr{X}\).
A \emph{filter} of subsets of~\(\mathscr{X}\)---also called a filter on~\(\structure{\powerset(\mathscr{X}),\subseteq}\) ---is then a non-empty subset~\(\filter\) of the power set~\(\powerset(\mathscr{X})\) of~\(\mathscr{X}\) such that:
\begin{enumerate}[label=\upshape SF\({}_{\arabic*}\).,ref=\upshape SF\({}_{\arabic*}\),leftmargin=*]
\item\label{axiom:filters:increasing} if \(A\in\filter\) and~\(A\subseteq B\) then also \(B\in\filter\), for all~\(A,B\in\powerset(\mathscr{X})\);
\item\label{axiom:filters:intersections} if \(A\in\filter\) and \(B\in\filter\) then also \(A\cap B\in\filter\), for all~\(A,B\in\powerset(\mathscr{X})\).
\end{enumerate}
We call a filter \emph{proper} if \(\filter\neq\powerset(\mathscr{X})\), or equivalently, if \(\emptyset\notin\filter\).
In these definitions, the collection of subsets \(\powerset(\mathscr{X})\) may also be replaced by a bounded (distributive) lattice of subsets, where intersection plays the role of infimum and union the role of supremum.

The particular space~\(\mathscr{X}\) that we'll be considering in this paper, is the set~\(\cohsdts\) of all coherent SDTs.
To guide the interpretation of what we're doing, we'll assume that there's an \emph{actual} (but unknown) SDT, which we'll denote by~\(\idealcohsdt\).
This SDT~\(\idealcohsdt\) is assumed to be coherent, and therefore a specific element of the set~\(\cohsdts\).
The elements of~\(\idealcohsdt\) are the things that \emph{actually} are desirable, and all other things in~\(\things\) aren't.
Moreover, each coherent SDT~\(\cohsdt\in\cohsdts\) is a \emph{possible} identification of this actual set~\(\idealcohsdt\).

Any non-contradictory propositional statement about~\(\idealcohsdt\) corresponds to some non-empty subset~\(A\subseteq\cohsdts\) of coherent SDTs for which the statement holds true, and this subset~\(A\) represents the remaining possible identifications of~\(\idealcohsdt\) after the statement has been made.
We'll call such subsets \emph{events}.
The empty subset of~\(\cohsdts\) is the event that represents contradictory propositional statements.

Any proper filter~\(\filter\) of such events~\(A\subseteq\cohsdts\) then corresponds to a deductively closed collection of propositional statements---a so-called theory---about~\(\idealcohsdt\), where intersection of events represents the conjunction of propositional statements, and inclusion of events represents implication of propositional statements.
The only improper filter~\(\powerset(\cohsdts)\), which contains the empty event, then represents logical contradiction at this level.

We can interpret the desirability statements studied in Sections~\ref{sec:desirability} and~\ref{sec:desirability:sdfs} as statements about such an actual~\(\idealcohsdt\).
Let's make this more clear.
Stating that a `\emph{set of things~\(\sot\) is desirable}' corresponds to the event
\begin{equation}\label{eq:the:cohsdtify:operator}
\cohsdts[\sot]
\coloneqq\cset{\cohsdt\in\cohsdts}{\sot\cap\cohsdt\neq\emptyset},
\end{equation}
as this amounts to requiring that at least one element of~\(\sot\) must be actually desirable, and must therefore belong to~\(\idealcohsdt\).\footnote{It will be obvious that whatever is said about sets of things~\(\sot\subseteq\things\) and sets of sets of things~\(\sosot\subseteq\setsofthings\) in this section, also holds in particular for their finite versions~\(\finsot\Subset\things\) and \(\sofinsot\Subset\finitesetsofthings\).}
In other words, the desirability statement is equivalent to `\(\idealcohsdt\in\cohsdts[\sot]\)'.
As a special case, stating that a thing~\(\thing\) is desirable corresponds to the event~\(\cohsdts[\set{\thing}]\coloneqq\cset{\cohsdt\in\cohsdts}{\thing\in\cohsdt}\), as it amounts to requiring that \(\thing\) must belong to~\(\idealcohsdt\).
Also observe that \(\cohsdts[\emptyset]=\emptyset\), the so-called \emph{impossible event}.

More generally, working with SDSes~\(\sosot\) as we did in Sections~\ref{sec:desirability} and~\ref{sec:desirability:sdfs}, therefore corresponds to dealing with a conjunction of the desirability statements `the set of things~\(\sot\) is desirable' for all~\(\sot\in\sosot\), or in other words with events of the type
\begin{equation}\label{eq:definition:event}
\event{\sosot}
\coloneqq\bigcap_{\sot\in\sosot}\cohsdts[\sot]
=\bigcap_{\sot\in\sosot}\cset{\cohsdt\in\cohsdts}{\sot\cap\cohsdt\neq\emptyset},
\text{ where \(\sosot\subseteq\setsofthings\)},
\end{equation}
where, of course, as a special case we find that~\(W=\emptyset\) corresponds to a vacuous assessment, which leads to no restrictions on~\(\idealcohsdt\): \(\event{\emptyset}=\cohsdts\).\footnote{This is in accordance with Equation~\eqref{eq:definition:event}, as the empty intersection of subsets of~\(\cohsdts\) is \(\cohsdts\) itself.}

Working with the filters of subsets of~\(\cohsdts\)---filters of events---that are generated by such collections, then represents doing propositional logic with basic statements of the type `the set of things~\(\sot\) is desirable', for \(\sot\in\setsofthings\).
We might therefore suspect that the language of such filters could be able to represent, explain, and perhaps also refine the relationships between the inference mechanisms that lie behind the intersection structures and closure operators in Sections~\ref{sec:desirability} and~\ref{sec:desirability:sdfs}.
Investigating this type of representation in terms of filters of events is the main aim of this paper.

There is, however, a particular aspect of the inference mechanisms at hand that tends to complicate---or is it simplify?---matters somewhat.
Not all events in~\(\powerset(\cohsdts)\) are relevant to our problem; only the ones that are intersections (and, as we'll see further on in Theorem~\ref{thm:basic:sets:constitute:a:bounded:lattice}, unions) of the basic events of the type~\(\cohsdts[\sot]\), \(\sot\subseteq\things\) seem to require attention.
We'll therefore restrict our focus to these, and as a result, the representing collection of events will no longer constitute a Boolean lattice, but only a specific \emph{distributive sublattice}.
As we'll see in Sections~\ref{sec:representation:finitary}--\ref{sec:representation:finitary:finite}, the effect will be two-fold: we'll broadly speaking be led to a more general \emph{prime filter} rather than an \emph{ultrafilter representation}, and this representation will be an isomorphism rather than an endomorphism.

\section{The basic representation lattices}\label{sec:representation:lattices}
Since we expect the events~\(\event{\sosot}\), \(\sosot\subseteq\setsofthings\) to become important in what follows, let's take some time in this section to study them a bit closer.

In order to do this, we need to introduce some new notation, which we'll have occasion to use later in other sections as well.
In any partially ordered set~\(\structure{\lattice,\leq}\), we let
\begin{multline}\label{eq:up:operators}
\posetupset{a}
\coloneqq\cset{b\in\lattice}{a\leq b}
\text{ and }
\posetUpset{A}
\coloneqq\bigcup_{a\in A}\posetupset{a}
=\cset{b\in\lattice}{(\exists a\in A)a\leq b}\\
\text{ for all~\(a\in\lattice\) and~\(A\subseteq\lattice\)}.
\end{multline}
Further on, we'll also make use of the fact that \(\posetUpset{\bigcup_{i\in I}A_i}=\bigcup_{i\in I}\posetUpset{A_i}\) and that \(\posetUpset{\posetUpset{A}}=\posetUpset{A}\).
In fact, \(\posetUpset{\bolleke}\) is a closure operator.

In the special case that \(\structure{\lattice,\leq}=\structure{\cohsdts,\subseteq}\) that we're about to consider now, we get
\begin{multline*}
\cohsdtupset{\cohsdt}
=\cset{\cohsdt'\in\cohsdts}{\cohsdt\subseteq\cohsdt'}
\text{ and }
\cohsdtUpset{A}
=\cset{\cohsdt'\in\cohsdts}{(\exists\cohsdt\in A)\cohsdt\subseteq\cohsdt'}\\
\text{ for all~\(\cohsdt\in\cohsdts\) and~\(A\subseteq\cohsdts\)}.
\end{multline*}

The following proposition is fairly easy to prove, but it will be instrumental in our discussion further on.
It links the events~\(\event{\sosot}\) to our discussion of desirability in Section~\ref{sec:desirability}.
We start with a definition of a type of event that's intricately linked with the production axioms~\ref{axiom:desirable:sets:production} and~\ref{axiom:desirable:sets:finitary:production} there:\footnote{We can call \ref{axiom:desirable:sets:increasing}, \ref{axiom:desirable:sets:forbidden} and \ref{axiom:desirable:sets:production} \emph{production} axioms, because they allow us to produce new desirable sets from old ones. The axiom~\ref{axiom:desirable:sets:consistency} serves a more destructive role.}
\begin{equation}\label{eq:definition:production:event}
\basicevent{\sosot}
\coloneqq\cset[\big]{\initialclosure\group{\selection(\sosot)}}
{\selection\in\selections}
\cap\cohsdts,
\text{ for~\(\sosot\subseteq\setsofthings\)}.
\end{equation}
Observe that \(\basicevent{\sosot}\subseteq\cohsdts\) for all~\(\sosot\subseteq\setsofthings\); in particular, \(\basicevent{\emptyset}=\set{\initialclosure\group{\emptyset}}\cap\cohsdts=\set{\beautifulthings}\).

\begin{proposition}\label{prop:the:two:systems}
Consider any~\(\sosot\subseteq\setsofthings\), then \(\event{\sosot}=\cohsdtUpset{\basicevent{\sosot}}\).
\end{proposition}

It turns out that the propositional statement that `\(\sosot\) is an SDS' is never contradictory as soon as \(\sosot\) is a subset of some coherent SDS.
A similar result holds for finitely coherent SDSes, of course, at least for finite~\(\sosot\).

\begin{proposition}[Consistency]\label{prop:filter:consistency}
The following statements hold.
\begin{enumerate}[label=\upshape(\roman*),leftmargin=*,noitemsep]
\item For any coherent SDS~\(\cohsds\in\cohsdss\): \(\event{\sosot}\neq\emptyset\) for all~\(\sosot\subseteq\cohsds\);
\item For any finitely coherent SDS~\(\cohsds\in\fincohsdss\): \(\event{\sosot}\neq\emptyset\) for all~\(\sosot\Subset\cohsds\);
\item For any finitely coherent SDFS~\(\cohsdfs\in\cohsdfss\): \(\event{\sofinsot}\neq\emptyset\) for all~\(\sofinsot\Subset\cohsdfs\).
\end{enumerate}
\end{proposition}

There are two simple properties (in a finitary and infinitary version) that all our results about filter representation will essentially rest upon.
Indeed, the following proposition is the crucial workhorse for the argumentation in the coming sections.
It nicely connects the order-theoretic aspects of working with events to those of working with (finitely) coherent SD(F)Ses.

\begin{proposition}\label{prop:it:does:not:matter:which}
Consider any coherent SDS~\(\cohsds\in\cohsdss\), then the following statements hold for all~\(\sosot[1],\sosot[2]\subseteq\setsofthings\):
\begin{enumerate}[label=\upshape(\roman*),leftmargin=*]
\item\label{it:it:does:not:matter:which:inclusion:infinitary} if \(\eventwithindex{1}\subseteq\eventwithindex{2}\) and \(\sosot[1]\subseteq\cohsds\) then also \(\sosot[2]\subseteq\cohsds\);
\item\label{it:it:does:not:matter:which:equality:infinitary} if \(\eventwithindex{1}=\eventwithindex{2}\) then \(\sosot[1]\subseteq\cohsds\ifandonlyif\sosot[2]\subseteq\cohsds\).
\end{enumerate}
Similarly, consider any finitely coherent SDS~\(\cohsds\in\fincohsdss\), then the following statements hold for all~\(\sosot[1],\sosot[2]\Subset\setsofthings\):
\begin{enumerate}[label=\upshape(\roman*),leftmargin=*]
\item\label{it:it:does:not:matter:which:inclusion} if \(\eventwithindex{1}\subseteq\eventwithindex{2}\) and \(\sosot[1]\Subset\cohsds\) then also \(\sosot[2]\Subset\cohsds\);
\item\label{it:it:does:not:matter:which:equality} if \(\eventwithindex{1}=\eventwithindex{2}\) then \(\sosot[1]\Subset\cohsds\ifandonlyif\sosot[2]\Subset\cohsds\).
\end{enumerate}
And finally, consider any finitely coherent SDFS~\(\cohsdfs\in\cohsdfss\), then the following statements hold for all~\(\sofinsot[1],\sofinsot[2]\Subset\finitesetsofthings\):
\begin{enumerate}[label=\upshape(\roman*),leftmargin=*]
\item\label{it:it:does:not:matter:which:inclusion:finite} if \(\fineventwithindex{1}\subseteq\fineventwithindex{2}\) and \(\sofinsot[1]\Subset\cohsdfs\) then also \(\sofinsot[2]\Subset\cohsdfs\);
\item\label{it:it:does:not:matter:which:equality:finite} if \(\fineventwithindex{1}=\fineventwithindex{2}\) then \(\sofinsot[1]\Subset\cohsdfs\ifandonlyif\sofinsot[2]\Subset\cohsdfs\).
\end{enumerate}
\end{proposition}

We're now ready to introduce the particular collections of events we'll build our further discussion on.
Let's consider the sets of events
\begin{equation*}
\events\coloneqq\cset[\big]{\event{\sosot}}{\sosot\subseteq\setsofthings}\text{, }
\finevents\coloneqq\cset[\big]{\event{\sosot}}{\sosot\Subset\setsofthings}
\text{ and }
\finfinevents\coloneqq\cset[\big]{\event{\sofinsot}}{\sofinsot\Subset\finitesetsofthings},
\end{equation*}
and order each of them by set inclusion~\(\subseteq\).
That doing so leads to bounded distributive lattices, will be crucial for our filter representation efforts, since it will allow us to use the Prime Filter Representation Theorem, as we'll explain in the next section.

\begin{theorem}\label{thm:basic:sets:constitute:a:bounded:lattice}
The following statements hold:
\begin{enumerate}[label=\upshape(\roman*),leftmargin=*]
\item the partially ordered set~\(\structure{\events,\subseteq}\) is a completely distributive complete lattice, with union as join and intersection as meet, \(\emptyset\) as bottom and \(\cohsdts\) as top;
\item the partially ordered set \(\structure{\finevents,\subseteq}\) is a bounded distributive lattice, with union as join and intersection as meet, \(\emptyset\) as bottom and \(\cohsdts\) as top;
\item the partially ordered set \(\structure{\finfinevents,\subseteq}\) is a bounded distributive lattice, with union as join and intersection as meet, \(\emptyset\) as bottom and \(\cohsdts\) as top.
\end{enumerate}
\end{theorem}
\noindent Obviously, \(\finfinevents\subseteq\finevents\subseteq\events\), and the inclusion is also a sublattice relation.

\section{A brief primer on (inference with) filters}\label{sec:filters}
The introductory discussion in Section~\ref{sec:towards:representation} led us to try and represent inference about desirability statements using filters on appropriate sets of events.
After spending some effort on identifying these sets of events as bounded distributive lattices in Section~\ref{sec:representation:lattices}, we're now ready to start looking at how to do inference with filters, and how to use that inference mechanism to represent reasoning about desirability statements.
The present section summarises those aspects of filters and filter inference on (bounded distributive) lattices that are relevant to our representation effort.

We begin by recalling the definition of a filter on a bounded lattice~\(\structure{\lattice,\leq}\) with meet~\(\meet\) and join~\(\join\); we'll denote the lattice's top by~\(\latticetop\) and its bottom by~\(\latticebottom\).
It's an immediate generalisation of the definition of a filter of subsets we gave near the beginning of Section~\ref{sec:towards:representation}.

\begin{definition}[Filters]\label{def:filters}
Let~\(\structure{\lattice,\leq}\) be a bounded lattice with meet~\(\meet\) and join~\(\join\), bottom~\(\latticebottom\) and top~\(\latticetop\).
A non-empty subset~\(\filter\) of the set~\(\lattice\) is called a \emph{filter} on~\(\structure{\lattice,\leq}\) if it satisfies the following properties:
\begin{enumerate}[label=\upshape LF\({}_{\arabic*}\).,ref=\upshape LF\({}_{\arabic*}\),leftmargin=*]
\item\label{axiom:lattice:filters:increasing} if \(a\in\filter\) and~\(a\leq b\) then also \(b\in\filter\), for all~\(a,b\in\lattice\);
\item\label{axiom:lattice:filters:intersections} if \(a\in\filter\) and \(b\in\filter\) then also \(a\meet b\in\filter\), for all~\(a,b\in\lattice\).
\end{enumerate}
We call a filter~\(\filter\) \emph{proper} if \(\filter\neq\lattice\), or equivalently, if \(\latticebottom\notin\filter\).
We denote the set of all proper filters of~\(\structure{\lattice,\leq}\) by~\(\properlatticefilters\), and the set of all filters by~\(\latticefilters=\properlatticefilters\cup\set{\lattice}\).
\end{definition}

We continue with a short discussion of filter bases, inspired by the discussion in, for instance, Ref.~\cite[Section~12]{willard1970}.
This will be relevant for some of the proofs in the Appendix.
A  non-empty subset~\(\filterbase\) of a proper filter~\(\filter\) is called a \emph{filter base} for~\(\filter\) if
\[
\group{\forall a\in\lattice}
\group[\big]{a\in\filter\ifandonlyif\group{\exists b\in\filterbase}b\leq a},
\]
or in other words if \(\filter=\posetUpset{\filterbase}=\cset{a\in\lattice}{\group{\exists b\in\filterbase}b\leq a}\).
Clearly, a non-empty  subset~\(\filterbase\) of \(\lattice\setminus\set{\latticebottom}\) is a filter base for some proper filter if and only if it is \emph{directed downwards}, meaning that
\begin{equation}\label{eq:directed:downwards}
\group{\forall b_1,b_2\in\filterbase}\group{\exists b\in\filterbase}b\leq b_1\meet b_2.
\end{equation}

The inference mechanism that's associated with filters is, as are all such mechanisms, based on the idea of \emph{closure} and \emph{intersection structures}, which we brought to the fore in Section~\ref{sec:desirability}.
Here too, it's easy to see that the set~\(\properlatticefilters\) of all proper filters on a \emph{bounded}\footnote{That the lattice is bounded, guarantees that intersections of filters always contain the lattice top, and are therefore non-empty.} lattice~\(\structure{\lattice,\leq}\) is indeed an intersection structure, meaning that it's closed under arbitrary non-empty intersections: for any non-empty family~\(\filter[i]\), \(i\in I\) of elements of~\(\properlatticefilters\), we see that still \(\bigcap_{i\in I}\filter[i]\in\properlatticefilters\).\footnote{Similarly, \(\latticefilters=\properlatticefilters\cup\set{\lattice}\) is a \emph{topped} intersection structure: closed under arbitrary (also empty) intersections \cite[Chapter~7]{davey2002}.}

Again, if we associate with this intersection structure the map~\(\filterclosure\colon\powerset(\lattice)\to\latticefilters\) defined by
\begin{equation*}
\filterclosure(\setofpropositions)
\coloneqq\bigcap\cset{\filter\in\latticefilters}{\setofpropositions\subseteq\filter}
=\bigcap\cset{\filter\in\properlatticefilters}{\setofpropositions\subseteq\filter}
\text{ for all \(\setofpropositions\subseteq\lattice\)},
\end{equation*}
then this map is a \emph{closure operator}.

In this language, the filters are the perfect, or deductively closed, subsets of the bounded lattice~\(\structure{\lattice,\leq}\), and the closure operator can be used to extend any set of lattice elements to the smallest deductively closed set that includes it.
If we call a set of lattice elements~\(\setofpropositions\) \emph{filterisable} if it's included in some proper filter, or equivalently, if \(\filterclosure(\setofpropositions)\neq\lattice\), then we find that \(\filterclosure(\setofpropositions)\) is the smallest proper filter that includes~\(\setofpropositions\), for any filterisable set~\(\setofpropositions\).
Of course,
\begin{equation*}
\setofpropositions=\filterclosure(\setofpropositions)\ifandonlyif\setofpropositions\in\latticefilters,
\text{ for all~\(\setofpropositions\subseteq\lattice\)},
\end{equation*}
and therefore also \(\latticefilters=\filterclosure(\powerset(\lattice))\).
Once again, the following result is then a standard conclusion in order theory \cite[Chapter~7]{davey2002}.

\begin{proposition}\label{prop:filters:constititute:a:complete:lattice}
The partially ordered set~\(\structure{\latticefilters,\subseteq}\) is a complete lattice.
For any non-empty family~\(\filter[i]\), \(i\in I\) of elements of~\(\latticefilters\), we have for its infimum and its supremum that, respectively, \(\inf_{i\in I}\filter[i]=\bigcap_{i\in I}\filter[i]\) and \(\sup_{i\in I}\filter[i]=\filterclosure\group{\bigcup_{i\in I}\filter[i]}\).
\end{proposition}
\noindent The top~\(\bigcup\latticefilters\) of this complete lattice is the largest filter~\(\lattice\); its bottom~\(\filterclosure(\emptyset)=\bigcap\latticefilters=\bigcap\properlatticefilters\) is the smallest proper filter~\(\set{\latticetop}\).

Two special types of filters deserve more attention in the light of what's to come.

\subsection*{Principal filters}
In the special case that \(\structure{\lattice,\leq}\) is a \emph{complete} lattice, we can replace the finite meets in~\ref{axiom:lattice:filters:intersections} by arbitrary, possibly infinite ones, as in
\begin{enumerate}[label={\upshape LF}\({}_2^{\mathrm{p}}\).,ref={\upshape LF}\({}_2^{\mathrm{p}}\),leftmargin=*]
\item\label{axiom:lattice:filters:arbitrary:intersections} if \(A\subseteq\filter\) then also \(\inf A\in\filter\), for all non-empty~\(A\subseteq\lattice\).
\end{enumerate}
In particular, we then find that \(\inf\filter\in\filter\), and it's not hard to show that then \(\filter=\posetupset{\inf\filter}\).
Such a so-called \emph{principal filter} is clearly proper if and only if \(\inf\filter\neq\latticebottom\); see Ref.~\cite[Section~2.20]{davey2002} for the terminology.
We see that the set of all principal filters, partially ordered by set inclusion, is trivially order-isomorphic to the complete lattice~\(\structure{\lattice,\leq}\) itself.

\subsection*{Prime filters on distributive lattices}
A \emph{prime filter}~\(\primefilter\) on~\(\structure{\lattice,\leq}\) is a proper filter that also satisfies the following condition:
\begin{enumerate}[label=\upshape LPF.,ref=\upshape LPF,leftmargin=*]
\item\label{axiom:lattice:filters:prime} if \(a\join b\in\primefilter\) then also \(a\in\primefilter\) or \(b\in\primefilter\), for all~\(a,b\in\lattice\).
\end{enumerate}
We denote the set of all prime filters on~\(\structure{\lattice,\leq}\) by \(\latticeprimefilters\).
When \(\structure{\lattice,\leq}\) is a bounded distributive lattice, any proper filter can be represented by prime filters, as it's the intersection of all the prime filters that include it.
This is the Prime Filter Representation Theorem; see Ref.~\cite[Sections~10.7--21]{davey2002} for more details.

\begin{theorem}[Prime filter representation]\label{thm:prime:filter:representation}
Let \(\structure{\lattice,\leq}\) be a bounded distributive lattice.
Then any non-empty~\(\filter\subseteq\lattice\) is a filter if and only if \(\filter=\bigcap\cset{\primefilter\in\latticeprimefilters}{\filter\subseteq\primefilter}\).
\end{theorem}

\subsection*{Taking stock}
Now that we know what the inference mechanism underlying filters is, we can clarify what we mean by \emph{filter representation} of other inference mechanisms.
Axioms~\ref{axiom:desirable:sets:consistency}--\ref{axiom:desirable:sets:production} govern the inference mechanism behind the desirability of sets of things, and we've seen in Section~\ref{sec:desirability:sds}, and in particular in Proposition~\ref{prop:coherent:sds:constitute:a:complete:lattice}, that its mathematical essence can be condensed into the complete lattice~\(\structure{\toppedcohsdss,\subseteq}\) and the closure operator~\(\cohsdsclosure\).
Similarly, the finitary version of this inference mechanism is laid down in Axioms~\ref{axiom:desirable:sets:consistency}--\ref{axiom:desirable:sets:finitary:production}, and is captured by the complete lattice~\(\structure{\toppedfincohsdss,\subseteq}\) and the closure operator~\(\fincohsdsclosure\).

In Section~\ref{sec:towards:representation}, we voiced our suspicion that the language of filters might be able to represent, explain, and perhaps also refine the relationships between the inference mechanisms that lie behind the intersection structures and closure operators in Sections~\ref{sec:desirability} and \ref{sec:desirability:sdfs}.
Reduced to its essence, this suspicion leads to the following question: can we find bounded lattices~\(\structure{\lattice,\leq}\) such that the complete lattice~\(\structure{\latticefilters,\subseteq}\) and the closure operator~\(\filterclosure\) are essentially the same as---can be identified through an order isomorphism with---the complete lattice~\(\structure{\toppedcohsdss,\subseteq}\) and the closure operator~\(\cohsdsclosure\); or in the finitary case, the same as the complete lattice~\(\structure{\toppedfincohsdss,\subseteq}\) and the closure operator~\(\fincohsdsclosure\)?
We'll show in Sections~\ref{sec:representation:finitary} and~\ref{sec:representation} that, indeed, we can find such bounded lattices: the completely distributive complete lattice of events~\(\structure{\events,\subseteq}\) and the bounded distributive lattice of events~\(\structure{\finevents,\subseteq}\), respectively.

Why bother?
What's so special about such representations in terms of filters of events?
The answer's twofold.

First of all, there's the issue of interpretation we've already drawn attention to in Section~\ref{sec:towards:representation}.
The events in~\(\events\), \(\finevents\) and~\(\finfinevents\) represent propositional statements about the desirability of things, and filters of such events represent collections of such propositional statements that are closed under logical deduction---conjunction and modus ponens.
The order isomorphisms that we'll identify below then simply tell us that making inferences about desirable things and desirable sets of things based on the Axioms~\ref{axiom:desirable:things:closure}--\ref{axiom:desirable:things:background:consistency} and~\ref{axiom:desirable:sets:consistency}--\ref{axiom:desirable:sets:production}/\ref{axiom:desirable:sets:finitary:production} (and \ref{axiom:desirable:finite:sets:consistency}--\ref{axiom:desirable:finite:sets:production}) is mathematically equivalent to doing a type of propositional logic with propositional statements about the desirability of things.

The second reason has a more mathematical flavour.
Since the sets of events~\(\events\), \(\finevents\) and~\(\finfinevents\) constitute bounded distributive lattices when ordered by set inclusion, we can make use of the Prime Filter Representation Theorem on such bounded distributive lattices, which states that any filter can be written as the intersection of all the prime filters it's included in.
The order isomorphisms we're about to identify in the following sections, will then allow us to transport this theorem to the context of (finitely) coherent SD(F)Ses, and write these as intersections of special types of them, namely the conjunctive ones.
This will lead us directly to the so-called conjunctive representation results for coherent SDSes in Theorem~\ref{thm:conjunctive:representation}, for finitely coherent SDSes in Theorem~\ref{thm:conjunctive:representation:finitary}, and for finitely coherent SDFSes in Corollary~\ref{cor:conjunctive:representation:finitely:coherent:finite} below.

\section{Filter representation for finitely coherent SDSes}\label{sec:representation:finitary}
We're now first going to focus on \emph{finitely} coherent SDSes, and try to relate them to the filters on the bounded distributive lattice~\(\structure{\finevents,\subseteq}\).
This will lead directly to a so-called conjunctive representation result of finitely coherent SDSes in terms of---as limits inferior of---conjunctive ones.
We'll then see in subsequent sections that coherent SDSes and finitely coherent SDFSes also have conjunctive representation results, which turn out to be formally simpler.

We adapt the notations and definitions related to filters in Section~\ref{sec:filters} from a generic bounded distributive lattice~\(\smash{\structure{\lattice,\leq}}\) to the specific bounded distributive lattice~\(\structure{\finevents,\subseteq}\).
This leads to the complete lattice~\(\structure{\fineventfilters,\subseteq}\) of all filters on~\(\finevents\), with bottom~\(\fineventfiltersbottom=\bigcap\fineventfilters=\set{\cohsdts}\) and top~\(\fineventfilterstop=\finevents\); see also Theorem~\ref{thm:basic:sets:constitute:a:bounded:lattice} and Proposition~\ref{prop:filters:constititute:a:complete:lattice}.
We also denote the set of all proper filters by~\(\properfineventfilters\coloneqq\fineventfilters\setminus\set{\finevents}\), and the corresponding closure operator by~\(\fineventfilterclosure\).

In order to establish the existence of an order isomorphism between the complete lattices~\(\structure{\toppedfincohsdss,\subseteq}\) and~\(\structure{\fineventfilters,\subseteq}\), we consider the maps
\begin{equation*}
\finfilterise\colon\setsofsetsofthings\to\powerset(\finevents)
\colon\cohsds\mapsto\finfilterise(\cohsds)\coloneqq\cset{\event{\sosot}}{\sosot\Subset\cohsds}
\end{equation*}
and
\begin{equation*}
\findesirify\colon\powerset(\finevents)\to\setsofsetsofthings
\colon\filter\mapsto\findesirify(\filter)
\coloneqq\cset{\sot\in\setsofthings}{\cohsdts[\sot]\in\filter},
\end{equation*}
which, as we'll see presently, do the job.

\begin{theorem}[Order isomorphism: finitely coherent SDSes]\label{thm:the:finitary:representation:theorem}
The following statements hold, for all \(\cohsds,\cohsds[1],\cohsds[2]\subseteq\setsofthings\) and all \(\filter,\filter[1],\filter[2]\subseteq\finevents\):
\begin{enumerate}[label=\upshape(\roman*),leftmargin=*,noitemsep]
\item\label{it:the:finitary:representation:theorem:if:cohsds:then:filter} if \(\cohsds\) is a finitely coherent SDS then \(\finfilterise(\cohsds)\) is a proper filter on~\(\structure{\finevents,\subseteq}\);
\item\label{it:the:finitary:representation:theorem:if:filter:then:cohsds} if \(\filter\) is a proper filter on~\(\structure{\finevents,\subseteq}\), then \(\findesirify(\filter)\) is a finitely coherent SDS;
\item\label{it:the:finitary:representation:theorem:cohsds:embedding} if \(\cohsds\) is a finitely coherent SDS, then \((\findesirify\circ\finfilterise)(\cohsds)=\cohsds\);
\item\label{it:the:finitary:representation:theorem:filter:embedding} if \(\filter\) is a proper filter on~\(\structure{\finevents,\subseteq}\), then \((\finfilterise\circ\findesirify)(\filter)=\filter\);
\item\label{it:the:finitary:representation:theorem:cohsds:order:preserving} if \(\cohsds[1]\subseteq\cohsds[2]\) then \(\finfilterise(\cohsds[1])\subseteq\finfilterise(\cohsds[2])\);
\item\label{it:the:finitary:representation:theorem:filter:order:preserving} if \(\filter[1]\subseteq\filter[2]\) then \(\findesirify(\filter[1])\subseteq\findesirify(\filter[2])\);
\item\label{it:the:finitary:representation:theorem:filterise:bounds} \(\finfilterise(\fincohsdsbottom)=\fineventfiltersbottom\) and \(\finfilterise(\fincohsdstop)=\fineventfilterstop\);
\item\label{it:the:finitary:representation:theorem:desirify:bounds} \(\findesirify(\fineventfiltersbottom)=\fincohsdsbottom\) and \(\findesirify(\fineventfilterstop)=\fincohsdstop\).
\end{enumerate}
This tells us that \(\finfilterise\) is an order isomorphism between~\(\structure{\toppedfincohsdss,\subseteq}\) and~\(\structure{\fineventfilters,\subseteq}\), with inverse order isomorphism~\(\findesirify\).
Moreover,
\begin{enumerate}[label=\upshape(\roman*),leftmargin=*,resume]
\item\label{it:the:finitary:representation:theorem:completeness} if the proper filter~\(\filter=\finfilterise(\cohsds)\) and the finitely coherent SDS~\(\cohsds=\findesirify(\filter)\) are related by this order isomorphism, then \(\filter\) is a prime filter on~\(\structure{\finevents,\subseteq}\) if and only if \(\cohsds\) satisfies the so-called \emph{completeness condition}
\begin{equation}\label{eq:complete:cohsds}
\group{\forall\sot[1],\sot[2]\subseteq\things}
\group[\big]{\sot[1]\cup\sot[2]\in\cohsds\then
\group{\sot[1]\in\cohsds\text{ or }\sot[2]\in\cohsds}}.
\end{equation}
\end{enumerate}
\end{theorem}

As we're about to show, an important consequence of the existence of the order isomorphism in Theorem~\ref{thm:the:finitary:representation:theorem} is that it allows us to represent any finitely coherent SDS in terms of coherent but conjunctive models.
This is interesting because, by Proposition~\ref{prop:conjunctive:and:coherent:sds}, such coherent conjunctive SDSes are conceptually much simpler, as they represent SDTs---they only represent \emph{conjunctive} desirability statements.

To see how this representation in terms of conjunctive models comes about, we begin by recalling that the events \(\event{\sosot}\) for~\(\sosot\Subset\setsofthings\) are sets of coherent SDTs.
They are completely determined by the following chain of equivalences: consider any~\(\cohsdt\in\cohsdts\), then
\begin{equation*}
\cohsdt\in\event{\sosot}
\ifandonlyif\cohsdt\in\bigcap_{\sot\in\sosot}\cohsdts[\sot]
\ifandonlyif\group{\forall\sot\in\sosot}\sot\cap\cohsdt\neq\emptyset
\ifandonlyif\sosot\Subset\sdsify{\cohsdt},
\end{equation*}
where we recall from Equation~\eqref{eq:sos:to:sot:and:sot:to:sos} that \(\sdsify{\cohsdt}\coloneqq\cset{\sot\in\setsofthings}{\sot\cap\cohsdt\neq\emptyset}\).
This tells us that
\begin{equation}\label{eq:connection:with:binary:model:finitary}
\event{\sosot}=\cset{\cohsdt\in\cohsdts}{\sosot\Subset\sdsify{\cohsdt}},
\text{ for all~\(\sosot\Subset\setsofthings\)}.
\end{equation}

As a next step, we now consider any proper filter~\(\filter\in\properfineventfilters\) and any finitely coherent SDS~\(\cohsds\in\fincohsdss\) that correspond, in the sense that \(\cohsds=\findesirify(\filter)\) and \(\filter=\finfilterise(\cohsds)\).
On the one hand, we infer from \(\cohsds=\findesirify(\filter)\) that for any~\(\sot\in\setsofthings\):
\begin{align*}
\sot\in\cohsds
\ifandonlyif\cohsdts[\sot]\in\filter
\ifandonlyif\group{\exists\sop\in\filter}\sop\subseteq\cohsdts[\sot]
&\ifandonlyif\group{\exists\sop\in\filter}\group{\forall\cohsdt\in\sop}\sot\cap\cohsdt\neq\emptyset\\
&\ifandonlyif\group{\exists\sop\in\filter}\group{\forall\cohsdt\in\sop}\sot\in\sdsify{\cohsdt},
\end{align*}
where the second equivalence follows from the filter property~\ref{axiom:filters:increasing}.
This chain of equivalences therefore tells us that
\begin{equation}\label{eq:from:filter:to:cohsds:finitary}
\cohsds=\bigcup_{\sop\in\filter}\bigcap_{\cohsdt\in\sop}\sdsify{\cohsdt}.
\end{equation}
On the other hand, we infer from \(\filter=\finfilterise(\cohsds)\) and Equation~\eqref{eq:connection:with:binary:model:finitary} that
\begin{equation}\label{eq:from:cohsds:to:filter:finitary}
\filter=\cset[\big]{\cset{\cohsdt\in\cohsdts}{\sosot\Subset\sdsify{\cohsdt}}}{\sosot\Subset\cohsds}.
\end{equation}

Taking into account the consequences of Equations~\eqref{eq:connection:with:binary:model:finitary}--\eqref{eq:from:cohsds:to:filter:finitary} leads to the following representation result for finite consistency, finite coherence, and the corresponding closure operator~\(\fincohsdsclosure\) in terms of the conjunctive models~\(\sdsify{\cohsdt}\).

\begin{theorem}[Conjunctive representation]\label{thm:conjunctive:representation:finite}
Consider any SDS~\(\cohsds\subseteq\setsofthings\), then the following statements hold:
\begin{enumerate}[label=\upshape(\roman*),leftmargin=*]
\item\label{it:conjunctive:representation:finite:consistency} \(\cohsds\) is finitely consistent if and only if \(\event{\sosot}=\cset{\cohsdt\in\cohsdts}{\sosot\Subset\sdsify{\cohsdt}}\neq\emptyset\) for all~\(\sosot\Subset\cohsds\);
\item\label{it:conjunctive:representation:finite:closure} \(\fincohsdsclosure(\cohsds)=\bigcup_{\sosot\Subset\cohsds}\bigcap_{\cohsdt\in\cohsdts\colon\sosot\Subset\sdsify{\cohsdt}}\sdsify{\cohsdt}\);
\item\label{it:conjunctive:representation:finite:coherence} \(\cohsds\) is finitely coherent if and only if \(\cohsds\) is finitely consistent and
\[
\cohsds
=\bigcup_{\sosot\Subset\cohsds}\bigcap_{\cohsdt\in\cohsdts\colon\sosot\Subset\sdsify{\cohsdt}}
\sdsify{\cohsdt}.
\]
\end{enumerate}
\end{theorem}

\noindent This tells us that an SDS is finitely consistent if and only if each of its finite subsets is included in some conjunctive model, and that any finitely coherent SDS can be written as a limit inferior of conjunctive models.
Even if the representation in terms of such limits inferior is formally somewhat complicated, it has the advantage that the basic representing models are the conjunctive ones, which are easy to identify and `construct'.

There is, however, another representation result that's formally simpler, but where the representing models are less easy to `construct': a representation that's based on the representing role that prime filters play in bounded distributive lattices; see the discussion in Ref.~\cite[Sections~10.7--21]{davey2002} and Section~\ref{sec:filters} for more details.
Let's now, in the remainder of this section, explain how it comes about.

\begin{definition}[Completeness]
We call an SDS~\(\sosot\subseteq\setsofthings\) \emph{complete} if it satisfies the completeness condition~\eqref{eq:complete:cohsds}:
\begin{equation*}
\group{\forall\sot[1],\sot[2]\subseteq\things}
\group[\big]{\sot[1]\cup\sot[2]\in\sosot\then
\group{\sot[1]\in\sosot\text{ or }\sot[2]\in\sosot}};
\end{equation*}
we denote by~\(\completefincohsdss\) the set of all complete and finitely coherent SDSes, and by~\(\completesdss\) the set of all complete and coherent SDSes.
\end{definition}
\noindent Theorem~\ref{thm:the:finitary:representation:theorem}\ref{it:the:finitary:representation:theorem:completeness} tells us that the complete finitely coherent SDSes are in a one-to-one relationship with the prime filters on the distributive lattice~\(\structure{\finevents,\subseteq}\), and the order isomorphism~\(\findesirify\) identified in that theorem allows us to easily transform the prime filter representation result of Theorem~\ref{thm:prime:filter:representation} into the following alternative representation theorem for finitely coherent SDSes.

\begin{theorem}[Prime filter representation]\label{thm:prime:filter:representation:finitely:coherent}
A finitely consistent SDS~\(\cohsds\subseteq\setsofthings\) is finitely coherent if and only if \(\cohsds=\bigcap\cset{\cohsds'\in\completefincohsdss}{\cohsds\subseteq\cohsds'}\).
\end{theorem}

As suggested above, a disadvantage of this type of representation is that, generally speaking, the complete SDSes are---much like their prime filter counterparts---hard if not impossible to identify `constructively'.
That, however, they include all coherent conjunctive models, will be very helpful in our discussion of the coherent finitary models in Section~\ref{sec:finitary:models} further on.

\begin{proposition}\label{prop:conjunctive:is:complete}
Consider any coherent SDT~\(\cohsdt\in\cohsdts\), then the (finitely) coherent conjunctive SDS~\(\sdsify{\cohsdt}\) is complete.
\end{proposition}

\section{Filter representation for coherent SDSes}\label{sec:representation}
We next turn to a representation result for the coherent, rather than merely finitely coherent, SDSes, and as is to be expected, we'll focus on the filters on the set~\(\events\) in order to achieve that.

It turns out that our representation will only involve the proper \emph{principal} filters on this set~\(\events\).
Let's denote the set of all principal filters by \(\eventprincipalfilters\coloneqq\cset{\eventupset{\event{\sosot}}}{\sosot\subseteq\setsofthings}\), and the set of all proper principal filters on~\(\events\) by \(\propereventprincipalfilters\), where
\begin{equation*}
\propereventprincipalfilters
=\cset{\eventupset{\event{\sosot}}}{\sosot\subseteq\setsofthings}\setminus\set{\events}
=\cset{\eventupset{\event{\sosot}}}{\sosot\subseteq\setsofthings\text{ and }\event{\sosot}\neq\emptyset},
\end{equation*}
and where we let, as a special case of the general definition in Equation~\eqref{eq:up:operators},
\begin{equation*}
\eventupset{\event{\sosot}}
=\cset{\event{\sosot'}}{\sosot'\subseteq\setsofthings\text{ and }\event{\sosot}\subseteq\event{\sosot'}}.
\end{equation*}
It's readily verified that~\(\eventprincipalfilters\) is closed under arbitrary non-empty intersections, and therefore an intersection structure: for any~\(\sosot\subseteq\setsofthings\) and any non-empty family~\(\sosot[i]\subseteq\setsofthings\), \(i\in I\) we see that
\begin{equation*}
\group{\forall i\in I}\event{\sosot[i]}\subseteq\event{\sosot}
\ifandonlyif\bigcup_{i\in I}\event{\sosot[i]}\subseteq\event{\sosot},
\end{equation*}
and Equation~\eqref{eq:union:of:possibles} in Appendix~\ref{sec:proofs:representation:lattices} guarantees that there's some~\(\sosot'\subseteq\setsofthings\) for which \(\event{\sosot'}=\bigcup_{i\in I}\event{\sosot[i]}\).
It's in fact clearly also a topped intersection structure, and therefore \(\structure{\eventprincipalfilters,\subseteq}\) is a complete lattice, with intersection as infimum, and with bottom~\(\eventprincipalfiltersbottom=\set{\cohsdts}\) and top~\(\eventprincipalfilterstop=\events\).

In order to establish the existence of an order isomorphism between the complete lattices~\(\structure{\toppedcohsdss,\subseteq}\) and~\(\structure{\eventprincipalfilters,\subseteq}\), we now consider the maps
\begin{equation*}
\filterise\colon\setsofsetsofthings\to\powerset(\events)
\colon\cohsds\mapsto\filterise(\cohsds)\coloneqq\cset{\event{\sosot}}{\sosot\subseteq\cohsds}
\end{equation*}
and
\begin{equation*}
\desirify\colon\powerset(\events)\to\setsofsetsofthings
\colon\filter\mapsto\desirify(\filter)
\coloneqq\cset{\sot\in\setsofthings}{\cohsdts[\sot]\in\filter},
\end{equation*}
which are the obvious counterparts of the maps defined in Section~\ref{sec:representation:finitary}, and which as we'll see presently, do the job in this infinitary context.

\begin{theorem}[Order isomorphism: coherent SDSes]\label{thm:the:representation:theorem}
The following statements hold, for all \(\cohsds,\cohsds[1],\cohsds[2]\subseteq\setsofthings\) and all \(\filter,\filter[1],\filter[2]\subseteq\events\):
\begin{enumerate}[label=\upshape(\roman*),leftmargin=*]
\item\label{it:the:representation:theorem:if:cohsds:then:filter} if \(\cohsds\) is a coherent SDS then \(\filterise(\cohsds)\) is a proper principal filter on~\(\structure{\events,\subseteq}\);
\item\label{it:the:representation:theorem:if:filter:then:cohsds} if \(\filter\) is a proper principal filter on~\(\structure{\events,\subseteq}\), then \(\desirify(\filter)\) is a coherent SDS;
\item\label{it:the:representation:theorem:cohsds:embedding} if \(\cohsds\) is a coherent SDS, then \((\desirify\circ\filterise)(\cohsds)=\cohsds\);
\item\label{it:the:representation:theorem:filter:embedding} if \(\filter\) is a proper principal filter on~\(\structure{\events,\subseteq}\), then \((\filterise\circ\desirify)(\filter)=\filter\);
\item\label{it:the:representation:theorem:cohsds:order:preserving} if \(\cohsds[1]\subseteq\cohsds[2]\) then \(\filterise(\cohsds[1])\subseteq\filterise(\cohsds[2])\);
\item\label{it:the:representation:theorem:filter:order:preserving} if \(\filter[1]\subseteq\filter[2]\) then \(\desirify(\filter[1])\subseteq\desirify(\filter[2])\);
\item\label{it:the:representation:theorem:filterise:bounds} \(\filterise(\cohsdsbottom)=\eventprincipalfiltersbottom\) and \(\filterise(\cohsdstop)=\eventprincipalfilterstop\);
\item\label{it:the:representation:theorem:desirify:bounds} \(\desirify(\eventprincipalfiltersbottom)=\cohsdsbottom\) and \(\desirify(\eventprincipalfilterstop)=\cohsdstop\).
\end{enumerate}
This tells us that \(\filterise\) is an order isomorphism between~\(\structure{\toppedcohsdss,\subseteq}\) and~\(\structure{\propereventprincipalfilters,\subseteq}\), with inverse order isomorphism~\(\desirify\).
\end{theorem}

To use this order isomorphism to find a representation of coherent SDSes in terms of the conjunctive models, we begin by recalling the earlier finitary argument and transporting it to the current infinitary context: consider any~\(\sosot\subseteq\setsofthings\) and any~\(\cohsdt\in\cohsdts\), then
\[
\cohsdt\in\event{\sosot}
\ifandonlyif\cohsdt\in\bigcap_{\sot\in\sosot}\cohsdts[\sot]
\ifandonlyif\group{\forall\sot\in\sosot}\sot\cap\cohsdt\neq\emptyset
\ifandonlyif\sosot\subseteq\sdsify{\cohsdt},
\]
so
\begin{equation}\label{eq:connection:with:binary:model}
\event{\sosot}=\cset{\cohsdt\in\cohsdts}{\sosot\subseteq\sdsify{\cohsdt}},
\text{ for all~\(\sosot\subseteq\setsofthings\)}.
\end{equation}

Again similarly to what we did for finitely coherent SDSes in Section~\ref{sec:representation:finitary}, we now consider any proper principal filter~\(\filter\in\propereventprincipalfilters\) and any coherent SDS~\(\cohsds\in\cohsdss\) that correspond, in the sense that \(\cohsds=\desirify(\filter)\) and \(\filter=\filterise(\cohsds)\).
Observe that the principal filter~\(\filter=\filterise(\cohsds)\) is completely determined by its smallest element \(\bigcap\filter\), which is the subset of~\(\cohsdts\) given by:
\begin{equation}\label{eq:conjunctive:representation:smallest:element}
\bigcap\filter
=\bigcap\filterise(\cohsds)
=\bigcap\cset{\event{\sosot}}{\sosot\subseteq\cohsds}
=\event{\cohsds},
\end{equation}
where the last equality follows readily from Equation~\eqref{eq:definition:event}.
Hence, on the one hand,
\begin{equation}\label{eq:from:cohsds:to:filter}
\filter=\filterise(\cohsds)=\eventupset{\event{\cohsds}}.
\end{equation}
On the other hand, we infer from \(\cohsds=\desirify(\filter)\) that for any~\(\sot\in\setsofthings\),
\begin{align*}
\sot\in\cohsds
\ifandonlyif\cohsdts[\sot]\in\filter
\ifandonlyif\cohsdts[\sot]\in\eventupset{\event{\cohsds}}
\ifandonlyif\event{\cohsds}\subseteq\cohsdts[\sot]
&\ifandonlyif\group{\forall\cohsdt\in\event{\cohsds}}\cohsdt\cap\sot\neq\emptyset\\
&\ifandonlyif\group{\forall\cohsdt\in\event{\cohsds}}\sot\in\sdsify{\cohsdt}.
\end{align*}
So, if we combine this chain of equivalences with Equation~\eqref{eq:connection:with:binary:model}, we can conclude that
\begin{equation}\label{eq:from:filter:to:cohsds}
\cohsds
=\bigcap_{\cohsdt\in\event{\cohsds}}\sdsify{\cohsdt}
=\bigcap_{\cohsdt\in\cohsdts\colon\cohsds\subseteq\sdsify{\cohsdt}}\sdsify{\cohsdt}.
\end{equation}

Taking into account the consequences of Equations~\eqref{eq:connection:with:binary:model}--\eqref{eq:from:filter:to:cohsds} leads to the following representation result for consistency, coherence, and the corresponding closure operator~\(\cohsdsclosure\) in terms of the conjunctive models~\(\sdsify{\cohsdt}\).

\begin{theorem}[Conjunctive representation]\label{thm:conjunctive:representation}
Consider any SDS~\(\cohsds\subseteq\setsofthings\), then the following statements hold:
\begin{enumerate}[label=\upshape(\roman*),leftmargin=*]
\item\label{it:conjunctive:representation:consistency} \(\cohsds\) is consistent if and only if \(\event{\cohsds}=\cset{\cohsdt\in\cohsdts}{\cohsds\subseteq\sdsify{\cohsdt}}\neq\emptyset\);
\item\label{it:conjunctive:representation:closure} \(\cohsdsclosure(\cohsds)=\bigcap_{\cohsdt\in\cohsdts\colon\cohsds\subseteq\sdsify{\cohsdt}}\sdsify{\cohsdt}\);
\item\label{it:conjunctive:representation:coherence} \(\cohsds\) is coherent if and only if \(\cohsds\) is consistent and \(\cohsds=\bigcap_{\cohsdt\in\cohsdts\colon\cohsds\subseteq\sdsify{\cohsdt}}\sdsify{\cohsdt}\).
\end{enumerate}
\end{theorem}
\noindent A similar result was derived in a different manner, without the simplifying use of the order-isomorphisms~\(\desirify\) and \(\filterise\), by De Bock in Ref.~\cite[Theorem~1]{debock2023:things:arxiv}.

\section{Finitary SDSes}\label{sec:finitary:models}
We conclude from the discussion in the previous section that the conjunctive representation for coherent SDSes is remarkably simpler than the one for merely finitely coherent SDSes.
But, as we'll explain presently, it turns out that we can recover the simpler conjunctive representation---in terms of intersections rather than limits inferior---also for finitely coherent SDSes, \emph{provided that we focus on their finite elements}.
This was already proved by De Bock \cite{debock2023:things:arxiv}, based on ideas in our earlier papers~\cite{debock2018,debock2019b}, but we intend to derive this remarkable result here using our filter representation approach, which allows for a different and arguably somewhat simpler proof, based on the Prime Filter Representation Theorem.

\subsection*{Finitary SDSes}
Recall that we denote by \(\finitesetsofthings\) the set of all finite sets of things, \(\finitesetsofthings=\cset{\sot\in\setsofthings}{\sot\Subset\things}\), and that we follow the convention that the empty set~\(\emptyset\) is finite, so \(\emptyset\in\finitesetsofthings\).

If \(\sosot\subseteq\setsofthings\) is an SDS, then we call \(\finpart{\sosot}\coloneqq\sosot\cap\finitesetsofthings\) its \emph{finite part} and
\begin{equation*}
\fintypart{\sosot}
\coloneqq\sotUpset{\finpart{\sosot}}
=\sotUpset{\sosot\cap\finitesetsofthings}
=\cset{\sot\in\setsofthings}{(\exists\finsot\in\sosot\cap\finitesetsofthings)\finsot\Subset\sot}
\end{equation*}
its \emph{finitary part}, where we let, for ease of notation, for all~\(\sot\in\setsofthings\) and~\(\sosot\subseteq\setsofthings\):
\begin{multline*}
\sotupset{\sot}\coloneqq\cset{\sot'\in\setsofthings}{\sot\subseteq\sot'}
\text{ and }\\
\sotUpset{\sosot}
\coloneqq\bigcup_{\sot\in\sosot}\sotupset{\sot}
=\cset{\sot'\in\setsofthings}{\group{\exists\sot''\in\sosot}\sot''\subseteq\sot'}.
\end{multline*}
We'll call an SDS~\(W\) \emph{finitary} if each of its desirable sets has a finite desirable subset, meaning that
\begin{equation*}
\group{\forall\sot\in\sosot}(\exists\finsot\in\sosot\cap\finitesetsofthings)\finsot\Subset\sot
\text{, or equivalently, }
\sosot\subseteq\fintypart{\sosot},
\end{equation*}

Interestingly, for any (finitely) coherent SDS~\(\cohsds\), the coherence condition~\ref{axiom:desirable:sets:increasing} guarantees that, since \(\cohsds\cap\finitesetsofthings\subseteq\cohsds\), also \(\fintypart{\cohsds}\subseteq\cohsds\).
This tells us that a (finitely) coherent SDS~\(\cohsds\) is finitary if and only if \(\cohsds=\fintypart{\cohsds}\), or in other words:

\begin{proposition}\label{prop:finitary:equal:to:finitary:part}
A finitary and (finitely) coherent SDS~\(\cohsds\) is equal to its finitary part~\(\fintypart{\cohsds}\), and therefore completely determined by its finite part~\(\finpart{\cohsds}=\cohsds\cap\finitesetsofthings\).
\end{proposition}
\noindent Moreover, it's easy to see that \(\fintypart{\fintypart{\cohsds}}=\fintypart{\cohsds}\) for any (finitely) coherent SDS~\(\cohsds\), implying that \emph{its finitary part \(\fintypart{\cohsds}\) is always finitary.}
If nothing else, this shows that our terminology is consistent.

Does the (finite) coherence of an SDS imply the coherence (finite or otherwise) of its finitary part?
The following proposition provides the beginning of an answer, which we'll be able to complete further on in Corollary~\ref{cor:finitary:and:coherent}.

\begin{proposition}\label{prop:finite:coherence:of:finitary:part}
If an SDS~\(\cohsds\) is (finitely) coherent, then its finitary part~\(\fintypart{\cohsds}\) is finitely coherent.
\end{proposition}

\subsection*{The relation with completeness and conjunctivity}
Let's now find out more about how, for a (finitely) coherent SDS, being finitary relates to being complete, and in particular to being conjunctive.
This will allow us to establish in Theorem~\ref{thm:conjunctive:representation:finitary} below that for finitary and finitely coherent SDSes, the conjunctive representation of Theorem~\ref{thm:conjunctive:representation:finite} simplifies.

For a start, all coherent and conjunctive SDSes are, of course, finitary.

\begin{proposition}\label{prop:conjunctive:is:finitary}
Consider any coherent SDT~\(\cohsdt\in\cohsdts\), then the (finitely) coherent conjunctive SDS~\(\sdsify{\cohsdt}\) is finitary: \(\fintypart{\sdsify{\cohsdt}}=\sdsify{\cohsdt}\).
\end{proposition}

But what is the relation for a coherent SDS between its being conjunctive and its being complete?
On the one hand, coherent SDSes that are conjunctive are always complete; see Propositions~\ref{prop:conjunctive:and:coherent:sds} and~\ref{prop:conjunctive:is:complete}.
On the other hand, complete coherent SDSes are not necessarily conjunctive, but we're now about to find out that they \emph{necessarily have a conjunctive finitary part}.
Consequently, the coherent conjunctive SDSes are exactly the complete and coherent SDSes that are finitary; and the same holds for finitely coherent conjunctive SDSes provided that the closure operator~\(\initialclosure\) is finitary.

The following proposition gives a more detailed statement.

\begin{proposition}\label{prop:complete:finitary:is:conjunctive}
For any complete and coherent SDS \(\cohsds\in\completesdss\), there's some~\(\cohsdt\in\cohsdts\) such that \(\finpart{\cohsds}=\finpart{\sdsify{\cohsdt}}\), and therefore \(\fintypart{\cohsds}=\fintypart{\sdsify{\cohsdt}}=\sdsify{\cohsdt}\), namely \(\cohsdt=\sdtify{\cohsds}\).
Moreover, if the closure operator~\(\initialclosure\) is finitary, then for any complete and finitely coherent SDS \(\cohsds\in\completefincohsdss\), there's some~\(\cohsdt\in\cohsdts\) such that \(\finpart{\cohsds}=\finpart{\sdsify{\cohsdt}}\), and therefore \(\fintypart{\cohsds}=\fintypart{\sdsify{\cohsdt}}=\sdsify{\cohsdt}\), namely \(\cohsdt=\sdtify{\cohsds}\).
\end{proposition}

Since the finitary part of a complete and (finitely) coherent SDS is conjunctive, the Prime Filter Representation Theorem (in the version formulated as Theorem~\ref{thm:prime:filter:representation:finitely:coherent}) results in a representation with conjunctive models, provided that the closure operator~\(\initialclosure\) is finitary.

\begin{theorem}[Conjunctive representation]\label{thm:conjunctive:representation:finitary}
If the closure operator~\(\initialclosure\) is finitary, then a finitary and finitely consistent SDS~\(\cohsds\subseteq\setsofthings\) is finitely coherent if and only if \(\cohsds=\bigcap\cset{\sdsify{\cohsdt}}{\cohsdt\in\cohsdts\text{ and }\cohsds\subseteq\sdsify{\cohsdt}}\).
\end{theorem}

This leads to the remarkable conclusion that for finitary SDSes there's no difference between finite coherence and coherence, as long as the closure operator~\(\initialclosure\) is finitary.
Similar conclusions, arrived at in a very different manner, were drawn by De Bock in Ref.~\cite[Corollary~1 and Proposition~1]{debock2023:things:arxiv}

\begin{corollary}\label{cor:finitary:and:coherent}
If the closure operator~\(\initialclosure\) is finitary, then
\begin{enumerate}[label=\upshape(\roman*),leftmargin=*]
\item\label{it:finitary:and:coherent:finitely:coherent} any finitary SDS is finitely coherent if and only if it's coherent;
\item\label{it:finitary:and:coherent:finitary:is:coherent} the finitary part \(\fintypart{\cohsds}\) of any (finitely) coherent SDS~\(\cohsds\) is coherent.
\end{enumerate}
\end{corollary}
\noindent Observe that the Satan's Apple counterexample in Section~\ref{sec:desirability:conjunctive:models} shows that the finitary character of the closure operator~\(\initialclosure\) cannot generally be let go of in this result.

\section{Filter representation for finitely coherent SDFSes}\label{sec:representation:finitary:finite}
At this point, we have gathered enough background material to turn to finitely coherent SDFSes, and try to relate them to the filters on the distributive lattice~\(\structure{\finfinevents,\subseteq}\).
As is the case for finitely coherent SDSes, this will lead directly to a conjunctive representation result of finitely coherent SDSFes in terms of---as limits inferior of---conjunctive ones.
And, as a next step, we'll then use the Prime Filter Representation Theorem and the relevant results in Sections~\ref{sec:representation:finitary} and~\ref{sec:finitary:models} to investigate when a simpler representation using \emph{intersections} of conjunctive SDFSes is possible.

The first step is almost completely analogous to the discussion of filter representation for finitely coherent SDSes in Section~\ref{sec:representation:finitary}.
We adapt the notations and definitions related to filters in Section~\ref{sec:filters} from a generic bounded distributive lattice~\(\smash{\structure{\lattice,\leq}}\) to the specific bounded distributive lattice~\(\structure{\finfinevents,\subseteq}\).
This leads to the complete lattice~\(\structure{\finfineventfilters,\subseteq}\) of all filters on~\(\finfinevents\), with bottom~\(\finfineventfiltersbottom=\bigcap\finfineventfilters=\set{\cohsdts}\) and top~\(\finfineventfilterstop=\finfinevents\); see also Theorem~\ref{thm:basic:sets:constitute:a:bounded:lattice} and Proposition~\ref{prop:filters:constititute:a:complete:lattice}.
We also denote the set of all proper filters by~\(\properfinfineventfilters\coloneqq\finfineventfilters\setminus\set{\finfinevents}\), and the corresponding closure operator by~\(\finfineventfilterclosure\).

In order to establish the existence of an order isomorphism between the complete lattices~\(\structure{\toppedcohsdfss,\subseteq}\) and~\(\structure{\finfineventfilters,\subseteq}\), we consider the maps
\begin{equation*}
\finfinfilterise\colon\setsoffinitesetsofthings\to\powerset(\finfinevents)
\colon\cohsdfs\mapsto\finfinfilterise(\cohsdfs)\coloneqq\cset{\event{\sofinsot}}{\sofinsot\Subset\cohsdfs}
\end{equation*}
and
\begin{equation*}
\finfindesirify\colon\powerset(\finfinevents)\to\setsoffinitesetsofthings
\colon\filter\mapsto\finfindesirify(\filter)
\coloneqq\cset{\finsot\in\finitesetsofthings}{\cohsdts[\finsot]\in\filter},
\end{equation*}
which, as we'll see presently, do the job.

\begin{theorem}[Order isomorphism: finitely coherent SDFSes]\label{thm:the:finitary:representation:theorem:finite}
The following statements hold, for all \(\cohsdfs,\cohsdfs[1],\cohsdfs[2]\subseteq\finitesetsofthings\) and all \(\filter,\filter[1],\filter[2]\subseteq\finevents\):
\begin{enumerate}[label=\upshape(\roman*),leftmargin=*]
\item\label{it:the:finitary:representation:theorem:finite:if:cohsds:then:filter} if \(\cohsdfs\) is a finitely coherent SDFS then \(\finfinfilterise(\cohsdfs)\) is a proper filter on~\(\structure{\finfinevents,\subseteq}\);
\item\label{it:the:finitary:representation:theorem:finite:if:filter:then:cohsds} if \(\filter\) is a proper filter on~\(\structure{\finfinevents,\subseteq}\), then \(\finfindesirify(\filter)\) is a finitely coherent SDFS;
\item\label{it:the:finitary:representation:theorem:finite:cohsds:embedding} if \(\cohsdfs\) is a finitely coherent SDFS, then \((\finfindesirify\circ\finfinfilterise)(\cohsdfs)=\cohsdfs\);
\item\label{it:the:finitary:representation:theorem:finite:filter:embedding} if \(\filter\) is a proper filter on~\(\structure{\finfinevents,\subseteq}\), then \((\finfinfilterise\circ\finfindesirify)(\filter)=\filter\);
\item\label{it:the:finitary:representation:theorem:finite:cohsds:order:preserving} if \(\cohsdfs[1]\subseteq\cohsdfs[2]\) then \(\finfinfilterise(\cohsdfs[1])\subseteq\finfinfilterise(\cohsdfs[2])\);
\item\label{it:the:finitary:representation:theorem:finite:filter:order:preserving} if \(\filter[1]\subseteq\filter[2]\) then \(\finfindesirify(\filter[1])\subseteq\finfindesirify(\filter[2])\);
\item\label{it:the:finitary:representation:theorem:finite:filterise:bounds} \(\finfinfilterise(\cohsdfsbottom)=\finfineventfiltersbottom\) and \(\finfinfilterise(\cohsdfstop)=\finfineventfilterstop\);
\item\label{it:the:finitary:representation:theorem:finite:desirify:bounds} \(\finfindesirify(\finfineventfiltersbottom)=\cohsdfsbottom\) and \(\finfindesirify(\finfineventfilterstop)=\cohsdfstop\).
\end{enumerate}
This tells us that \(\finfinfilterise\) is an order isomorphism between~\(\structure{\toppedcohsdfss,\subseteq}\) and~\(\structure{\finfineventfilters,\subseteq}\), with inverse order isomorphism~\(\finfindesirify\).
Moreover,
\begin{enumerate}[label=\upshape(\roman*),leftmargin=*,resume]
\item\label{it:the:finitary:representation:theorem:finite:completeness} if the proper filter~\(\filter=\finfinfilterise(\cohsdfs)\) and the finitely coherent SDFS~\(\cohsdfs=\finfindesirify(\filter)\) are related by this order isomorphism, then \(\filter\) is a prime filter on~\(\structure{\finfinevents,\subseteq}\) if and only if \(\cohsdfs\) satisfies the so-called \emph{completeness condition}
\begin{equation}\label{eq:complete:cohsdfs}
\group{\forall\finsot[1],\finsot[2]\Subset\things}
\group[\big]{\finsot[1]\cup\finsot[2]\in\cohsdfs\then
\group{\finsot[1]\in\cohsdfs\text{ or }\finsot[2]\in\cohsdfs}}.
\end{equation}
\end{enumerate}
\end{theorem}

To see how the representation in terms of conjunctive models comes about in this context, we recall that the events \(\event{\sofinsot}\) for~\(\sofinsot\Subset\finitesetsofthings\) are sets of coherent SDTs, completely determined by
\begin{equation*}
\cohsdt\in\event{\sofinsot}
\ifandonlyif\cohsdt\in\bigcap_{\finsot\in\sofinsot}\cohsdts[\finsot]
\ifandonlyif\group{\forall\finsot\in\sofinsot}\finsot\cap\cohsdt\neq\emptyset
\ifandonlyif\sofinsot\Subset\sdfsify{\cohsdt},
\text{ for all~\(\cohsdt\in\cohsdts\),}
\end{equation*}
where we recall from Equation~\eqref{eq:sofinsot:to:finsot:and:finsot:to:sofinsot} that \(\sdfsify{\cohsdt}\coloneqq\cset{\finsot\Subset\things}{\finsot\cap\cohsdt\neq\emptyset}\).
This tells us that
\begin{equation}\label{eq:connection:with:binary:model:finitary:finite}
\event{\sofinsot}=\cset{\cohsdt\in\cohsdts}{\sofinsot\Subset\sdfsify{\cohsdt}},
\text{ for all~\(\sofinsot\Subset\finitesetsofthings\)}.
\end{equation}

As a next step, we now consider any proper filter~\(\filter\in\properfinfineventfilters\) and any finitely coherent SDFS~\(\cohsdfs\in\cohsdfss\) that correspond, in the sense that \(\cohsdfs=\finfindesirify(\filter)\) and \(\filter=\finfinfilterise(\cohsdfs)\).
On the one hand, we infer from \(\cohsdfs=\finfindesirify(\filter)\) that for any~\(\finsot\in\finitesetsofthings\):
\begin{align*}
\finsot\in\cohsdfs
\ifandonlyif\cohsdts[\finsot]\in\filter
\ifandonlyif\group{\exists\sop\in\filter}\sop\subseteq\cohsdts[\finsot]
&\ifandonlyif\group{\exists\sop\in\filter}\group{\forall\cohsdt\in\sop}\finsot\cap\cohsdt\neq\emptyset\\
&\ifandonlyif\group{\exists\sop\in\filter}\group{\forall\cohsdt\in\sop}\finsot\in\sdfsify{\cohsdt},
\end{align*}
where the second equivalence follows from the filter property~\ref{axiom:filters:increasing}.
This chain of equivalences therefore tells us that
\begin{equation}\label{eq:from:filter:to:cohsdfs}
\cohsdfs=\bigcup_{\sop\in\filter}\bigcap_{\cohsdt\in\sop}\sdfsify{\cohsdt}.
\end{equation}
On the other hand, we infer from \(\filter=\finfinfilterise(\cohsdfs)\) and Equation~\eqref{eq:connection:with:binary:model:finitary:finite} that
\begin{equation}\label{eq:from:cohsdfs:to:filter}
\filter
=\cset[\big]{\cset{\cohsdt\in\cohsdts}{\sofinsot\Subset\sdfsify{\cohsdt}}}{\sofinsot\Subset\cohsdfs}.
\end{equation}
Taking into account the consequences of Equations~\eqref{eq:connection:with:binary:model:finitary:finite}--\eqref{eq:from:cohsdfs:to:filter} leads to the following representation result for finite consistency, finite coherence, and the corresponding closure operator~\(\cohsdfsclosure\) in terms of the conjunctive models~\(\sdfsify{\cohsdt}\).

\begin{theorem}[Conjunctive representation]\label{thm:conjunctive:representation:finite:finite}
Consider any SDFS~\(\cohsdfs\subseteq\finitesetsofthings\), then the following statements hold:
\begin{enumerate}[label=\upshape(\roman*),leftmargin=*]
\item\label{it:conjunctive:representation:finite:finite:consistency} \(\cohsdfs\) is finitely consistent if and only if \(\event{\sofinsot}=\cset{\cohsdt\in\cohsdts}{\sofinsot\Subset\sdfsify{\cohsdt}}\neq\emptyset\) for all~\(\sofinsot\Subset\cohsdfs\);
\item\label{it:conjunctive:representation:finite:finite:closure} \(\cohsdfsclosure(\cohsdfs)=\bigcup_{\sofinsot\Subset\cohsdfs}\bigcap_{\cohsdt\in\cohsdts\colon\sofinsot\Subset\sdfsify{\cohsdt}}\sdfsify{\cohsdt}\);
\item\label{it:conjunctive:representation:finite:finite:coherence} \(\cohsdfs\) is finitely coherent if and only if \(\cohsdfs\) is finitely consistent and
\[
\cohsdfs
=\bigcup_{\sofinsot\Subset\cohsdfs}\bigcap_{\cohsdt\in\cohsdts\colon\sofinsot\Subset\sdfsify{\cohsdt}}
\sdfsify{\cohsdt}.
\]
\end{enumerate}
\end{theorem}

As a next step, we'll allow ourselves to be inspired by the developments in Section~\ref{sec:representation:finitary}: we'll use the Prime Filter Representation Theorem~\ref{thm:prime:filter:representation} to find a simpler representation result in terms of complete and finitely coherent SDFSes.

\begin{definition}[Completeness]
We call an SDFS~\(\sofinsot\subseteq\finitesetsofthings\) \emph{complete} if it satisfies the completeness condition~\eqref{eq:complete:cohsdfs}:
\begin{equation*}
\group{\forall\finsot[1],\finsot[2]\Subset\things}
\group[\big]{\finsot[1]\cup\finsot[2]\in\sofinsot\then
\group{\finsot[1]\in\sofinsot\text{ or }\sot[2]\in\sofinsot}},
\end{equation*}
and we denote by~\(\completecohsdfss\) the set of all complete and finitely coherent SDFSes.
\end{definition}

\begin{theorem}[Prime filter representation]\label{thm:prime:filter:representation:finitely:coherent:finite}
A finitely consistent SDFS~\(\cohsdfs\subseteq\finitesetsofthings\) is finitely coherent if and only if \(\cohsdfs=\bigcap\cset{\cohsdfs'\in\completecohsdfss}{\cohsdfs\subseteq\cohsdfs'}\).
\end{theorem}

Contrary to the case of SDSes, the finitely coherent and complete SDFSes can be identified `constructively'.

\begin{proposition}\label{prop:complete:and:conjunctive:finite}
Consider any coherent SDT~\(\cohsdt\in\cohsdts\) and any finitely coherent and complete~\(\cohsdfs\in\completecohsdfss\).
Then the following statements hold:
\begin{enumerate}[label=\upshape(\roman*),leftmargin=*]
\item\label{it:complete:and:conjunctive:finite:sdt:to:sdfs} the finitely coherent conjunctive SDFS~\(\sdfsify{\cohsdt}\) is complete;
\item\label{it:complete:and:conjunctive:finite:sdfs:is:conjunctive} \(\cohsdfs=\sdfsify{\sdtify{\cohsdfs}}\), so \(\cohsdfs\) is conjunctive.
\end{enumerate}
Hence, if the closure operator~\(\initialclosure\) is finitary, then a finitely coherent SDFS is complete if and only if it is conjunctive, and then \(\completecohsdfss=\cset{\sdfsify{\cohsdt}}{\cohsdt\in\cohsdts}\).
\end{proposition}

Here too, our Satan's Apple counterexample can be used to show that we cannot generally let go of the finitary character of the closure operator~\(\initialclosure\) in these results.

\begin{counterexample}[Satan's Apple]
As before, we consider the set of things~\(\things\coloneqq\naturalswithzero\) with the (non-finitary) closure operator
\[
\initialclosure(\sot)
\coloneqq
\begin{cases}
\set{1}&\text{if \(\sot=\emptyset\)}\\
\set{1,\dots,\max\sot}&\text{if \(\sot\) is a finite and non-empty subset of~\(\naturals\)},\\
\naturalswithzero&\text{otherwise}
\end{cases}
\quad\text{for all~\(\sot\subseteq\things\)},
\]
and the set of forbidden things~\(\uglythings\coloneqq\set{0}\).
Recall that here \(\cohsdts=\cset{\set{1,\dots,n}}{n\in\naturals}\) and \(\beautifulthings=\set{1}\).

We consider the SDFS~\(\cohsdfs\coloneqq\finites(\naturalswithzero)\setminus\set{\emptyset,\set{0}}\).
It's easy to check that \(\cohsdfs\) is complete.
It is clearly conjunctive, and equal to its conjunctive part~\(\sdfsify{\sdtify{\cohsdfs}}\).
The SDT~\(\sdtify{\cohsdfs}=\naturals\) is clearly \emph{not coherent}.
Since the closure operator~\(\initialclosure\) isn't finitary, we can't use Proposition~\ref{prop:cohsdt:to:cohsdfs} to also prove the \emph{finite incoherence} of \(\cohsdfs\).
On the contrary, we'll now prove that~\(\cohsdfs\) \emph{is} in fact a \emph{finitely coherent} SDFS.
Nevertheless, it can't be associated with a coherent set of desirable things, as \(\sdtify{\cohsdfs}=\naturals\).

Indeed, since~\(\cohsdfs\) clearly satisfies~\ref{axiom:desirable:finite:sets:consistency}--\ref{axiom:desirable:finite:sets:background}, we concentrate on~\ref{axiom:desirable:finite:sets:production}.
Consider any non-empty~\(\sofinsot\Subset\cohsdfs\), so \(\finsot\notin\set{\emptyset,\set{0}}\) for all~\(\finsot\in\sofinsot\).
This already implies that there is some~\(\selection[o]\in\finselections\) such that \(\selection[o](\finsot)\neq0\) for all~\(\finsot\in\sofinsot\).
Now, consider for any~\(\selection\in\finselections\) any choice of the \(\thing[\selection]\in\initialclosure(\selection(\sofinsot))\), and assume towards contradiction that \(\cset{\thing[\selection]}{\selection\in\finselections}\in\set{\emptyset,\set{0}}\).
Since \(\finselections\) is non-empty because \(\sofinsot\) is, it must then be that \(\cset{\thing[\selection]}{\selection\in\finselections}=\set{0}\), so it must follow that \(0\in\initialclosure(\selection(\sofinsot))\), implying that \(\selection(\sofinsot)\) is an infinite subset of~\(\naturals\) or contains~\(0\), for all~\(\selection\in\finselections\).
That \(\selection(\sofinsot)\) should be infinite, is impossible for finite~\(\sofinsot\), so we find that \(0\in\selection(\sofinsot)=\cset{\selection(\finsot)}{\finsot\in\sofinsot}\) for all~\(\selection\in\finselections\), a contradiction.
\hfill\(\triangleleft\)
\end{counterexample}

If we now combine Theorem~\ref{thm:prime:filter:representation:finitely:coherent:finite} with Proposition~\ref{prop:complete:and:conjunctive:finite}, we find a simplified conjunctive representation result.
This provides an alternative and arguably simpler proof for a theorem that was also recently proved by Jasper De Bock \cite[Theorem~3]{debock2023:things:arxiv}, and that generalises our earlier results when things are gambles \cite{cooman2021:archimedean:choice,debock2018,debock2020:axiomatic:archimedean,debock2019b}, and whose proofs were much more involved.

\begin{corollary}[Conjunctive representation]\label{cor:conjunctive:representation:finitely:coherent:finite}
Assume that the closure operator~\(\initialclosure\) is finitary.
Then a finitely consistent SDFS~\(\cohsdfs\subseteq\finitesetsofthings\) is finitely coherent if and only if \(\cohsdfs=\bigcap\cset{\sdfsify{\cohsdt}}{\cohsdt\in\cohsdts\text{ and }\cohsdfs\subseteq\sdfsify{\cohsdt}}\).
\end{corollary}

The formal similarity between this conjunctive representation result for finitely coherent SDFSes and the one in Theorem~\ref{thm:conjunctive:representation:finitary} for finitely coherent finitary SDSes, is striking.
That this is no coincidence, is also made clear by our final result, which shows that when the closure operator~\(\initialclosure\) is finitary, there is a one-to-one correspondence between the finitely coherent SDFSes and the finitely coherent finitary SDSes; see also Ref.~\cite[Proposition~10]{debock2023:things:arxiv}.
The maps that provide this one-to-one connection, and which are each others inverses on the relevant sets, are \(\finpart{\bolleke}\) and \(\sotUpset{\bolleke}\).

\begin{theorem}\label{thm:finitely:coherent:isomorphism}
Assume that the closure operator~\(\initialclosure\) is finitary.
\begin{enumerate}[label=\upshape(\roman*),leftmargin=*]
\item\label{it:finitely:coherent:isomorphism:sdfs:to:sds} if \(\cohsdfs\) is a finitely coherent SDFS, then \(\sotUpset\cohsdfs\) is a finitely coherent SDS;
\item\label{it:finitely:coherent:isomorphism:sds:to:sdfs} if \(\cohsds\) is a (finitely) coherent SDS, then \(\finpart{\cohsds}\) is a finitely coherent SDFS;
\item\label{it:finitely:coherent:isomorphism:up:after:fin} for all finitary~\(\cohsds\in\fincohsdss\), \(\sotUpset{\finpart{\cohsds}}=\cohsds\);
\item\label{it:finitely:coherent:isomorphism:fin:after:up} for all~\(\cohsdfs\in\cohsdfss\), \(\finpart{\sotUpset{\cohsdfs}}=\cohsdfs\).
\end{enumerate}
\end{theorem}

What does it really mean that the conjunctive representation changes from its form in Theorem~\ref{thm:conjunctive:representation:finite:finite} to the simpler version in Corollary~\ref{cor:conjunctive:representation:finitely:coherent:finite}? (And {\itshape mutatis mutandis}, from its form in Theorem~\ref{thm:conjunctive:representation:finite} to the simpler one in Theorem~\ref{thm:conjunctive:representation:finitary}; which has a similar answer).

First of all, since we can infer from Theorem~\ref{thm:conjunctive:representation:finite:finite}\ref{it:conjunctive:representation:finite:finite:closure} that \(\cohsdfsclosure(\sofinsot)=\bigcap\cset{\sdfsify{\cohsdt}}{\sofinsot\Subset\sdfsify{\cohsdt}}\) for all~\(\sofinsot\Subset\finitesetsofthings\), the same result then guarantees that \(\cohsdfsclosure(\cohsdfs)=\bigcup_{\sofinsot\Subset\cohsdfs}\cohsdfsclosure(\sofinsot)\) for all~\(\cohsdfs\subseteq\finitesetsofthings\): the coherence axioms~\ref{axiom:desirable:finite:sets:consistency}--\ref{axiom:desirable:finite:sets:production} make sure that the resulting closure operator~\(\cohsdfsclosure\) is \emph{finitary}, and the `complicated form' of the conjunctive representation is simply an expression of this finitary character!

But then, where does the simplification in Corollary~\ref{cor:conjunctive:representation:finitely:coherent:finite} find its origin?
Explaining this will take some steps, so bear with us.

In a first step, we define the map
\[
\sdfs(\bolleke)\colon
\finites(\finitesetsofthings)\to\toppedcohsdfss\colon
\sofinsot\mapsto\sdfs(\sofinsot)\coloneqq\bigcap\cset{\sdfsify{\cohsdt}}{\cohsdt\in\cohsdts\text{ and }\sofinsot\Subset\sdfsify{\cohsdt}},
\]
in analogy with the map [see also Equations~\eqref{eq:definition:event} and~\eqref{eq:connection:with:binary:model:finitary:finite}]
\[
\event{\bolleke}\colon
\finites(\finitesetsofthings)\to\finfinevents\colon
\sofinsot\mapsto\event{\sofinsot}\coloneqq\cset{\cohsdt\in\cohsdts}{\sofinsot\Subset\sdfsify{\cohsdt}},
\]
then clearly \(\sdfs(\bolleke)=\bigcap\cset{\sdfsify{\cohsdt}}{\cohsdt\in\event{\bolleke}}\).

But, observe as a second step that the map \(\event{\bolleke}\) has actually been defined on a larger domain in Equation~\eqref{eq:definition:event}: indeed, \(\event{\sosot}=\cset{\cohsdt\in\cohsdts}{\sosot\subseteq\sdfsify{\cohsdt}}\) for all~\(\sosot\subseteq\finitesetsofthings\); see how Equation~\eqref{eq:connection:with:binary:model} actually extends Equation~~\eqref{eq:connection:with:binary:model:finitary:finite}.
This suggests that we can similarly extend the map~\(\sdfs(\bolleke)\) from~\(\finites(\finitesetsofthings)\) to the larger domain~\(\powerset(\finitesetsofthings)\), as follows:
\[
\sdfs(\bolleke)\colon
\powerset(\finitesetsofthings)\to\toppedcohsdfss\colon
\sosot\mapsto\sdfs(\sosot)\coloneqq\bigcap\cset{\sdfsify{\cohsdt}}{\cohsdt\in\cohsdts\text{ and }\sosot\subseteq\sdfsify{\cohsdt}},
\]
and then still \(\sdfs(\bolleke)=\bigcap\cset{\sdfsify{\cohsdt}}{\cohsdt\in\event{\bolleke}}\).

Thirdly, check that \(\event{\sdfsify{\cohsdt}}=\cohsdtupset{\cohsdt}\) and therefore \(\sdfs(\sdfsify{\cohsdt})=\sdfsify{\cohsdt}\) for all~\(\cohsdt\in\cohsdts\).
This allows us to conclude that \(\sdfs(\sosot)=\bigcap_{\cohsdt\in\event{\sosot}}\sdfs(\sdfsify{\cohsdt})\) for all~\(\sosot\subseteq\finitesetsofthings\): the map~\(\sdfs\) is completely determined by the (identity) values~\(\sdfsify{\cohsdt}\) it assumes on the~\(\sdfsify{\cohsdt}\), \(\cohsdt\in\cohsdts\), which---at least when the closure operator~\(\initialclosure\) is finitary---are all the conjunctive finitely coherent models [see Proposition~\ref{prop:complete:and:conjunctive:finite}].

Now, as a fourth step, observe that the conjunctive representation result in Theorem~\ref{thm:conjunctive:representation:finite:finite} tells us that, in general,
\[
\cohsdfsclosure(\sofinsot)=\sdfs(\sofinsot)
\text{ for all~\(\sofinsot\Subset\finitesetsofthings\)}
\]
and that, therefore,\footnote{Here and below, the `\(\lim\dots\)' is the (Moore--Smith) limit of the net~\(\sdfs(\sofinsot)\) associated with the set of finite subsets~\(\sofinsot\) of~\(\sosot\), which is a directed set under set inclusion \cite[Section~11]{willard1970}.}
\begin{equation}\label{eq:continuity:first:step}
\cohsdfsclosure(\sosot)
=\bigcup_{\sofinsot\Subset\sosot}\sdfs(\sofinsot)
=\lim_{\sofinsot\Subset\sosot}\sdfs(\sofinsot)
\text{ for all~\(\sosot\subseteq\finitesetsofthings\).}
\end{equation}
When the closure operator~\(\initialclosure\) is finitary, the simplified conjunctive representation result in Corollary~\ref{cor:conjunctive:representation:finitely:coherent:finite} tells us that, essentially,
\begin{equation}\label{eq:continuity:second:step}
\cohsdfsclosure(\sosot)=\sdfs(\sosot)=\bigcap\cset{\sdfsify{\cohsdt}}{\cohsdt\in\event{\sosot}}
\text{ for all~\(\sosot\subseteq\finitesetsofthings\)},
\end{equation}
and therefore, combining Equations~\eqref{eq:continuity:first:step} and~\eqref{eq:continuity:second:step}, leads to the important conclusion that
\[
\sdfs(\sosot)
=\lim_{\sofinsot\Subset\sosot}\sdfs(\sofinsot)
\text{ for all~\(\sosot\subseteq\finitesetsofthings\)},
\]
which is a `continuity' result for the map~\(\sdfs(\bolleke)\).

It's this `continuity' that allows us to achieve a simpler filter representation for SDFSes~\(\cohsdfs\), now no longer in terms of the filters~\(\finfinfilterise(\cohsdfs)=\cset{\event{\sofinsot}}{\sofinsot\Subset\cohsdfs}\) on the bounded distributive lattice \(\structure{\finfinevents,\subseteq}\), but rather in terms of the principal filters~\(\cset{\event{\sosot}}{\sosot\subseteq\cohsdfs}\) on the completely distributive complete lattice \(\structure{\events,\subseteq}\), which are representationally simpler, as they are completely characterised by their smallest elements~\(\event{\cohsdfs}\); observe that~\(\event{\cohsdfs}\in\events\), but not necessarily~\(\event{\cohsdfs}\in\finfinevents\).
The `continuity' of the map~\(\sdfs\) then allows us to go \emph{directly} from these smallest elements \(\event{\cohsdfs}\) to the closures \(\cohsdfsclosure(\cohsdfs)=\sdfs(\cohsdfs)\), rather than indirectly via the limits of the values \(\sdfs(\sofinsot)\) for \(\sofinsot\Subset\cohsdfs\), which we would be forced to if we were to restrict ourselves to working only with the events~\(\event{\sofinsot}\), \(\sofinsot\Subset\cohsdfs\) in the set~\(\finfinevents\).

\section{Example: propositional logic}\label{sec:propositional:logic}
In the remainder of this paper, we'll illustrate the ideas in the previous sections by looking at two interesting and relevant special cases.

As a first and fairly straightforward example, we consider propositional logic under the standard axiomatisation; see Ref.~\cite[Section 11.11 onwards]{davey2002} for a more detailed account of the facts we're about to summarise below.
We'll assume the reader to be familiar with the basic set-up of this logic using well-formed formulas.

\subsection*{The basic setup}
In this context, the things~\(\thing\) in~\(\things\) are the well-formed formulas (wffs) in the logical language; `desirable' means `true' or `derivable'; the set~\(\uglythings\) contains all contradictions and the set~\(\beautifulthings\) all tautologies.
We'll denote and-ing by~`\(\wedge\)', or-ing by~`\(\vee\)' and negation by~`\(\neg\)'.
Moreover, `closed' means `deductively closed'; the closure operator \(\initialclosure\) represents the usual deductive closure under finitary conjunction and modus ponens, and the `coherent' sets~\(\cohsdt\) in~\(\cohsdts\) are the sets of wffs that are deductively closed and contain no contradictions; see the summary in Table~\ref{table:propositional:logic}.

\begin{table}[th]
\centering
\begin{tabular}{r|l}
{\bfseries abstract theory of things} & {\bfseries propositional logic}\\[.5ex]
thing & wff\\
desirable thing & true, derivable wff\\
\(\beautifulthings\) & all tautologies\\
\(\uglythings\) & all contradictions\\
closed & deductively closed\\
consistent SDT & logically consistent set of wffs\\
coherent SDT & logically consistent and deductively closed set of wffs\\[1ex]
\end{tabular}
\caption{Propositional logic as a special case of desirable sets of things}\label{table:propositional:logic}
\end{table}

To see how distributive---in this case even Boolean---lattices and the corresponding filters get to play a role in this context, it will be useful to consider the Lindenbaum algebra associated with this propositional logic.
We use the notations `\(\thing[1]\entails\thing[2]\)' for `\(\thing[2]\in\initialclosure(\set{\thing[1]})\)' and `\(\thing[1]\equiv\thing[2]\)' for `\(\thing[1]\entails\thing[2]\) and \(\thing[2]\entails\thing[1]\)'.
Then the equivalence relation~\(\equiv\) partitions the set~\(\things\) into classes of logically equivalent wffs, which we collect in the set
\begin{equation*}
\lindenbaum=\toclass{\things}\coloneqq\cset{\thingclass}{\thing\in\things},
\text{ where }
\thingclass\coloneqq\cset{\thing'\in\things}{\thing'\equiv\thing}.
\end{equation*}
This set~\(\lindenbaum\) can be ordered by the partial order~\(\leq\), defined by
\begin{equation*}
\thingclass[1]\leq\thingclass[2]
\ifandonlyif\thing[1]\entails\thing[2],
\text{ for all~\(\thing[1],\thing[2]\in\things\)},
\end{equation*}
which turns it into a Boolean lattice---or Boolean algebra, the so-called \emph{Lindenbaum algebra}---with meet~\(\meet\) and join~\(\join\) defined by
\begin{equation*}
\thingclass[1]\meet\thingclass[2]\coloneqq\toclass{(\thing[1]\wedge\thing[2])}
\text{ and }
\thingclass[1]\join\thingclass[2]\coloneqq\toclass{(\thing[1]\vee\thing[2])}
\text{ for all~\(\thing[1],\thing[2]\in\things\)},
\end{equation*}
and complement operator~\(\compl\) given by
\begin{equation*}
\compl(\thingclass)\coloneqq\toclass{(\neg\thing)}
\text{ for all~\(\thing\in\things\)}.
\end{equation*}
Its top collects all tautologies, and its bottom all contradictions.

To get to filters, it's important to recall---or observe---that the map~\(\toclass\bolleke\colon\things\to\lindenbaum\) connects the coherent sets of desirable wffs to the proper filters of the Lindenbaum algebra~\(\lindenbaum\): for any coherent set of desirable wffs~\(\cohsdt\in\cohsdts\), the corresponding set of equivalence classes
\begin{equation*}
\toclass{\cohsdt}\coloneqq\cset{\thingclass}{\thing\in\cohsdt}
\end{equation*}
is a proper filter on~\(\lindenbaum\), and conversely, for any proper filter~\(\filter\in\lindenbaumproperfilters\) on~\(\lindenbaum\), the corresponding set of wffs
\begin{equation*}
\cset{\thing\in\things}{\thingclass\in\filter}
\end{equation*}
is a coherent set of desirable wffs.
Moreover, the closed but inconsistent set of wffs~\(\things\) is mapped to the improper filter~\(\lindenbaum\).

\subsection*{Towards desirable sets}
Let's now bring desirable \emph{sets of} wffs to the forefront, in order to complete the picture of Table~\ref{table:propositional:logic}.
As the closure operator for propositional logic is \emph{finitary}---since it's based on the \emph{finitary} conjunction and modus ponens production rules---we'll focus on the finitary aspects of this type of coherence, and rely on the representation results of Section~\ref{sec:representation:finitary:finite}.

To allow ourselves to be inspired by interpretation, we recall that a \emph{finite} set of wffs~\(\finsot\in\finitesetsofthings\) is considered to be desirable if it contains at least one desirable (true) wff, or equivalently in this special case, if its disjunction
\begin{equation*}
\bigvee\finsot\coloneqq\thing[1]\vee\dots\vee\thing[n]
\text{ with }
\finsot=\set{\thing[1],\dots,\thing[n]}
\end{equation*}
is desirable (true).
This simple observation leads us to the following definition, for which we can prove a basic but very revealing proposition.
We let, for any SDFS~\(\cohsdfs\),
\begin{equation*}
\cohsdt(\cohsdfs)\coloneqq\cset[\Big]{\bigvee\finsot}{\finsot\in\cohsdfs}.
\end{equation*}

\begin{proposition}\label{prop:propositional:basic}
Consider any finitely coherent set of desirable finite sets of wffs~\(\cohsdfs\in\cohsdfss\), then the following statements hold:
\begin{enumerate}[label=\upshape(\roman*),leftmargin=*]
\item\label{it:propositional:basic:is:coherent} \(\cohsdt(\cohsdfs)\) is a coherent set of desirable wffs;
\item\label{it:propositional:basic:is:smallest} \(\cohsdt(\cohsdfs)\) is the smallest~\(\cohsdt\in\cohsdts\) such that \(\cohsdfs\subseteq\sdfsify{\cohsdt}\);
\item\label{it:propositional:basic:representation} \(\cohsdfs=\sdfsify{\cohsdt(\cohsdfs)}\).
\end{enumerate}
\end{proposition}

We conclude that in the special case of propositional logic, all finitely coherent sets of desirable finite sets of wffs are conjunctive, or in other words, that working with desirable sets of wffs does not lead to anything more interesting than simply working with desirable wffs.
The reason for this is, of course, that in the case of propositional logic, the language of desirable wffs is already powerful enough to also accommodate for or-ing desirability statements, besides the and-ing that's inherently possible in any language of desirable things.

\section{Example: coherent choice}\label{sec:choice:functions}
As a second and final example, we'll consider coherent choice functions, and coherent sets of desirable gamble sets.
Our discussion here will be based mostly on earlier work by some of us \cite{cooman2021:archimedean:choice,debock2019b,debock2018,debock2021:S-independence}, to which we refer for the full details, and for explicit proofs of the results mentioned below.

Consider a variable~\(X\), and suppose that the value it takes in a \emph{finite} set~\(\states\) is unknown.
Any map~\(h\colon\states\to\reals\) then corresponds to a real-valued uncertain reward~\(h(X)\), and is called a \emph{gamble} on~\(X\).
We'll typically assume this reward is expressed in units of some linear utility scale.
The set~\(\gbls\) of all such gambles~\(h\) constitutes a linear  space.

In a typical decision problem, You are uncertain about the value of~\(X\), and are asked to express Your preferences between several possible decisions or acts, where each such act has an associated uncertain reward, or gamble.

We consider the strict vector ordering~\(\gblgt\), defined by
\begin{equation*}
h\gblgt g
\ifandonlyif(\forall x\in\states)h(x)>g(x),
\end{equation*}
as a background ordering, reflecting those (minimal) strict preferences we want any subject who's uncertain about~\(X\) to always have, regardless of their beliefs about~\(X\).
We'll also consider the sets~\(\posgbls=\cset{h\in\gbls}{\gblgt0}\) and ~\(\neggbls=\cset{h\in\gbls}{h\gbllt0}=-\posgbls\).

We'll also be concerned with so-called \emph{linear previsions}~\(\prev\) on~\(\gbls\), i.e.~the expectation operators associated with the probability mass functions~\(p\in\simplex\), so
\begin{equation*}
\prev(h)=E_p(h)\coloneqq\sum_{x\in\states}p(x)h(x)
\text{ and }
\simplex\coloneqq\cset[\bigg]{p\colon\states\to\nonnegreals}{\sum_{x\in\states}p(x)=1}.
\end{equation*}
For more details about linear previsions, we point to Refs.~\cite{walley1991,augustin2013:itip,troffaes2013:lp}.

\subsection*{Choice and rejection functions}
In the general, abstract setting, the gambles~\(\opt\) in the linear space~\(\gbls\) are intended to represent possible options in a decision problem under uncertainty.

In a specific application, there will be a \emph{finite} subset~\(\finsot\Subset\gbls\) containing gambles that You have to---in some way or other---express preferences between.
Such a decision problem may be `solved' when You specify Your subset~\(\rf\group{\finsot}\subseteq\finsot\) of \emph{rejected}, or \emph{inadmissible}, gambles.
We may interpret rejecting a gamble~\(h\in\rf\group{\finsot}\) from~\(\finsot\) as `\(\finsot\) contains another gamble that You prefer to~\(h\)'.
The remaining options~\(\cf\group{\finsot}\coloneqq\finsot\setminus\rf\group{\finsot}\) are then Your \emph{admissible} or non-rejected gambles in~\(\finsot\).

Initially, You may be uninformed, which will be reflected by Your set~\(\rf\group{\finsot}\) of rejected options being small.
You may gather more information, and when You do so, You may be able to additionally identify options that You reject, yielding a larger set~\(\rf\group{\finsot}\) and hence a smaller set of still admissible options~\(\cf\group{\finsot}\).
It may happen that You are optimally informed, and in this case Your set~\(\cf\group{\finsot}\) of admissible options will be a singleton, or possibly a set of `best' options that are `equally good' in the sense that You are indifferent between them.
However, we'll not assume that this state of being optimally informed is always attainable: You needn't always be indifferent between the options in~\(\cf\group{\finsot}\).
Instead, these admissible options may be incomparable, and as such, the setup we're describing here can deal with partial preferences, and leads to so-called imprecise-probabilistic decision-making approaches.
This makes this approach to decision theory more general than the classical idea of maximising expected utility where all the admissible options in~\(\cf\group{\finsot}\) have the same highest expected utility, so You are indifferent between them.

These ideas formalise to all decision problems---all sets in~\(\finiteoptsets\)---as follows.
The function \(\rf\colon\finiteoptsets\to\finiteoptsets\colon\finsot\mapsto\rf\group{\finsot}\subseteq\finsot\) that maps any finite gamble set~\(\finsot\in\finiteoptsets\) to its subset of rejected options, is called a \emph{rejection function}, and the corresponding dual function \(\cf\colon\finiteoptsets\to\finiteoptsets\colon\finsot\mapsto\cf\group{\finsot}\subseteq\finsot\) that identifies the admissible options, is called a \emph{choice function}.
As~\(\cf\group{\finsot}=\finsot\setminus\rf\group{\finsot}\) for every~\(\finsot\) in~\(\finiteoptsets\), either function can be retrieved from the other, so they are both equivalent representations of the same information.
Choice functions in an imprecise-probabilistic decision-making context were first introduced by~\citet{Kadane2004}, who later \cite{seidenfeld2010} also advanced a representation result for what they called coherent choice functions in terms of sets of probabilities and a decision criterion called \emph{E-admissibility}, going back to Isaac Levi \cite{levi1980a}.
Some time after that, some of us \cite{debock2018,cooman2021:archimedean:choice} established, amongst other things, more general representations in more abstract contexts.
We intend to briefly report on this work below, and to show how it relates to the desirable sets of things framework.

Recall that the background ordering~\(\gblgt\) for the option space~\(\opts\) reflects those strict preferences between gambles that it is rational for You to have, regardless of any information or beliefs You might have about the decision problem at hand.
We'll assume that this idea is reflected by the following requirement on rejection functions: If~\(\opt\gblgt\altopt\) then~\(\altopt\in\rf\group{\set{\opt,\altopt}}\), for all~\(\opt\) and~\(\altopt\) in~\(\opts\).
This will imply, together with~\citeauthor{sen1977:social:choice:re-examination}'s~\cite{sen1977:social:choice:re-examination} Property~\(\alpha\),\footnote{\citeauthor{sen1977:social:choice:re-examination}'s~\cite{sen1977:social:choice:re-examination} Property~\(\alpha\) is the requirement that~\(\rf\group{\optset[1]}\subseteq\rf\group{\optset[2]}\), for all~\(\optset[1]\) and~\(\optset[2]\) in~\(\finiteoptsets\) such that~\(\optset[1]\subseteq\optset[2]\), which is an axiom for choice under uncertainty that is almost always assumed.} that~\(\altopt\) is inadmissible in a decision problem~\(\finsot\) as soon as~\(\finsot\) contains another option~\(\opt\gblgt\altopt\), or in other words, as soon as it is \emph{dominated} by some option in~\(\finsot\);
\begin{equation}\label{eq:rejection:functions:dominance}
\group{\forall\altopt\in\finsot}
\group[\Big]{\group[\big]{\group{\exists\opt\in\finsot}\opt\gblgt\altopt}\then\altopt\in\rf\group{\finsot}};
\text{ for all \(\finsot\in\finiteoptsets\)}.
\end{equation}
The linearity of the utility scale is reflected in the assumption that any so-called \emph{coherent} rejection function~\(\rf\) must at least satisfy
\begin{equation}\label{eq:rejection:functions:linearity}
\group{\forall\opt\in\finsot}
\group[\big]{\opt\in\rf\group{\finsot}\ifandonlyif0\in\rf\group{\finsot-\set{\opt}}};
\text{ for all~\(\finsot\) in~\(\finiteoptsets\)},
\end{equation}
where \(\finsot[1]+\finsot[2]\coloneqq\cset{\opt[1]+\opt[2]}{\opt[1]\in\finsot[1],\opt[2]\in\finsot[2]}\) and \(-\finsot\coloneqq\cset{-\opt}{\opt\in\finsot}\) for any~\(\finsot,\finsot[1],\finsot[2]\Subset\opts\).
As a precursor to the discussion below, we see that to infer whether an option~\(\opt\) is rejected from~\(\finsot\)---whether \(\opt\in\rf\group{\finsot}\)---it suffices to check whether \(0\in\rf\group{\finsot-\set{\opt}}\), in which case we'll call~\(\finsot-\set{\opt}\) a \emph{desirable finite gamble set}.

\subsection*{Binary choice and sets of coherent sets of desirable options}
Equation~\eqref{eq:rejection:functions:linearity} has a particularly interesting consequence for \emph{binary decision problems}, which are decision problems that focus on binary gamble sets~\(\finsot=\set{\opt,\altopt}\).
Indeed, it implies that
\begin{equation}
\altopt\in\rf\group{\set{\opt,\altopt}}
\ifandonlyif0\in\rf\group{\set{0,\opt-\altopt}};
\text{ for all~\(\opt\) and~\(\altopt\) in~\(\opts\)}.
\end{equation}
But \(0\in\rf\group{\set{0,\opt-\altopt}}\) means that You reject the status quo represented by~\(0\) from the gamble set \(\set{0,\opt-\altopt}\), which can also be interpreted to mean that You strictly prefer the gamble~\(\opt-\altopt\) to the status quo~\(0\); we'll then call \(\opt-\altopt\) a \emph{desirable gamble}.
Your \emph{set of desirable gambles}~\(\cohsdt\) is the subset of~\(\opts\) that contains those gambles that are desirable to You in the sense that You strictly prefer them to~\(0\).
Your set of desirable gambles~\(\cohsdt\) determines Your binary preferences, since
\[
\altopt\in\rf\group{\set{\opt,\altopt}}
\ifandonlyif0\in\rf\group{\set{0,\opt-\altopt}}
\ifandonlyif\opt-\altopt\in\cohsdt.
\]
We'll call a set of desirable gambles~\(\cohsdt\) \emph{coherent} \cite{cooman2021:archimedean:choice} (but see also Refs.~\cite{walley2000,couso2011:desirable,cooman2010,seidenfeld1995} for related definitions) when
\begin{enumerate}[label=\upshape OD\({}_{\arabic*}\).,ref=\upshape OD\({}_{\arabic*}\),leftmargin=*]
\item\label{axiom:set:of:desirable:options:consistency} \(0\notin\cohsdt\);
\item\label{axiom:set:of:desirable:options:background} \(\posopts\subseteq\cohsdt\);
\item\label{axiom:set:of:desirable:options:combination} if \(\opt,\altopt\in\cohsdt\) and~\((\lambda,\mu)>0\),\footnote{We'll use the notation~\((\lambda,\mu)>0\) to mean that~\(\lambda\geq0\) and~\(\mu\geq0\) but not both equal to zero.} then \(\lambda\opt+\mu\altopt\in\cohsdt\), for all~\(\opt,\altopt\in\opts\) and~\(\lambda,\mu\in\reals\).
\end{enumerate}
We see that if we identify gambles as special cases of the abstract things, where we let the set of things~\(\things\) be the set of gambles~\(\opts\), and desirable gambles with desirable things, then we have made a start with identifying the correspondences between things and gambles in Table~\ref{table:desirabe:options}.
Let's now work towards completing this table, beginning at the level of desirable things, where we still have to identify the sets~\(\uglythings\), \(\beautifulthings\) and the closure operator~\(\initialclosure\).

\begin{table}[th]
\centering
\begin{tabular}{r|l}
{\bfseries abstract theory of things} & {\bfseries desirable gamble sets} \\[.5ex]
set of all things~\(\things\) & linear space of all gambles~\(\opts\)\\
thing & gamble\\
desirable thing & desirable gamble\\
\(\uglythings\) & all gambles in~\(\nonposopts\)\\
\(\initialclosure\) & \(\posi(\,\bolleke\,\cup\posopts)\)\\
\(\beautifulthings\) & all gambles in~\(\posopts\)\\
closed & closed under the coherence axioms\\
consistent SDT & consistent set of desirable gambles\\
coherent SDT & coherent set of desirable gambles\\
desirable finite set of things & desirable finite gamble set\\
consistent SDFS & consistent set of desirable finite gamble sets\\
finitely coherent SDFS & finitely coherent set of desirable finite gamble sets\\[1ex]
\end{tabular}
\caption{Desirable finite gamble sets as a special case of desirable finite sets of things}\label{table:desirabe:options}
\end{table}

First of all, observe that a coherent~\(\cohsdt\) can't have anything in common with the set~\(\nonposopts\coloneqq\negopts\cup\set{0}\).
Indeed, assume towards contradiction that \(\cohsdt\) contains some gamble~\(\opt\in\nonposopts\), then necessarily~\(\opt\optlt0\) by~\ref{axiom:set:of:desirable:options:consistency}.
Since then \(-\opt\gblgt0\) because~\(\gblgt\) is a vector ordering, we infer from~\ref{axiom:set:of:desirable:options:background} that \(-\opt\in\cohsdt\), and therefore, by~\ref{axiom:set:of:desirable:options:combination}, that \(0=\opt-\opt\in\cohsdt\), contradicting~\ref{axiom:set:of:desirable:options:consistency}.
Hence, indeed, \(\cohsdt\cap\nonposopts=\emptyset\), indicating that the convex cone~\(\nonposopts\) plays the role of the set of forbidden things~\(\uglythings\).

To identify the closure operator governing the desirability of gambles, we observe that \ref{axiom:set:of:desirable:options:combination} makes sure that coherent sets of desirable gambles~\(\cohsdt\) are \emph{convex cones}: they satisfy \(\cohsdt=\posi\group{\cohsdt}\), where
\begin{equation*}\label{eq:posioperator}
\posi\group{\sot}\coloneqq\cset[\bigg]{\sum_{k=1}^n\lambda_k\opt[k]}{n\in\naturals,\lambda_k\in\posreals,\opt[k]\in\sot}
\text{ for any~\(\sot\subseteq\opts\)}
\end{equation*}
is the set of all positive linear combinations of elements of~\(\sot\), and therefore the smallest convex cone that includes~\(\sot\).
We see that the coherent sets of desirable gambles~\(\cohsdt\subseteq\opts\) are exactly the convex cones in~\(\opts\) that include~\(\posopts\) and have nothing in common with~\(\nonposopts\), which tells us that the map~\(\posi(\,\bolleke\,\cup\posopts)\) takes the role of the closure operator \(\initialclosure\), and implements the inference mechanism behind the desirability of gambles, also called \emph{natural extension}~\cite{cooman2010}.
Interestingly, it is clear from its definition that this closure operator is \emph{finitary}.
Also observe that \(\posi(\emptyset\cup\posopts)=\posi(\posopts)=\posopts\) plays the role of the set~\(\beautifulthings\).

To conclude, let's check that with these identifications, the desirability axioms~\ref{axiom:desirable:things:closure}--\ref{axiom:desirable:things:background:consistency} in Section~\ref{sec:desirable:things} are verified.
For~\ref{axiom:desirable:things:closure}, assume that all gambles in \(\sot\subseteq\opts\) are desirable to You.
Then we infer from~\ref{axiom:set:of:desirable:options:background} and \ref{axiom:set:of:desirable:options:combination} that any positive linear combination of elements of \(\sot\) and \(\posopts\) must also be desirable to You.
As these positive linear combinations are precisely the elements of the closure \(\posi(\sot\cup\posopts)\) of the set \(\sot\), we see that \ref{axiom:desirable:things:closure} is indeed satisfied.
We have already argued that \(\cohsdt\cap\nonposopts=\emptyset\) for Your set of desirable gambles~\(\cohsdt\), so \ref{axiom:desirable:things:forbidden} is satisfied as well.
And finally, for~\ref{axiom:desirable:things:background:consistency}, note that, indeed, \(\beautifulthings\cap\uglythings=\posopts\cap\nonposopts=\emptyset\).

\subsection*{Coherent sets of desirable finite gamble sets}
To continue filling out Table~\ref{table:desirabe:options}, we now lift the framework of sets of desirable gambles to \emph{sets of desirable finite gamble sets}, as was done in Refs.~\cite{debock2018,debock2019b,cooman2021:archimedean:choice,debock2023:things:arxiv}.\footnote{The terminology used there is slightly different, because our `finite gamble sets' are simply called `gamble sets' there.}
The underlying idea is that, rather than merely use gambles as elements that are potentially desirable, we now turn to finite \emph{gamble sets}, instead.
In doing so, we'll move from (strict) binary preferences between gambles to more general preferences that aren't necessarily binary.

We'll allow You to state for a \emph{finite} gamble set~\(\finsot\in\finiteoptsets\) that at least one of its elements is desirable to You, but without Your needing to specify which; we'll then say that \(\finsot\) is \emph{desirable} to You, and call \(\finsot\) a desirable finite gamble set.
Your set of desirable finite gamble sets~\(\cohsdfs\) may then contain singletons~\(\set{\opt}\), reflecting that You find~\(\opt\) desirable, but also finite gamble sets~\(\finsot\) that aren't singletons.
In fact, it's perfectly possible for \(\cohsdfs\) to contain no singletons, apart from the~\(\set{\opt}\) for \(\opt\in\posopts\) that result from the background ordering.
It would then contain no non-trivial binary preferences.

Generally speaking, we'll call a set of desirable finite gamble sets~\(\cohsdfs\subseteq\finiteoptsets\) \emph{finitely coherent} when
\begin{enumerate}[label=\upshape OF\({}_{\arabic*}\).,ref=\upshape OF\({}_{\arabic*}\),leftmargin=*]
\item\label{axiom:desirable:option:sets:consistency} \(\emptyset\notin\cohsdfs\);
\item\label{axiom:desirable:option:sets:increasing} if \(\finsot[1]\in\cohsdfs\) and \(\finsot[1]\Subset\finsot[2]\) then \(\finsot[2]\in\cohsdfs\), for all~\(\finsot[1],\finsot[2]\in\finiteoptsets\);
\item\label{axiom:desirable:option:sets:forbidden} if \(\finsot\in\cohsdfs\) then \(\finsot\setminus\nonposopts\in\cohsdfs\), for all~\(\finsot\in\finiteoptsets\);
\item\label{axiom:desirable:option:sets:background} \(\set{\opt[{+}]}\in\cohsdfs\) for all~\(\opt[{+}]\in\posopts\);
\item\label{axiom:desirable:option:sets:production} if, with \(n\in\naturals\), \(\finsot[1],\dots,\finsot[n]\in\cohsdfs\) then also \(\cset{\sum_{k=1}^n\lambda^k_{\opt[1],\dots,\opt[n]}\opt[k]}{\opt[k]\in\finsot[k],k=1,\dots,n}\in\cohsdfs\), with \(\lambda^k_{{\opt[1],\dots,\opt[n]}}\geq0\) and \(\sum_{k=1}^n\lambda^k_{\opt[1],\dots,\opt[n]}>0\).
\end{enumerate}
It can be argued (see, for instance, Ref.~\cite{debock2023:things:arxiv}) that these finite coherence axioms are equivalent to the coherence axioms proposed in Refs.~\cite{cooman2021:archimedean:choice,debock2018,debock2019b}.
It's not too difficult to see that these finite coherence requirements can be reinterpreted as, essentially, Axioms~\ref{axiom:desirable:finite:sets:consistency}--\ref{axiom:desirable:finite:sets:production}, after a proper identification of the relevant concepts here with those in the abstract treatment of SFDSes, as summarised in Table~\ref{table:desirabe:options}.
As a consequence, the inference mechanism for finitely coherent SDFSes expressed by the closure operator~\(\cohsdfsclosure\), and all concomitant machinery, can be applied to sets of desirable finite gamble sets.

\subsection*{Back to choice and rejection functions}
We can now easily relate sets of desirable finite gamble sets~\(\cohsdfs\) back to rejection functions~\(\rf\).
To do so, we'll consider any finite gamble set~\(\finsot\), and any gamble~\(\opt\in\finsot\) for which we want to find out whether it's being rejected from~\(\finsot\).
We'll follow Ref.~\cite{cooman2021:archimedean:choice} in introducing the corresponding finite gamble set~\(\finsot\ominus\opt\coloneqq\group{\finsot\setminus\set{\opt}}-\set{\opt}\), which then allows for an efficient connection, taking into account Equation~\eqref{eq:rejection:functions:linearity} and our interpretation of rejecting a gamble:
\begin{equation}\label{eq:translating:back}
\opt\in\rf\group{\finsot}
\ifandonlyif
0\in\rf\group{\finsot-\set{\opt}}
\ifandonlyif
\group{\exists\altopt\in\finsot\ominus\opt}\text{\(\altopt\) is preferred to~\(0\)}
\ifandonlyif
\finsot\ominus\opt\in\cohsdfs.
\end{equation}
This connection allows for a one-to-one correspondence between rejection functions~\(\rf\) and sets of desirable finite gamble sets~\(\cohsdfs\), allowing us to transport the finite coherence notions from the latter to the former.

\subsection*{What about filter representation?}
So, now that we know that working with desirable gambles and desirable finite gamble sets fits in the context of the present paper, we also know that there will be representations in terms of (principal) filters of events.

We've seen above that coherent sets of desirable finite gamble sets are special instances of the finitely coherent SDFSes that are derived from the finitary closure operator~\(\initialclosure=\posi(\,\bolleke\,\cup\posopts)\).
This puts us squarely in the context of Section~\ref{sec:representation:finitary:finite}, and of the representation results proved therein.
What underlies all of these results is the lattice of events \(\finfinevents=\cset{\event{\sofinsot}}{\sofinsot\Subset\finiteoptsets}\), where each \(\event{\sofinsot}\) is some subset of \(\cohsdts\), so some collection of coherent sets of desirable gambles, which can be interpreted as a set of possible identifications of the actual set of desirable gambles~\(\idealcohsdt\); see the discussion in Section~\ref{sec:towards:representation}.

Now, as discussed in great detail in Refs.~\cite{debock2019b,cooman2021:archimedean:choice}, it's possible to impose additional (rationality) requirements on sets of desirable finite gamble sets~\(\cohsdfs\), besides coherence, and it will be interesting to briefly hint at them here.
There is a representation result that guarantees that a set of desirable finite gamble sets~\(\cohsdfs\) satisfies certain specific \emph{Archimedeanity} and \emph{mixingness} conditions if and only if there's some non-empty set of probability mass functions, also called \emph{credal set}, \(\solp\subseteq\simplex\) such that
\begin{equation}\label{eq:representation:archimedean:mixing}
\cohsdfs=\bigcap\cset{\sdfsify{E_p}}{p\in\solp}
\text{ where }
\sdfsify{E_p}\coloneqq\cset{\finsot\in\finiteoptsets}{\group{\exists\opt\in\finsot}E_p(\opt)>0},
\end{equation}
and where the largest such representing credal set \(\solp\) is given by
\[
\solp(\cohsdfs)
\coloneqq\cset{p\in\simplex}{\group{\forall\finsot\in\cohsdfs}\group{\exists\opt\in\finsot}E_p(\opt)>0}.
\]
This leads to an interesting conclusion.
Rather than saying something about an actual model \(\idealcohsdt\in\cohsdts\), the desirability statements present in an Archimedean and mixing~\(\cohsdfs\) can be interpreted as propositional statements about an actual model~\(p_\true\) in a set of possible identifications~\(\simplex\), and the corresponding `event'~\(\solp(\cohsdfs)\) is then the set of all possible identifications of \(p_\true\) that remain after making the desirability statements in~\(\cohsdfs\).
In fact, this allows us to identify a representation in terms of principal filters of subsets of~\(\simplex\).
We therefore recover, as a special case, the filter representation results proved in the seminal work by Catrin Campbell\textendash Moore \cite{campbellmoore2021:probability:filters:isipta2021}.

Let us, to conclude this section, find out what the representation result~\eqref{eq:representation:archimedean:mixing} tells us about the choice function~\(\cf\) and the rejection function~\(\rf\) that are associated with an Archimedean and mixing set of desirable finite gamble sets~\(\cohsdfs\), where we use the correspondence established in Equation~\eqref{eq:translating:back}.
First of all, we find that for any~\(p\in\simplex\) and any~\(\finsot\in\finiteoptsets\),
\[
\finsot\ominus\opt\in\sdfsify{E_p}
\ifandonlyif\group{\exists\altopt\in\finsot\setminus\set{\opt}}E_p(\altopt-\opt)>0
\ifandonlyif\group{\exists\altopt\in\finsot}E_p(\altopt)>E_p(\opt),
\]
and therefore
\[
\opt\in\rf\group{\finsot}
\ifandonlyif\group{\forall p\in\solp(\cohsdfs)}\group{\exists\altopt\in\finsot}E_p(\altopt)>E_p(\opt),
\]
and
\begin{equation}\label{eq:e:admissibility}
\opt\in\cf\group{\finsot}
\ifandonlyif\group{\exists p\in\solp(\cohsdfs)}\group{\forall\altopt\in\finsot}E_p(\altopt)\leq E_p(\opt).
\end{equation}
This tells us that a gamble~\(\opt\) is admissible in a gamble set~\(\finsot\) if it is \emph{Bayes-admissible} for at least one mass function in the credal set~\(\solp(\cohsdfs)\), in the sense that it maximises expectation.
We therefore \emph{recover Levi's E-admissibility criterion} \cite{levi1980a} in decision making as a special consequence of our representation results, and in this sense, all representation results in this paper can be seen as generalisations of Levi's E-admissibility.

\section{Discussion and conclusions}\label{sec:conclusion}
What this paper studies, discusses and eventually solves, is (i) how to deal with disjunctive statements in deductive inference systems that by their very nature deal mainly with conjunction, and (ii) how to identify the event-and-filter type of inference mechanism that underlies all of these systems.

Laying bare the exact nature of the event-and-filter type conservative inference mechanism behind coherent SDSes has allowed us to prove powerful representation results for such coherent SDSes in terms of the simpler, conjunctive, models which are essentially coherent SDTs.
These resulting representations, in their simplest form (Theorem~\ref{thm:conjunctive:representation} and Corollary~\ref{cor:conjunctive:representation:finitely:coherent:finite}), are reminiscent of, and in fact formal generalisations of, decision making using Levi's E-admissibility Rule \cite{levi1980a}.

Indeed, E-admissibility can be recovered as a very special consequence, where the desirable things are desirable gambles, and where rather than mere coherence, stronger requirements of Arch\-imedeanity and mixingness are imposed on sets of desirable gamble sets.

In another interesting special case, where the desirable things are asserted propositions in propositional logic, the additional layer of working with asserted sets of propositions---as instances of desirable sets of things---doesn't add anything new: all coherent sets of desirable sets of things are conjunctive there.
This is, of course, not really surprising, as desirable \emph{sets} of things are introduced to deal with disjunctive statements, which are already present in the language of \emph{things themselves} as propositions in propositional logic.
It would be interesting to find out whether something similar also happens in other logical languages, besides classical propositional logic, for which the way to deal with disjunctive statements is also already present in the language itself.

The case where things are gambles, on the other hand, shows that in other inference contexts where disjunctive statements are not already part of the language of things, going from desirable things to desirable sets of things is indeed meaningful and useful.

A more detailed and comprehensive study of these and other special cases is the topic of current research.

\section*{Acknowledgments}
The authors are grateful to Catrin Campbell\textendash Moore, Kevin Blackwell and Jason Konek for quite a number of animated discussions about this topic.
Catrin and Jason's ideas about using filters of sets of probabilities provided the initial inspiration for some of the developments here.
Gert wishes to express his gratitude to Jason Konek for funding several short research stays, as well as a one-month sabbatical stay, at University of Bristol's Department of Philosophy, which allowed him to work more closely together with Arthur, and discuss results with Kevin, Catrin and Jason.
Gert's research on this paper was partly supported by a sabbatical grant from Ghent University, and from the FWO, reference number K801523N.
Work by Gert and Jasper on this paper was also partly supported by the Research Foundation -- Flanders (FWO), project number 3G028919.
Jasper's work was furthermore supported by his BOF Starting Grant “Rational decision making under uncertainty: a new paradigm based on choice functions”, number 01N04819.
Arthur Van Camp's research was funded by Jason Konek's ERC Starting Grant ``Epistemic Utility for Imprecise Probability'' under the European Union’s Horizon 2020 research and innovation programme (grant agreement no. 852677).

\ifbibincluded

\else
\bibliographystyle{plainnat}
\bibliography{general}

\begin{thebibliography}{31}
\providecommand{\natexlab}[1]{#1}
\providecommand{\url}[1]{\texttt{#1}}
\expandafter\ifx\csname urlstyle\endcsname\relax
  \providecommand{\doi}[1]{doi: #1}\else
  \providecommand{\doi}{doi: \begingroup \urlstyle{rm}\Url}\fi

\bibitem[Arntzenius et~al.(2004)Arntzenius, Elga, and Hawthorne]{arntzenius2004:satans:apple}
Frank Arntzenius, Adam Elga, and John Hawthorne.
\newblock Bayesianism, infinite decisions, and binding.
\newblock \emph{Mind}, 113:\penalty0 251--283, 2004.
\newblock \doi{10.1093/mind/113.450.251}.

\bibitem[Augustin et~al.(2014)Augustin, Coolen, De Cooman, and Troffaes]{augustin2013:itip}
Thomas Augustin, Frank P.~A. Coolen, Gert de Cooman, and Matthias C.~M.
  Troffaes, editors.
\newblock \emph{Introduction to Imprecise Probabilities}.
\newblock John Wiley \& Sons, 2014.

\bibitem[Campbell-Moore(2021)]{campbellmoore2021:probability:filters:isipta2021}
Catrin Campbell--Moore.
\newblock Probability filters as a model of belief; comparisons to the
  framework of desirable gambles.
\newblock \emph{Proceedings of Machine Learning Research}, 147:\penalty0
  42–--50, 2021.
\newblock Proceedings of the ISIPTA 2021.

\bibitem[Campbell-Moore(2024)]{campbellmoore2024:desirable:gamble:sets}
Catrin Campbell--Moore.
\newblock Results about sets of desirable gamble sets.
\newblock \emph{arXiv}, 2024.
\newblock \doi{48550/arXiv.2404.17924}.
\newblock URL \url{https://arxiv.org/abs/2404.17924}.

\bibitem[Couso and Moral(2011)]{couso2011:desirable}
In\'es Couso and Seraf\'{\i}n Moral.
\newblock Sets of desirable gambles: conditioning, representation, and precise
  probabilities.
\newblock \emph{International Journal of Approximate Reasoning}, 52\penalty0
  (7), 2011.
\newblock \doi{10.1016/j.ijar.2011.08.001}.

\bibitem[Davey and Priestley(2002)]{davey2002}
B.~A. Davey and H.~A. Priestley.
\newblock \emph{Introduction to Lattices and Order}.
\newblock Cambridge University Press, Cambridge, second edition, 2002.

\bibitem[De~Bock(2015)]{debock2015:phdthesis}
Jasper De~Bock.
\newblock \emph{Credal Networks under Epistemic Irrelevance: Theory and
  algorithms}.
\newblock PhD thesis, Ghent University, 2015.

\bibitem[De~Bock(2020)]{debock2020:axiomatic:archimedean}
Jasper De~Bock.
\newblock Archimedean choice functions: an axiomatic foundation for imprecise
  decision making.
\newblock \emph{arXiv}, 2020.
\newblock \doi{10.48550/ARXIV.2002.05196}.
\newblock URL \url{https://arxiv.org/abs/2002.05196}.

\bibitem[De~Bock(2023{\natexlab{a}})]{debock2023:things:arxiv}
Jasper De~Bock.
\newblock A theory of desirable things.
\newblock \emph{arXiv}, 2023{\natexlab{a}}.
\newblock \doi{10.48550/ARXIV.2302.07412}.
\newblock URL \url{https://arxiv.org/abs/2302.07412}.

\bibitem[De~Bock(2023{\natexlab{b}})]{debock2023:things:isipta}
Jasper De~Bock.
\newblock A theory of desirable things.
\newblock \emph{Proceedings of Machine Learning Research}, 215:\penalty0
  141--152, 2023{\natexlab{b}}.
\newblock Proceedings of the ISIPTA 2023.

\bibitem[De~Bock and De Cooman(2018)]{debock2018}
Jasper De~Bock and Gert de Cooman.
\newblock A desirability-based axiomatisation for coherent choice functions.
\newblock In S.~Destercke, T.~Denoeux, M.A. Gil, P.~Grzergorzewski, and
  O.~Hryniewicz, editors, \emph{Uncertainty Modelling in Data Science:
  Proceedings of SMPS 2018}, pages 46--53, 2018.

\bibitem[De~Bock and De Cooman(2019)]{debock2019b}
Jasper De~Bock and Gert de Cooman.
\newblock Interpreting, axiomatising and representing coherent choice functions
  in terms of desirability.
\newblock In Jasper~De Bock, Cassio~P. de~Campos, Gert de~Cooman, Erik
  Quaeghebeur, and Gregory Wheeler, editors, \emph{Proceedings of the Eleventh
  International Symposium on Imprecise Probabilities: Theories and
  Applications}, volume 103 of \emph{Proceedings of Machine Learning Research},
  pages 125--134, Thagaste, Ghent, Belgium, 3--6 July 2019.

\bibitem[De~Bock and De Cooman(2023)]{debock2021:S-independence}
Jasper De~Bock and Gert de Cooman.
\newblock On a notion of independence proposed by {T}eddy {S}eidenfeld.
\newblock In Thomas Augustin, Fabio~Gagliardi Cozman, and Gregory Wheeler,
  editors, \emph{Reflections on the Foundations of Probability and Statistics:
  Essays in Honor of Teddy Seidenfeld}, pages 243--284. Springer, 2023.

\bibitem[De Cooman(2022)]{cooman2021:archimedean:choice}
Gert de Cooman.
\newblock Coherent and Archimedean choice in general Banach spaces.
\newblock \emph{International Journal of Approximate Reasoning}, 140:\penalty0
  255--281, 2022.
\newblock \doi{10.1016/j.ijar.2021.09.005}.

\bibitem[De Cooman and Quaeghebeur(2012)]{cooman2010}
Gert de Cooman and Erik Quaeghebeur.
\newblock Exchangeability and sets of desirable gambles.
\newblock \emph{International Journal of Approximate Reasoning}, 53\penalty0
  (3):\penalty0 363--395, 2012.
\newblock Special issue in honour of Henry E.~Kyburg, Jr.

\bibitem[De Cooman et~al.(2015)De Cooman, De~Bock, and
  Diniz]{decooman2015:coherent:predictive:inference}
Gert de Cooman, Jasper De~Bock, and M\'arcio~Alves Diniz.
\newblock Coherent predictive inference under exchangeability with imprecise
  probabilities.
\newblock \emph{Artificial Intelligence Journal}, 52:\penalty0 1--95, 2015.

\bibitem[Kadane et~al.(2004)Kadane, Schervish, and Seidenfeld]{Kadane2004}
Joseph~B. Kadane, Mark~J. Schervish, and Teddy Seidenfeld.
\newblock A {Rubinesque} theory of decision.
\newblock \emph{Institute of Mathematical Statistics Lecture Notes-Monograph
  Series}, 45:\penalty0 45--55, 2004.
\newblock \doi{10.1214/lnms/1196285378}.
\newblock URL \url{http://www.jstor.org/stable/4356297}.

\bibitem[Levi(1980)]{levi1980a}
Isaac Levi.
\newblock \emph{The Enterprise of Knowledge}.
\newblock MIT Press, London, 1980.

\bibitem[Miranda and Zaffalon(2023)]{miranda2022:nonlinear:desirability}
Enrique Miranda and Marco Zaffalon.
\newblock Nonlinear desirability theory.
\newblock \emph{International Journal of Approximate Reasoning}, 154:\penalty0
  176--199, 2023.
\newblock \doi{10.1016/j.ijar.2022.12.015}.

\bibitem[Quaeghebeur(2014)]{quaeghebeur2012:itip}
Erik Quaeghebeur.
\newblock Introduction to imprecise probabilities.
\newblock chapter Desirability. John Wiley \& Sons, 2014.

\bibitem[Rubin(1987)]{Rubin1987}
Herman Rubin.
\newblock A weak system of axioms for ``rational" behavior and the
  nonseparability of utility from prior.
\newblock \emph{Statistics \& Risk Modeling}, 5\penalty0 (1-2):\penalty0
  47--58, 1987.
\newblock \doi{10.1524/strm.1987.5.12.47}.

\bibitem[Seidenfeld et~al.(1995)Seidenfeld, Schervish, and
  Kadane]{seidenfeld1995}
Teddy Seidenfeld, Mark~J. Schervish, and Jay~B. Kadane.
\newblock A representation of partially ordered preferences.
\newblock \emph{The Annals of Statistics}, 23:\penalty0 2168--2217, 1995.
\newblock Reprinted in \cite{seidenfeld1999}, pp.~69--129.

\bibitem[Seidenfeld et~al.(1999)Seidenfeld, Schervish, and
  Kadane]{seidenfeld1999}
Teddy Seidenfeld, Mark~J. Schervish, and Jay~B. Kadane.
\newblock \emph{Rethinking the Foundations of Statistics}.
\newblock Cambridge University Press, Cambridge, 1999.

\bibitem[Seidenfeld et~al.(2010)Seidenfeld, Schervish, and
  Kadane]{seidenfeld2010}
Teddy Seidenfeld, Mark~J. Schervish, and Joseph~B. Kadane.
\newblock Coherent choice functions under uncertainty.
\newblock \emph{Synthese}, 172\penalty0 (1):\penalty0 157--176, 2010.
\newblock \doi{10.1007/s11229-009-9470-7}.

\bibitem[Sen(1977)]{sen1977:social:choice:re-examination}
Amartya Sen.
\newblock Social choice theory: A re-examination.
\newblock \emph{Econometrica}, 45\penalty0 (1):\penalty0 53--89, 1977.
\newblock ISSN 00129682, 14680262.
\newblock URL \url{http://www.jstor.org/stable/1913287}.

\bibitem[Sen(1986)]{sen1986:social:choice:theory}
Amartya Sen.
\newblock Social choice theory.
\newblock In \emph{Handbook of Mathematical Economics}, volume~3, chapter~22,
  pages 1073--1181. 1986.

\bibitem[Troffaes and De Cooman(2014)]{troffaes2013:lp}
Matthias C.~M. Troffaes and Gert de Cooman.
\newblock \emph{Lower Previsions}.
\newblock Wiley, 2014.

\bibitem[Van~Camp(2018)]{2017vancamp:phdthesis}
Arthur Van~Camp.
\newblock \emph{Choice functions as a tool to model uncertainty}.
\newblock PhD thesis, Universiteit Gent, 2018.

\bibitem[Walley(1991)]{walley1991}
Peter Walley.
\newblock \emph{Statistical Reasoning with Imprecise Probabilities}.
\newblock Chapman and Hall, London, 1991.

\bibitem[Walley(2000)]{walley2000}
Peter Walley.
\newblock Towards a unified theory of imprecise probability.
\newblock \emph{International Journal of Approximate Reasoning}, 24:\penalty0
  125--148, 2000.

\bibitem[Willard(1970)]{willard1970}
Stephen Willard.
\newblock \emph{General Topology}.
\newblock Addison-Wesley, 1970.

\end{thebibliography}
\fi

\appendix
\section{Proofs and technical results}\label{sec:proofs}

\subsection{Proofs of results in Section~\ref{sec:desirability}}\label{sec:proofs:desirability}

\begin{proof}[Proof of Equation~\eqref{eq:closed:under:intersections}]
Let \(\cohsdt\coloneqq\bigcap_{i\in I}\cohsdt[i]\) for ease of notation.
Since, clearly, \(\cohsdt\cap\uglythings=\emptyset\), we need only prove that \(\cohsdt\in\toppedcohsdts\), or in other words, that \(\initialclosure(\cohsdt)\subseteq\cohsdt\) [use \ref{axiom:closure:super}].
For any~\(i\in I\), we have that \(\cohsdt\subseteq\cohsdt[i]\), and therefore also that \(\initialclosure(\cohsdt)\subseteq\initialclosure(\cohsdt[i])=\cohsdt[i]\), where the inclusion follows from~\ref{axiom:closure:increasing}, and the equality from the assumption that \(\cohsdt[i]\in\cohsdts\).
Hence, indeed, \(\initialclosure(\cohsdt)\subseteq\bigcap_{i\in I}\cohsdt[i]=\cohsdt\).
\end{proof}

\begin{proof}[Proof of Equation~\eqref{eq:cohsdtclosure:consistent:sodts}]
To show that \(\initialclosure(\sot)\subseteq\bigcap\cset{\cohsdt\in\cohsdts}{\sot\subseteq\cohsdt}\), consider any~\(\cohsdt\) in~\(\cohsdts\) for which \(\sot\subseteq\cohsdt\).
Then also \(\initialclosure(\sot)\subseteq\initialclosure(\cohsdt)=\cohsdt\), where the inclusion follows from~\ref{axiom:closure:increasing} and the equality holds because \(\cohsdt\in\cohsdts\).
Hence, indeed, \(\initialclosure(\sot)\subseteq\bigcap\cset{\cohsdt\in\cohsdts}{\sot\subseteq\cohsdt}\).

For the converse inclusion, note that \(\sot\subseteq\initialclosure(\sot)\) by~\ref{axiom:closure:super}.
But since also \(\initialclosure(\sot)\in\cohsdts\) by the assumed consistency of~\(\sot\), we find that \(\initialclosure(\sot)\in\cset{\cohsdt\in\cohsdts}{\sot\subseteq\cohsdt}\), whence, indeed, \(\bigcap\cset{\cohsdt\in\cohsdts}{\sot\subseteq\cohsdt}\subseteq\initialclosure(\sot)\).
\end{proof}

\begin{proof}[Proof that \(\beautifulthings\) is the smallest closed SDT]
To prove that \(\initialclosure(\emptyset)\subseteq\bigcap\toppedcohsdts\), we consider that for any~\(\cohsdt\in\toppedcohsdts\), \(\initialclosure(\emptyset)\subseteq\initialclosure(\cohsdt)=\cohsdt\), where the inclusion follows from~\ref{axiom:closure:increasing}.
Hence, indeed, \(\initialclosure(\emptyset)\subseteq\bigcap\toppedcohsdts\).

For the converse inclusion, simply observe that \(\initialclosure(\emptyset)\in\toppedcohsdts\), by~\ref{axiom:closure:idempotent}.
\end{proof}

\begin{proof}[Proof that \ref{axiom:desirable:sets:background}\&\ref{axiom:desirable:sets:production} are equivalent to~\ref{axiom:desirable:sets:production:with:background}]
It clearly suffices to show that~\ref{axiom:desirable:sets:production:with:background} restricted to the case that \(\sosot=\emptyset\) is equivalent to~\ref{axiom:desirable:sets:background}.
But there, \(\selections\) contains only the empty map~\(\selection[\emptyset]\) with \(\selection[\emptyset](\sosot)=\emptyset\), and therefore \(\initialclosure(\selection[\emptyset](\sosot))=\initialclosure(\emptyset)=\beautifulthings\), so \ref{axiom:desirable:sets:production:with:background} now simply requires that for any~\(\thing[{+}]\in\beautifulthings\), also \(\set{\thing[{+}]}\in\cohsds\), which is, of course, exactly the purport of~\ref{axiom:desirable:sets:background}.
\end{proof}

\begin{proof}[Proof of Proposition~\ref{prop:the:smallest:cohsds}]
If we let, for ease of notation, \(\cohsds[+]\coloneqq\cset{\sot\in\setsofthings}{\sot\cap\beautifulthings\neq\emptyset}\), then it follows from~\ref{axiom:desirable:sets:background} and~\ref{axiom:desirable:sets:increasing} that it clearly suffices to prove that \(\cohsds[+]\) is a coherent SDS.

It's a matter of immediate verification that it trivially satisfies~\ref{axiom:desirable:sets:consistency}, \ref{axiom:desirable:sets:increasing} and~\ref{axiom:desirable:sets:background}.

For~\ref{axiom:desirable:sets:forbidden}, consider that for any~\(\sot\in\setsofthings\), \((\sot\setminus\uglythings)\cap\beautifulthings=\sot\cap(\beautifulthings\setminus\uglythings)=\sot\cap\beautifulthings\), where the last equality follows from the fact that \(\beautifulthings\cap\uglythings=\emptyset\) [use~\ref{axiom:desirable:things:background:consistency}].

For~\ref{axiom:desirable:sets:production}, first observe that if \(\beautifulthings=\emptyset\), and therefore also \(\cohsds[+]=\emptyset\), then the condition is satisfied vacuously.
Otherwise, consider any non-empty~\(\sosot\subseteq\cohsds[+]\) and observe that there always is some~\(\selection[+]\in\selections[\sosot]\) such that \(\selection[+](\sot)\in\sot\cap\beautifulthings\) for all~\(\sot\in\sosot\).
Then \(\selection[+](\sosot)\subseteq\beautifulthings\), and therefore also \(\initialclosure(\selection[+](\sosot))=\beautifulthings\) [to see why this equality holds, observe that \(\emptyset\subseteq\selection[+](\sosot)\subseteq\beautifulthings\) implies \(\initialclosure(\emptyset)\subseteq\initialclosure(\selection[+](\sosot))\subseteq\initialclosure(\beautifulthings)\) by~\ref{axiom:closure:increasing}, and recall that \(\beautifulthings=\initialclosure(\emptyset)\) by definition, and that then \(\initialclosure(\beautifulthings)=\initialclosure(\initialclosure(\emptyset))=\initialclosure(\emptyset)=\beautifulthings\) by \ref{axiom:closure:idempotent}].
If we now choose any~\(\thing[\selection]\in\initialclosure(\selection(\sosot))\) for all~\(\selection\in\selections[\sosot]\), then we must prove that \(\sot[+]\in\cohsds[+]\) for the corresponding \(\sot[+]\coloneqq\cset{\thing[\selection]}{\selection\in\selections[\sosot]}\), or in other words, that \(\sot[+]\cap\beautifulthings\neq\emptyset\).
But this is obvious, since on the one hand \(\thing[{\selection[+]}]\in\sot[+]\), and on the other hand \(\thing[{\selection[+]}]\in\initialclosure(\selection[+](\sosot))=\beautifulthings\).
\end{proof}

\begin{proof}[Proof of Proposition~\ref{prop:cohsds:to:cohsdt}]
Before we really begin, we can already observe that if \(\cohsds\) is (finitely) coherent, it follows from~\ref{axiom:desirable:sets:background} that \(\beautifulthings\subseteq\sdtify{\cohsds}\).

For the proof of the first statement, consider any~\(\sot\) in~\(\sdsify{\sdtify{\cohsds}}\).
Then \(\sot\cap\sdtify{\cohsds}\neq\emptyset\), or, in other words, \(\thing[o]\in\sdtify{\cohsds}\) for some~\(\thing[o]\) in~\(\sot\).
But then \(\set{\thing[o]}\in\cohsds\), and therefore \ref{axiom:desirable:sets:increasing} guarantees that also \(\sot\in\cohsds\), since \(\set{\thing[o]}\subseteq\sot\).

To prove the coherence of~\(\sdtify{\cohsds}\), it is [by~\ref{axiom:closure:super}] enough to show that \(\initialclosure(\sdtify{\cohsds})\subseteq\sdtify{\cohsds}\) and that \(\sdtify{\cohsds}\cap\uglythings=\emptyset\).

To show that \(\sdtify{\cohsds}\cap\uglythings=\emptyset\), we only need to rely on \(\cohsds\) satisfying the conditions~\ref{axiom:desirable:sets:consistency}--\ref{axiom:desirable:sets:forbidden}, and therefore not on the finitary or infinitary aspect of the coherence of~\(\cohsds\), nor on whether the closure operator~\(\initialclosure\) is finitary or not.
Indeed, assume towards contradiction that \(\uglythings\cap\sdtify{\cohsds}\neq\emptyset\), so there's some~\(\thing[-]\in\uglythings\cap\sdtify{\cohsds}\).
This implies that \(\set{\thing[-]}\in\cohsds\), and therefore, [use~\ref{axiom:desirable:sets:forbidden} with \(\set{\thing[-]}\setminus\uglythings=\emptyset\)], that \(\emptyset\in\cohsds\), contradicting~\ref{axiom:desirable:sets:consistency}.

To prove that \(\initialclosure(\sdtify{\cohsds})\subseteq\sdtify{\cohsds}\), we'll need a different approach in \ref{it:cohsds:to:cohsdt:finitely:coherent} the finitary and in \ref{it:cohsds:to:cohsdt:coherent} the infinitary case.
But we may in both cases assume without loss of generality that \(\initialclosure(\sdtify{\cohsds})\neq\emptyset\).

\underline{\ref{it:cohsds:to:cohsdt:finitely:coherent}}.
Consider any~\(\thing[o]\) in~\(\initialclosure\group{\sdtify{\cohsds}}\).
Since we already know that \(\beautifulthings\subseteq\sdtify{\cohsds}\), we may assume without loss of generality that \(\thing[o]\notin\beautifulthings\).
Since \(\initialclosure\) is finitary, this implies that there's some non-empty~\(\sot\Subset\sdtify{\cohsds}\) for which \(\thing[o]\in\initialclosure\group{\sot}\) [\(\sot\) must be non-empty because, otherwise, we'd have that \(\thing[o]\in\initialclosure(\sot)=\initialclosure(\emptyset)=\beautifulthings\), contradicting our assumption that \(\thing[o]\notin\beautifulthings\)].
If we now let \(\sosot\coloneqq\cset{\set{\thing}}{\thing\in\sot}\) then clearly \(\emptyset\neq\sosot\Subset\cohsds\).
Every element of~\(\sosot\) is a singleton, so \(\selections\) consists of the single map~\(\selection[o]\colon\sosot\to\things\) defined by \(\selection[o](\set{\thing})\coloneqq\thing\) for all~\(\thing\in\sot\).
But then of course \(\selection[o]\group{\sosot}=\cset{\selection[o]\group{\set{\thing}}}{\thing\in\sot}=\sot\) and therefore \(\initialclosure\group{\selection[o]\group{\sosot}}=\initialclosure\group{\sot}\).
If we now consider the special choice~\(\thing[{\selection[o]}]\coloneqq\thing[o]\in\initialclosure\group{\selection[o]\group{\sosot}}\), then \ref{axiom:desirable:sets:finitary:production} guarantees that \(\set{\thing[o]}=\cset{\thing[\selection]}{\selection\in\selections}\in\cohsds\), so \(\thing[o]\in\sdtify{\cohsds}\).
Hence, indeed, \(\initialclosure\group{\sdtify{\cohsds}}\subseteq\sdtify{\cohsds}\).

\underline{\ref{it:cohsds:to:cohsdt:coherent}}.
In this case, the proof is fairly similar, but let's spell it out anyway, in order to verify explicitly that we don't need the finitary character of the closure operator~\(\initialclosure\) here.
Again, consider any~\(\thing[o]\) in~\(\initialclosure\group{\sdtify{\cohsds}}\).
Since we already know that \(\beautifulthings\subseteq\sdtify{\cohsds}\), we may assume without loss of generality that \(\thing[o]\notin\beautifulthings\).
But this implies that \(\sdtify{\cohsds}\) must be non-empty.
If we now let \(\sosot\coloneqq\cset{\set{\thing}}{\thing\in\sdtify{\cohsds}}\) then clearly \(\emptyset\neq\sosot\subseteq\cohsds\).
Every element of~\(\sosot\) is a singleton, so \(\selections\) consists of the single map~\(\selection[o]\colon\sosot\to\things\) defined by \(\selection[o](\set{\thing})\coloneqq\thing\) for all~\(\thing\in\sdtify{\cohsds}\).
But then of course \(\selection[o]\group{\sosot}=\cset{\selection[o]\group{\set{\thing}}}{\thing\in\sdtify{\cohsds}}=\sdtify{\cohsds}\) and therefore \(\initialclosure\group{\selection[o]\group{\sosot}}=\initialclosure\group{\sdtify{\cohsds}}\).
If we now consider the special choice~\(\thing[{\selection[o]}]\coloneqq\thing[o]\in\initialclosure\group{\selection[o]\group{\sosot}}\), then \ref{axiom:desirable:sets:production} guarantees that \(\set{\thing[o]}=\cset{\thing[\selection]}{\selection\in\selections}\in\cohsds\), so \(\thing[o]\in\sdtify{\cohsds}\).
Hence, indeed, \(\initialclosure\group{\sdtify{\cohsds}}\subseteq\sdtify{\cohsds}\).
\end{proof}

\begin{proof}[Proof of Proposition~\ref{prop:cohsdt:to:cohsds}]
First off, consider the following chain of equivalences, valid for any~\(\thing\in\things\):
\begin{equation*}
\thing\in\sdtify{\sdsify{\cohsdt}}
\ifandonlyif\set{\thing}\in\sdsify{\cohsdt}
\ifandonlyif\set{\thing}\cap\cohsdt\neq\emptyset
\ifandonlyif\thing\in\cohsdt,
\end{equation*}
which implies that, generally speaking, \(\sdtify{\sdsify{\cohsdt}}=\cohsdt\).

\underline{\ref{it:cohsdt:to:cohsds:coherent:sdt}\(\then\)\ref{it:cohsdt:to:cohsds:coherent:sds}}.
Assume that~\(\cohsdt\) is coherent, then we must show that \ref{axiom:desirable:sets:consistency}--\ref{axiom:desirable:sets:production} are satisfied for~\(\sdsify{\cohsdt}\).
We concentrate on the proofs for~\ref{axiom:desirable:sets:forbidden}--\ref{axiom:desirable:sets:production}, since the proofs for~\ref{axiom:desirable:sets:consistency} and~\ref{axiom:desirable:sets:increasing} are trivial.

\underline{\ref{axiom:desirable:sets:forbidden}}.
Assume that \(\sot\in\sdsify{\cohsdt}\), meaning that \(\sot\cap\cohsdt\neq\emptyset\).
Since \(\cohsdt\cap\uglythings=\emptyset\) by Equation~\eqref{eq:definition:cohsdts}, this implies that
\begin{equation*}
\emptyset
\subset\sot\cap\cohsdt
=\sot\cap\group[\big]{(\cohsdt\cap\uglythings)\cup(\cohsdt\setminus\uglythings)}
=\sot\cap(\cohsdt\setminus\uglythings)
=(\sot\setminus\uglythings)\cap\cohsdt,
\end{equation*}
so, indeed, \(\sot\setminus\uglythings\in\sdsify{\cohsdt}\).

\underline{\ref{axiom:desirable:sets:background}}.
Since the coherence of \(\cohsdt\) implies that \(\beautifulthings\subseteq\cohsdt\), it follows trivially that, indeed, \(\set{\thing[{+}]}\in\sdsify{\cohsdt}\) for all~\(\thing[{+}]\in\beautifulthings\).

\underline{\ref{axiom:desirable:sets:production}}.
Consider any non-empty~\(\sosot\subseteq\sdsify{\cohsdt}\), meaning that \(\sot\cap\cohsdt\neq\emptyset\) for all~\(\sot\in\sosot\).
This implies that we can always find some~\(\selection[o]\in\selections\) for which \(\selection[o](\sot)\in\cohsdt\) for all~\(\sot\in\sosot\), and therefore \(\selection[o](\sosot)\subseteq\cohsdt\), so we also find that \(\initialclosure(\selection[o](\sosot))\subseteq\cohsdt\).
This guarantees that always \(\thing[{\selection[o]}]\in\cohsdt\), and therefore that, indeed, \(\cset{\thing[\selection]}{\selection\in\selections}\in\sdsify{\cohsdt}\).

\underline{\ref{it:cohsdt:to:cohsds:coherent:sds}\(\then\)\ref{it:cohsdt:to:cohsds:coherent:sdt}}.
Assume that \(\sdsify{\cohsdt}\) is coherent, then Proposition~\ref{prop:cohsds:to:cohsdt}\ref{it:cohsds:to:cohsdt:coherent} guarantees that the SDT~\(\sdtify{\sdsify{\cohsdt}}\) is coherent.
Since we proved above that \(\sdtify{\sdsify{\cohsdt}}=\cohsdt\), this implies that \(\cohsdt\) is coherent.

\underline{\ref{it:cohsdt:to:cohsds:coherent:sds}\(\then\)\ref{it:cohsdt:to:cohsds:finitely:coherent:sds}}.
Trivial, because \ref{axiom:desirable:sets:production} implies \ref{axiom:desirable:sets:finitary:production}.

\underline{\ref{it:cohsdt:to:cohsds:finitely:coherent:sds}\(\then\)\ref{it:cohsdt:to:cohsds:coherent:sdt}}.
Assume that \(\sdsify{\cohsdt}\) is finitely coherent and that the closure operator~\(\initialclosure\) is finitary, then Proposition~\ref{prop:cohsds:to:cohsdt}\ref{it:cohsds:to:cohsdt:finitely:coherent} guarantees that the SDT~\(\sdtify{\sdsify{\cohsdt}}\) is coherent.
Since we proved above that \(\sdtify{\sdsify{\cohsdt}}=\cohsdt\), this implies that \(\cohsdt\) is coherent.
\end{proof}

\begin{proof}[Proof of Proposition~\ref{prop:conjunctive:and:coherent:sds}]
The `if' statements follow from Propositions\ref{prop:cohsdt:to:cohsds} and~\ref{prop:sds:to:sdt:and:back:ordering}, so we turn to the proof of the `only if' statements.
First off, consider the following chain of equivalences, valid for any (finitely) coherent SDS~\(\cohsds\):
\begin{align*}
\text{\(\cohsds\) is conjunctive}
&\ifandonlyif\group{\forall\sot\in\cohsds}\group{\exists\thing\in\sot}\set{\thing}\in\cohsds
\ifandonlyif\group{\forall\sot\in\cohsds}\group{\exists\thing\in\sot}\thing\in\sdtify{\cohsds}\\
&\ifandonlyif\group{\forall\sot\in\cohsds}\sot\cap\sdtify{\cohsds}\neq\emptyset
\ifandonlyif\group{\forall\sot\in\cohsds}\sot\in\sdsify{\sdtify{\cohsds}}
\ifandonlyif\cohsds\subseteq\sdsify{\sdtify{\cohsds}}
\ifandonlyif\cohsds=\sdsify{\sdtify{\cohsds}},
\end{align*}
where the last equivalence follows from Proposition~\ref{prop:cohsds:to:cohsdt}, which tells us that \(\sdsify{\sdtify{\cohsds}}\subseteq\cohsds\), because \(\cohsds\) was assumed to be (finitely) coherent.
It remains to argue that the SDT~\(\sdtify{\cohsds}\) is coherent.
In the case of~\ref{it:conjunctive:and:coherent:sds} this follows immediately from Proposition~\ref{prop:cohsds:to:cohsdt}\ref{it:cohsds:to:cohsdt:coherent}; and in the case of~\ref{it:conjunctive:and:finitely:coherent:sds} this follows immediately from Proposition~\ref{prop:cohsds:to:cohsdt}\ref{it:cohsds:to:cohsdt:finitely:coherent}.

Finally, if \(\cohsds=\sdsify{\cohsdt}\) then necessarily \(\sdtify{\cohsds}=\sdtify{\sdsify{\cohsdt}}=\cohsdt\), where the last equality follows from Proposition~\ref{prop:cohsdt:to:cohsds}.
\end{proof}

\subsection{Proofs of results in Section~\ref{sec:desirability:sdfs}}\label{sec:proofs:desirability:sdfs}

\begin{proof}[Proof of Proposition~\ref{prop:the:smallest:cohsdfs}]
The proof is very similar to the proof of Proposition~\ref{prop:the:smallest:cohsds}, but we nevertheless include it here for the sake of completeness.
If we let, for ease of notation, \(\cohsdfs[+]\coloneqq\cset{\sot\in\finitesetsofthings}{\finsot\cap\beautifulthings\neq\emptyset}\), then it follows from~\ref{axiom:desirable:finite:sets:background} and~\ref{axiom:desirable:finite:sets:increasing} that it clearly suffices to prove that \(\cohsdfs[+]\) is a finitely coherent SDFS.

It's a matter of immediate verification that it trivially satisfies~\ref{axiom:desirable:finite:sets:consistency}, \ref{axiom:desirable:finite:sets:increasing} and~\ref{axiom:desirable:finite:sets:background}.

For~\ref{axiom:desirable:finite:sets:forbidden}, consider that for any~\(\finsot\in\finitesetsofthings\), \((\finsot\setminus\uglythings)\cap\beautifulthings=\finsot\cap(\beautifulthings\setminus\uglythings)=\finsot\cap\beautifulthings\), where the last equality follows from the fact that \(\beautifulthings\cap\uglythings=\emptyset\) [use~\ref{axiom:desirable:things:background:consistency}].

For~\ref{axiom:desirable:finite:sets:production}, first observe that if \(\beautifulthings=\emptyset\), and therefore also \(\cohsdfs[+]=\emptyset\), then the condition is satisfied vacuously.
Otherwise, consider any non-empty~\(\sofinsot\Subset\cohsdfs[+]\) and observe that there always is some~\(\selection[+]\in\finselections\) such that \(\selection[+](\finsot)\in\finsot\cap\beautifulthings\) for all~\(\finsot\in\sofinsot\).
Then \(\selection[+](\sofinsot)\subseteq\beautifulthings\), and therefore also \(\initialclosure(\selection[+](\sofinsot))=\beautifulthings\) [to see why this equality holds, observe that \(\emptyset\subseteq\selection[+](\sofinsot)\subseteq\beautifulthings\) implies that \(\initialclosure(\emptyset)\subseteq\initialclosure(\selection[+](\sofinsot))\subseteq\initialclosure(\beautifulthings)\) by~\ref{axiom:closure:increasing}, and recall that \(\beautifulthings=\initialclosure(\emptyset)\) by definition, and that then \(\initialclosure(\beautifulthings)=\initialclosure(\initialclosure(\emptyset))=\initialclosure(\emptyset)=\beautifulthings\) by \ref{axiom:closure:idempotent}].
If we now choose any~\(\thing[\selection]\in\initialclosure(\selection(\sofinsot))\) for all~\(\selection\in\selections[\sofinsot]\), then we must prove that \(\finsot[+]\in\cohsdfs[+]\) for the corresponding \(\finsot[+]\coloneqq\cset{\thing[\selection]}{\selection\in\finselections}\), or in other words, that \(\finsot[+]\cap\beautifulthings\neq\emptyset\).
But this is obvious, since both \(\thing[{\selection[+]}]\in\finsot[+]\), and \(\thing[{\selection[+]}]\in\initialclosure(\selection[+](\sofinsot))=\beautifulthings\).
\end{proof}

\begin{proof}[Proof of Proposition~\ref{prop:cohsdfs:to:cohsdt}]
The proof is very similar to the proof of Proposition~\ref{prop:cohsds:to:cohsdt}, but we nevertheless include it here for the sake of completeness.

For the proof of the first statement, consider any~\(\finsot\) in~\(\sdfsify{\sdtify{\cohsdfs}}\).
Then \(\finsot\cap\sdtify{\cohsdfs}\neq\emptyset\), or, in other words, \(\thing[o]\in\sdtify{\cohsdfs}\) for some~\(\thing[o]\) in~\(\finsot\).
But then \(\set{\thing[o]}\in\cohsdfs\), and therefore \ref{axiom:desirable:finite:sets:increasing} guarantees that also \(\finsot\in\cohsdfs\), since \(\set{\thing[o]}\subseteq\finsot\).

We now turn to the last statement.
First of all, observe that if \(\cohsdfs\) is finitely coherent, then it follows from~\ref{axiom:desirable:finite:sets:background} that \(\beautifulthings\subseteq\sdtify{\cohsdfs}\).

To prove the coherence of~\(\sdtify{\cohsdfs}\), it is [by~\ref{axiom:closure:super}] enough to show that \(\initialclosure(\sdtify{\cohsdfs})\subseteq\sdtify{\cohsdfs}\) and that \(\sdtify{\cohsdfs}\cap\uglythings=\emptyset\).

To show that \(\sdtify{\cohsdfs}\cap\uglythings=\emptyset\), assume towards contradiction that \(\uglythings\cap\sdtify{\cohsdfs}\neq\emptyset\), so there's some~\(\thing[-]\in\uglythings\cap\sdtify{\cohsdfs}\).
This implies that \(\set{\thing[-]}\in\cohsdfs\), and therefore [use~\ref{axiom:desirable:finite:sets:forbidden} with \(\set{\thing[-]}\setminus\uglythings=\emptyset\)] that \(\emptyset\in\cohsdfs\), contradicting~\ref{axiom:desirable:finite:sets:consistency}.

To prove that \(\initialclosure(\sdtify{\cohsdfs})\subseteq\sdtify{\cohsdfs}\), we may clearly assume that \(\initialclosure(\sdtify{\cohsdfs})\neq\emptyset\).
Consider, then, any~\(\thing[o]\) in~\(\initialclosure\group{\sdtify{\cohsdfs}}\).
Since we already know that \(\beautifulthings\subseteq\sdtify{\cohsdfs}\), we may assume without loss of generality that \(\thing[o]\notin\beautifulthings\).
Since \(\initialclosure\) is finitary, this implies that there's some non-empty~\(\finsot\Subset\sdtify{\cohsdfs}\) for which \(\thing[o]\in\initialclosure\group{\finsot}\) [\(\finsot\) must be non-empty because, otherwise, we'd have that \(\thing[o]\in\initialclosure(\finsot)=\initialclosure(\emptyset)=\beautifulthings\), contradicting our assumption that \(\thing[o]\notin\beautifulthings\)].
If we now let \(\sofinsot\coloneqq\cset{\set{\thing}}{\thing\in\finsot}\) then clearly \(\emptyset\neq\sofinsot\Subset\cohsdfs\).
Every element of~\(\sofinsot\) is a singleton, so \(\finselections\) consists of the single map~\(\selection[o]\colon\sofinsot\to\things\) defined by \(\selection[o](\set{\thing})\coloneqq\thing\) for all~\(\thing\in\finsot\).
But then of course \(\selection[o]\group{\sofinsot}=\cset{\selection[o]\group{\set{\thing}}}{\thing\in\finsot}=\finsot\) and therefore \(\initialclosure\group{\selection[o]\group{\sofinsot}}=\initialclosure\group{\finsot}\).
If we now consider the special choice~\(\thing[{\selection[o]}]\coloneqq\thing[o]\in\initialclosure\group{\selection[o]\group{\sofinsot}}\), then \ref{axiom:desirable:finite:sets:production} guarantees that \(\set{\thing[o]}=\cset{\thing[\selection]}{\selection\in\finselections}\in\cohsdfs\), so \(\thing[o]\in\sdtify{\cohsdfs}\).
Hence, indeed, \(\initialclosure\group{\sdtify{\cohsdfs}}\subseteq\sdtify{\cohsdfs}\).
\end{proof}

\begin{proof}[Proof of Proposition~\ref{prop:cohsdt:to:cohsdfs}]
The proof is very similar to the proof of Proposition~\ref{prop:cohsdt:to:cohsds}, but we nevertheless include it here for the sake of completeness.
For the first statement, consider the following chain of equivalences, valid for any~\(\thing\in\things\):
\begin{equation*}
\thing\in\sdtify{\sdfsify{\cohsdt}}
\ifandonlyif\set{\thing}\in\sdfsify{\cohsdt}
\ifandonlyif\set{\thing}\cap\cohsdt\neq\emptyset
\ifandonlyif\thing\in\cohsdt,
\end{equation*}
which implies that, generally speaking, indeed \(\sdtify{\sdfsify{\cohsdt}}=\cohsdt\).

Next, assume that~\(\cohsdt\) is coherent, then we must show that \ref{axiom:desirable:finite:sets:consistency}--\ref{axiom:desirable:finite:sets:production} are satisfied for~\(\sdfsify{\cohsdt}\).
We concentrate on the proofs for~\ref{axiom:desirable:finite:sets:forbidden}--\ref{axiom:desirable:finite:sets:production}, since the proofs for~\ref{axiom:desirable:finite:sets:consistency} and~\ref{axiom:desirable:finite:sets:increasing} are trivial.

\underline{\ref{axiom:desirable:finite:sets:forbidden}}.
Assume that \(\finsot\in\sdfsify{\cohsdt}\), meaning that \(\finsot\cap\cohsdt\neq\emptyset\).
Since \(\cohsdt\cap\uglythings=\emptyset\) by Equation~\eqref{eq:definition:cohsdts}, this implies that
\begin{equation*}
\emptyset
\subset\finsot\cap\cohsdt
=\finsot\cap\group[\big]{(\cohsdt\cap\uglythings)\cup(\cohsdt\setminus\uglythings)}
=\finsot\cap(\cohsdt\setminus\uglythings)
=(\finsot\setminus\uglythings)\cap\cohsdt,
\end{equation*}
so, indeed, \(\finsot\setminus\uglythings\in\sdfsify{\cohsdt}\).

\underline{\ref{axiom:desirable:finite:sets:background}}.
Since the coherence of \(\cohsdt\) implies that \(\beautifulthings\subseteq\cohsdt\), it follows trivially that, indeed, \(\set{\thing[{+}]}\in\sdfsify{\cohsdt}\) for all~\(\thing[{+}]\in\beautifulthings\).

\underline{\ref{axiom:desirable:finite:sets:production}}.
Consider any non-empty~\(\sofinsot\Subset\sdfsify{\cohsdt}\), meaning that \(\finsot\cap\cohsdt\neq\emptyset\) for all~\(\finsot\in\sofinsot\).
This implies that we can always find some~\(\selection[o]\in\finselections\) for which \(\selection[o](\finsot)\in\cohsdt\) for all~\(\finsot\in\sofinsot\), and therefore \(\selection[o](\sofinsot)\subseteq\cohsdt\), so we also find that \(\initialclosure(\selection[o](\sofinsot))\subseteq\cohsdt\).
This guarantees that always \(\thing[{\selection[o]}]\in\cohsdt\), and therefore that, indeed, \(\cset{\thing[\selection]}{\selection\in\finselections}\in\sdfsify{\cohsdt}\).

Finally, assume that the closure operator~\(\initialclosure\) is finitary and that \(\sdfsify{\cohsdt}\) is finitely coherent, then Proposition~\ref{prop:cohsdfs:to:cohsdt} guarantees that the SDT~\(\sdtify{\sdfsify{\cohsdt}}\) is coherent.
Since we proved above that \(\sdtify{\sdfsify{\cohsdt}}=\cohsdt\), this implies that \(\cohsdt\) is coherent.
\end{proof}

\begin{proof}[Proof of Proposition~\ref{prop:conjunctive:and:coherent:sdfs}]
The proof is very similar to the proof of Proposition~\ref{prop:conjunctive:and:coherent:sds}, but we nevertheless include it here for the sake of completeness.
For the `if' statement, assume that there's some coherent~\(\cohsdt\in\cohsdts\) such that \(\cohsdfs=\cohsdfs[\cohsdt]\).
Then it follows from Proposition~\ref{prop:cohsdt:to:cohsdfs} that \(\cohsdfs=\cohsdfs[\cohsdt]\) is finitely coherent.
To show that it is also conjunctive, consider any \(\finsot\in\cohsdfs\), so \(\finsot\cap\cohsdt\neq\emptyset\).
Fix any \(\thing\in\finsot\cap\cohsdt\), then on the one hand \(\thing\in\finsot\) and on the other hand \(\thing\in\cohsdt\) and therefore, indeed, \(\set{\thing}\in\cohsdfs[\cohsdt]\).

Next, we turn to the proof of the `only if' statement.
First off, consider the following chain of equivalences, valid for any finitely coherent SDFS~\(\cohsdfs\):
\begin{align*}
\text{\(\cohsdfs\) is conjunctive}
&\ifandonlyif\group{\forall\finsot\in\cohsdfs}\group{\exists\thing\in\finsot}\set{\thing}\in\cohsdfs
\ifandonlyif\group{\forall\finsot\in\cohsdfs}\group{\exists\thing\in\finsot}\thing\in\sdtify{\cohsdfs}\\
&\ifandonlyif\group{\forall\finsot\in\cohsdfs}\finsot\cap\sdtify{\cohsdfs}\neq\emptyset
\ifandonlyif\group{\forall\finsot\in\cohsdfs}\finsot\in\sdfsify{\sdtify{\cohsdfs}}
\ifandonlyif\cohsdfs\subseteq\sdfsify{\sdtify{\cohsdfs}}
\ifandonlyif\cohsdfs=\sdfsify{\sdtify{\cohsdfs}},
\end{align*}
where the last equivalence follows from Proposition~\ref{prop:cohsdfs:to:cohsdt}, which tells us that \(\sdfsify{\sdtify{\cohsdfs}}\subseteq\cohsdfs\), because \(\cohsdfs\) was assumed to be finitely coherent.
Proposition~\ref{prop:cohsdfs:to:cohsdt} furthermore guarantees that the SDT~\(\sdtify{\cohsdfs}\) is coherent.

Finally, if \(\cohsdfs=\sdfsify{\cohsdt}\) then necessarily \(\sdtify{\cohsdfs}=\sdtify{\sdfsify{\cohsdt}}=\cohsdt\), where the last equality follows from Proposition~\ref{prop:cohsdt:to:cohsdfs}.
\end{proof}

\subsection{Proofs of results in Section~\ref{sec:representation:lattices}}\label{sec:proofs:representation:lattices}

\begin{proof}[Proof of Proposition~\ref{prop:the:two:systems}]
For a start, consider any~\(\cohsdt\in\cohsdts\), and verify the equivalences in the following chain:
\begin{align*}
\cohsdt\in\event{\sosot}
\ifandonlyif\cohsdt\in\smash{\bigcap_{\sot\in\sosot}}\cohsdts[\sot]
&\ifandonlyif\group{\forall\sot\in\sosot}\sot\cap\cohsdt\neq\emptyset
\ifandonlyif\group{\exists\selection\in\selections}\group{\forall\sot\in\sosot}\selection(\sot)\in\cohsdt\notag\\
&\ifandonlyif\group{\exists\selection\in\selections}\selection(\sosot)\subseteq\cohsdt
\ifandonlyif\group{\exists\selection\in\selections}\initialclosure(\selection(\sosot))\subseteq\cohsdt\\
&\ifandonlyif\cohsdt\in\cohsdtUpset{\basicevent{\sosot}},
\end{align*}
where the penultimate equivalence follows from~\ref{axiom:closure:super}--\ref{axiom:closure:idempotent} and the assumption that \(\cohsdt\in\cohsdts\).
The proof of the last equivalence goes as follows.
The converse implication is immediate, looking at the definition of~\(\basicevent{\sosot}\) in Equation~\eqref{eq:definition:production:event}.
For the direct implication, assume that these is some~\(\selection\in\selections\) such that \(\initialclosure(\selection(\sosot))\subseteq\cohsdt\).
Since \(\cohsdt\in\cohsdts\), we see that then necessarily also \(\initialclosure(\selection(\sosot))\in\cohsdts\), and therefore, by Equation~\eqref{eq:definition:production:event}, that \(\initialclosure(\selection(\sosot))\in\basicevent{\sosot}\).
\end{proof}

\begin{proof}[Proof of Proposition~\ref{prop:filter:consistency}]
We give a proof for the first statement involving coherence of SDSes.
The proof for the second and third statement involving finite coherence of SDSes and finite coherence of SDFSes are completely analogous.

Assume, towards contradiction, that there's some~\(\sosot\subseteq\cohsds\) such that \(\event{\sosot}=\emptyset\), then Proposition~\ref{prop:the:two:systems} tells us that \(\basicevent{\sosot}=\emptyset\), so we get from Equation~\eqref{eq:definition:production:event} that \(\initialclosure\group{\selection(\sosot)}\cap\uglythings\neq\emptyset\) for all~\(\selection\in\selections\).
We can therefore select a \(\thing[\selection]\in\initialclosure\group{\selection(\sosot)}\cap\uglythings\) for all~\(\selection\in\selections\), and then it follows from~\ref{axiom:desirable:sets:production:with:background} that \(\sot\coloneqq\cset{\thing[\selection]}{\selection\in\selections}\in\cohsds\).
But since, by construction, \(\sot\subseteq\uglythings\), it follows from~\ref{axiom:desirable:sets:forbidden} that \(\emptyset=\sot\setminus\uglythings\in\cohsds\), contradicting~\ref{axiom:desirable:sets:consistency}.
\end{proof}

\begin{proof}[Proof of Proposition~\ref{prop:it:does:not:matter:which}]
We restrict ourselves to the proof of the version for finitely coherent SDSes, as the proofs of the other versions are similar, and somewhat less involved.
We only give a proof for the first statement~\ref{it:it:does:not:matter:which:inclusion}, as the second statement~\ref{it:it:does:not:matter:which:equality} is a trivial consequence of the first.
Since always \(\emptyset\Subset\cohsds\), it's clear that we may assume without loss of generality that \(\sosot[2]\neq\emptyset\).
It follows from Proposition~\ref{prop:filter:consistency} that we may also assume without loss of generality that~\(\eventwithindex{1}\neq\emptyset\).
Using Proposition~\ref{prop:the:two:systems}, we can then infer from \(\eventwithindex{1}\subseteq\eventwithindex{2}\) that
\begin{equation}\label{eq:it:does:not:matter:which}
\emptyset
\neq\cohsdtUpset{\cset{\initialclosure(\selection(\sosot[1]))}{\selection\in\selections[{\sosot[1]}]}\cap\cohsdts}
\subseteq\smash{\bigcap_{\sot\in\sosot[2]}}\cohsdts[\sot].
\end{equation}
Now, \(\selections[{\sosot[1]}]\) can be partitioned into two disjoint sets
\begin{equation*}
\indexedselections[{\sosot[1]}]{\ast}
\coloneqq\cset[\big]{\selection\in\selections[{\sosot[1]}]}{\initialclosure(\selection(\sosot[1]))\in\cohsdts}
\text{ and }
\indexedselections[{\sosot[1]}]{\ast\ast}
\coloneqq\cset[\big]{\selection\in\selections[{\sosot[1]}]}{\initialclosure(\selection(\sosot[1]))\notin\cohsdts}.
\end{equation*}
We infer from Equation~\eqref{eq:it:does:not:matter:which} that \(\indexedselections[{\sosot[1]}]{\ast}\neq\emptyset\) and that \(\initialclosure(\selection(\sosot[1]))\cap\sot\neq\emptyset\) for all~\(\sot\in\sosot[2]\) and all~\(\selection\in\indexedselections[{\sosot[1]}]{\ast}\), and it follows from the definition of~\(\indexedselections[{\sosot[1]}]{\ast\ast}\) that \(\initialclosure(\selection(\sosot[1]))\cap\uglythings\neq\emptyset\) for all~\(\selection\in\indexedselections[{\sosot[1]}]{\ast\ast}\).
If we now fix any~\(\sot\in\sosot[2]\), then this tells us that we can always choose a \(\thing[\selection]\in\initialclosure(\selection(\sosot[1]))\) such that also \(\thing[\selection]\in\sot\), for all~\(\selection\in\indexedselections[{\sosot[1]}]{\ast}\).
Similarly and at the same time, we can always choose a \(\thing[\selection]\in\initialclosure(\selection(\sosot[1]))\) such that also \(\thing[\selection]\in\uglythings\), for all~\(\selection\in\indexedselections[{\sosot[1]}]{\ast\ast}\).
Now let \(\sot'\coloneqq\cset{\thing[\selection]}{\selection\in\indexedselections[{\sosot[1]}]{\ast}}\subseteq\sot\) and \(\sot''\coloneqq\cset{\thing[\selection]}{\selection\in\indexedselections[{\sosot[1]}]{\ast\ast}}\subseteq\uglythings\), then, since we assumed that \(\sosot[1]\Subset\cohsds\), we infer from~\ref{axiom:desirable:sets:finitary:production} that \(\sot'\cup\sot''=\cset{\thing[\selection]}{\selection\in\selections[{\sosot[1]}]}\in\cohsds\), and from~\ref{axiom:desirable:sets:forbidden} that then also \(\sot'\setminus\uglythings\in\cohsds\).
But since, by construction, \(\sot'\setminus\uglythings\subseteq\sot'\subseteq\sot\), we infer from~\ref{axiom:desirable:sets:increasing} that also \(\sot\in\cohsds\).
Hence, indeed, \(\sosot[2]\Subset\cohsds\).
\end{proof}

\begin{proof}[Proof of Theorem~\ref{thm:basic:sets:constitute:a:bounded:lattice}]
We begin with a proof for the first statement.
Consider any non-empty family of~\(\sosot[i]\subseteq\setsofthings\), \(i\in I\) then
\begin{equation}\label{eq:intersection:of:possibles}
\bigcap_{i\in I}\eventwithindex{i}
=\bigcap_{i\in I}\group[\bigg]{\bigcap_{\sot[i]\in\sosot[i]}\cohsdts[{\sot[i]}]}
=\bigcap_{\sot\in\sosot}\cohsdts[\sot]
=\event{\sosot}
=\event[\bigg]{\bigcup_{i\in I}\sosot[i]}
\in\events,
\end{equation}
where we let \(\sosot\coloneqq\bigcup_{i\in I}\sosot[i]\).
That \(\bigcap_{i\in I}\eventwithindex{i}\) still belongs to~\(\events\) allows us to conclude, via a standard result in order theory \cite[Corollary~2.29]{davey2002}, that intersection is indeed the infimum in the poset~\(\structure{\events,\subseteq}\).
Moreover, taking into account the complete distributivity of unions over intersections, we also find that
\begin{equation}\label{eq:union:of:possibles}
\bigcup_{i\in I}\eventwithindex{i}
=\bigcup_{i\in I}\group[\bigg]{\bigcap_{\sot[i]\in\sosot[i]}\cohsdts[{\sot[i]}]}
=\bigcap_{\psi\in\Psi}\bigcup_{i\in I}\cohsdts[\psi(i)]
=\bigcap_{\psi\in\Psi}\cohsdts[{\sot[\psi]}]
=\event{\sosot'}\in\events,
\end{equation}
where \(\Psi\) is the set of all choice maps~\(\psi\colon I\to\bigcup_{i\in I}\sosot[i]\) with \(\psi(i)\in\sosot[i]\) for all~\(i\in I\), where we let \(\sot[\psi]\coloneqq\bigcup_{i\in I}\psi(i)\in\setsofthings\) and \(\sosot'\coloneqq\cset{\sot[\psi]}{\psi\in\Psi}\subseteq\setsofthings\), and where the penultimate equality holds because
\begin{equation*}
\bigcup_{i\in I}\cohsdts[\psi(i)]
=\bigcup_{i\in I}\cset{\cohsdt\in\cohsdts}{\cohsdt\cap\psi(i)\neq\emptyset}
=\cset{\cohsdt\in\cohsdts}{\cohsdt\cap\sot[\psi]\neq\emptyset}
=\cohsdts[{\sot[\psi]}].
\end{equation*}
That \(\bigcup_{i\in I}\eventwithindex{i}\) still belongs to~\(\events\) allows us to conclude, again via the same standard result in order theory \cite[Corollary~2.29]{davey2002}, that union is indeed the supremum in the poset~\(\structure{\events,\subseteq}\).
This structure is therefore a complete lattice (of sets), since the non-empty index set~\(I\) was arbitrary in our argumentation.
It follows at once that this complete lattice is completely distributive, because any complete lattice of sets is \cite[Theorem~10.29]{davey2002}.

Finally, we infer from Equation~\eqref{eq:definition:event} that \(\event{\set{\emptyset}}=\cohsdts[\emptyset]=\emptyset\), so \(\emptyset\in\events\) and therefore \(\emptyset\) is the bottom of this structure.
Similarly, we can infer from Equation~\eqref{eq:definition:event} that \(\event{\emptyset}=\cohsdts\), so \(\cohsdts\in\events\) and therefore \(\cohsdts\) is the top of this structure.

The proof for the second and third statements use the same ideas, but in a simpler, finitary guise, as the index set~\(I\) and the subsets~\(\sosot[i]\) [or~\(\sofinsot[i]\)] in the argumentation above are now kept finite, and only finite distributivity is required.
\end{proof}

\subsection{Proofs of results in Section~\ref{sec:filters}}\label{sec:proofs:filters}

\begin{proof}[Proof of the statement involving Equation~\eqref{eq:directed:downwards}]
We must prove that the non-empty subset~\(\filterbase\) of \(\lattice\setminus\set{\latticebottom}\) is a filter base for some proper filter if and only if it satisfies the condition~\eqref{eq:directed:downwards}.

For sufficiency, assume that \(\filterbase\) satisfies~\eqref{eq:directed:downwards}  and let \(\filter\coloneqq\posetUpset{\filterbase}\).
Then it's enough show that \(\filter\) is a proper filter.
Since any image of \(\posetUpset{\bolleke}\) satisfies~\ref{axiom:lattice:filters:increasing}, it suffices to focus on~\ref{axiom:lattice:filters:intersections}.
Consider any \(a_1,a_2\in\filters\), so there are \(b_1,b_2\in\filterbase\) such that \(b_1\leq a_1\) and \(b_2\leq a_2\).
Then on the one hand \(b_1\meet b_2\leq a_1\meet a_2\), and on the other hand it follows from the assumption that there is some~\(b\in\filterbase\) such that \(b\leq b_1\meet b_2\) and therefore \(b\leq a_1\meet a_2\), which guarantees that, indeed, \(a_1\meet a_2\in\filter\).
Moreover, since \(\latticebottom\notin\filterbase\), it follows at once that also \(\latticebottom\notin\filter\).

For necessity, assume that there is some proper filter~\(\filter\) for which \(\filter=\posetUpset{\filterbase}\).
Consider any \(b_1,b_2\in\filterbase\) then \(b_1\meet b_2\in\filter\) by~\ref{axiom:lattice:filters:intersections}, so there is some~\(b\in\filterbase\) such that \(b\leq b_1\meet b_2\).
\end{proof}

\subsection{Proofs of results in Section~\ref{sec:representation:finitary}}\label{sec:proofs:representation:finitary}

\begin{proof}[Proof of Theorem~\ref{thm:the:finitary:representation:theorem}]
For statement~\ref{it:the:finitary:representation:theorem:if:cohsds:then:filter}, assume that \(\cohsds\) is a finitely coherent SDS.
We must show that \(\finfilterise(\cohsds)\) is a proper filter.
To see that \(\finfilterise(\cohsds)\) is non-empty, observe that always \(\emptyset\Subset\cohsds\), and therefore \(\cohsdts=\event{\emptyset}\in\finfilterise(\cohsds)\).
To show that \(\finfilterise(\cohsds)\) is increasing [satisfies~\ref{axiom:lattice:filters:increasing}], consider any~\(\sosot[1],\sosot[2]\Subset\setsofthings\) such that \(\eventwithindex{1}\in\finfilterise(\cohsds)\) and \(\eventwithindex{1}\subseteq\eventwithindex{2}\).
Then Proposition~\ref{prop:it:does:not:matter:which}\ref{it:it:does:not:matter:which:equality} implies that~\(\sosot[1]\Subset\cohsds\), and Proposition~\ref{prop:it:does:not:matter:which}\ref{it:it:does:not:matter:which:inclusion} then guarantees that also~\(\sosot[2]\Subset\cohsds\), whence, indeed, also \(\eventwithindex{2}\in\finfilterise(\cohsds)\).
To show that \(\finfilterise(\cohsds)\) is closed under finite intersections [satisfies~\ref{axiom:lattice:filters:intersections}], consider arbitrary~\(\sosot[1],\sosot[2]\Subset\cohsds\).
Then we infer from Equation~\eqref{eq:intersection:of:possibles} that \(\eventwithindex{1}\cap\eventwithindex{2}=\event{{\sosot[1]\cup\sosot[2]}}\), and since \(\sosot[1]\cup\sosot[2]\Subset\cohsds\), this tells us that, indeed, \(\eventwithindex{1}\cap\eventwithindex{2}\in\finfilterise(\cohsds)\).
To show that the filter~\(\finfilterise(\cohsds)\) is proper, simply observe that it follows from Proposition~\ref{prop:filter:consistency} that \(\event{\sosot}\neq\emptyset\) for all~\(\sosot\Subset\cohsds\), which ensures that \(\emptyset\notin\finfilterise(\cohsds)\).
This tells us that, indeed, \(\finfilterise(\cohsds)\neq\finevents\) [recall that \(\emptyset\in\finevents\) by Theorem~\ref{thm:basic:sets:constitute:a:bounded:lattice}].

For statement~\ref{it:the:finitary:representation:theorem:if:filter:then:cohsds}, assume that \(\filter\) is a proper filter.
We check that the relevant conditions are satisfied for \(\cohsds\) to be a finitely coherent SDS.

\underline{\ref{axiom:desirable:sets:consistency}}.
Since \(\cohsdts[\emptyset]=\emptyset\) and \(\filter\) is proper, we find that \(\cohsdts[\emptyset]\notin\filter\) [again, recall that \(\emptyset\in\finevents\) by Theorem~\ref{thm:basic:sets:constitute:a:bounded:lattice}].
Hence, indeed, \(\emptyset\notin\findesirify(\filter)\).

\underline{\ref{axiom:desirable:sets:increasing}}.
Consider any~\(\sot[1],\sot[2]\in\setsofthings\) with~\(\sot[1]\subseteq\sot[2]\).
Assume that \(\sot[1]\in\findesirify(\filter)\), so \(\cohsdts[{\sot[1]}]\in\filter\).
That \(\sot[1]\subseteq\sot[2]\) implies, via Lemma~\ref{lem:the:cohsdtify:operator}\ref{it:the:cohsdtify:operator:increasing}, that \(\cohsdts[{\sot[1]}]\subseteq\cohsdts[{\sot[2]}]\).
This allows us to infer that also \(\cohsdts[{\sot[2]}]\in\filter\), using~\ref{axiom:lattice:filters:increasing}.
Hence, indeed, \(\sot[2]\in\findesirify(\filter)\).

\underline{\ref{axiom:desirable:sets:forbidden}}.
Consider any~\(\sot\in\setsofthings\), and assume that~\(\sot\in\findesirify(\filter)\).
Infer from Lemma~\ref{lem:the:cohsdtify:operator}\ref{it:the:cohsdtify:operator:forbidden} that \(\cohsdts[\sot]=\cohsdts[{\sot\setminus\uglythings}]\) and therefore also \(\cohsdts[\sot]\in\filter\ifandonlyif\cohsdts[{\sot\setminus\uglythings}]\in\filter\).
Hence, indeed, \(\sot\setminus\uglythings\in\findesirify(\filter)\).

\underline{\ref{axiom:desirable:sets:background}}.
We may assume without loss of generality that \(\beautifulthings\neq\emptyset\), because otherwise this requirement is trivially satisfied. Consider then any~\(\thing[{+}]\in\beautifulthings\), then we must show that \(\set{\thing[{+}]}\in\findesirify(\filter)\), or in other words that \(\cohsdts[{\set{\thing[{+}]}}]\in\filter\).
Now simply observe that
\begin{equation*}
\cohsdts[{\set{\thing[{+}]}}]
=\cset{\cohsdt\in\cohsdts}{\set{\thing[{+}]}\cap\cohsdt\neq\emptyset}
=\cohsdts
\in\filter,
\end{equation*}
where the last equality holds because the coherence of~\(\cohsdt\) implies that \(\beautifulthings\subseteq\cohsdt\), and the final statement holds because the smallest filter~\(\fineventfiltersbottom=\set{\cohsdts}\) is included in all filters.

\underline{\ref{axiom:desirable:sets:finitary:production}}.
Consider any non-empty~\(\sosot\Subset\findesirify(\filter)\) and any choice~\(\thing[\selection]\in\initialclosure(\selection(\sosot))\) for all the~\(\selection\in\selections\), then we must prove that \(\sot[o]\coloneqq\cset{\thing[\selection]}{\selection\in\selections}\in\findesirify(\filter)\), or in other words that \(\cohsdts[{\sot[o]}]\in\filter\).
It follows from the assumptions that \(\cohsdts[\sot]\in\filter\) for all~\(\sot\in\sosot\), and therefore also, by~\ref{axiom:lattice:filters:intersections} and the finiteness of~\(\sosot\), that \(\event{\sosot}=\bigcap_{\sot\in\sosot}\cohsdts[\sot]\in\filter\), so it's enough to prove that \(\event{\sosot}\subseteq\cohsdts[{\sot[o]}]\), because then also, by~\ref{axiom:lattice:filters:increasing}, \(\cohsdts[{\sot[o]}]\in\filter\), as required.
Consider, to this end, any~\(\cohsdt\in\event{\sosot}\), then it follows from Proposition~\ref{prop:the:two:systems} and Equation~\eqref{eq:definition:production:event} that there's some selection map~\(\selection[o]\in\selections\) such that \(\initialclosure(\selection[o](\sosot))\subseteq\cohsdt\).
This guarantees that \(\thing[{\selection[o]}]\in\cohsdt\), and therefore \(\cohsdt\cap\sot[o]\neq\emptyset\).
Hence, indeed, \(\cohsdt\in\cohsdts[{\sot[o]}]\).

For statement~\ref{it:the:finitary:representation:theorem:cohsds:embedding}, let \(\filter'\coloneqq\finfilterise(\cohsds)\) and let \(\cohsds'\coloneqq\findesirify(\filter')\), then we have to prove that \(\cohsds=\cohsds'\).
We start with the following chain of equivalences:
\begin{equation}\label{eq:the:finitary:representation:theorem:cohsds:embedding}
\sot\in\cohsds'
\ifandonlyif\cohsdts[\sot]\in\filter'
\ifandonlyif\cohsdts[\sot]\in\finfilterise(\cohsds)
\ifandonlyif\group{\exists\sosot\Subset\cohsds}
\event{\sosot}=\event{\set{\sot}},
\text{ for all~\(\sot\in\setsofthings\),}
\end{equation}
where the third equivalence follows from the fact that, by Equation~\eqref{eq:definition:event}, \(\event{\set{\sot}}=\cohsdts[\sot]\).
To prove that \(\cohsds\subseteq\cohsds'\), consider any~\(\sot\in\cohsds\), then \(\set{\sot}\Subset\cohsds\).
Using Equation~\eqref{eq:the:finitary:representation:theorem:cohsds:embedding} with \(\sosot\coloneqq\set{\sot}\), this implies that, indeed, \(\sot\in\cohsds'\).
For the converse inclusion, consider any~\(\sot\in\cohsds'\), so there's, by Equation~\eqref{eq:the:finitary:representation:theorem:cohsds:embedding}, some~\(\sosot\Subset\cohsds\) for which \(\event{\sosot}=\event{\set{\sot}}\).
Proposition~\ref{prop:it:does:not:matter:which}\ref{it:it:does:not:matter:which:equality} then guarantees that also \(\set{\sot}\Subset\cohsds\), or in other words, that, indeed, \(\sot\in\cohsds\).

For statement~\ref{it:the:finitary:representation:theorem:filter:embedding}, let~\(\cohsds'\coloneqq\findesirify(\filter)\) and let \(\filter'\coloneqq\finfilterise(\cohsds')\), then we have to prove that \(\filter'=\filter\), or in other words, that \(\event{\sosot}\in\filter'\ifandonlyif\event{\sosot}\in\filter\) for all~\(\sosot\Subset\setsofthings\).
Now, consider the following chain of equivalences, valid for any~\(\sosot\Subset\setsofthings\):
\begin{align*}
\event{\sosot}\in\filter'
&\ifandonlyif\event{\sosot}\in\finfilterise(\cohsds')
\ifandonlyif\group{\exists\sosot'\Subset\cohsds'}\event{\sosot}=\event{\sosot'}\\
&\ifandonlyif\group{\exists\sosot'\Subset\setsofthings}
\group[\big]{\sosot'\Subset\findesirify(\filter)\text{ and }\event{\sosot}=\event{\sosot'}}\\
&\ifandonlyif\group{\exists\sosot'\Subset\setsofthings}
\group[\big]{\event{\sosot'}\in\filter\text{ and }\event{\sosot}=\event{\sosot'}}\\
&\ifandonlyif\event{\sosot}\in\filter,
\end{align*}
where the second equivalence follows from the definition of~\(\finfilterise(\cohsds')\) and the fourth equivalence follows from Lemma~\ref{lem:characterisation:desirify:finitary}.

Statements~\ref{it:the:finitary:representation:theorem:cohsds:order:preserving} and~\ref{it:the:finitary:representation:theorem:filter:order:preserving} follow readily from the definitions of the maps~\(\finfilterise\) and~\(\findesirify\).

Before we turn to the proofs of statements~\ref{it:the:finitary:representation:theorem:filterise:bounds} and~\ref{it:the:finitary:representation:theorem:desirify:bounds}, it will be helpful to recall from the finitary counterparts of Propositions~\ref{prop:coherent:sds:constitute:a:complete:lattice} and~\ref{prop:the:smallest:cohsds} that \(\fincohsdsbottom=\cset{\sot\in\setsofthings}{\sot\cap\beautifulthings\neq\emptyset}\) and \(\fincohsdstop=\setsofthings\), and from Proposition~\ref{prop:filters:constititute:a:complete:lattice} and Theorem~\ref{thm:basic:sets:constitute:a:bounded:lattice} that \(\fineventfiltersbottom=\set{\cohsdts}\) and \(\fineventfilterstop=\finevents\).

For statement~\ref{it:the:finitary:representation:theorem:filterise:bounds}, consider any~\(\sosot\Subset\fincohsdsbottom\) and~\(\sot\in\sosot\),\footnote{The case that~\(\sosot=\emptyset\) is also covered, because \(\event{\emptyset}=\cohsdts\) as well; see the discussion after Equation~\eqref{eq:definition:event}.} then \(\sot\cap\beautifulthings\neq\emptyset\), and therefore, by Lemma~\ref{lem:the:cohsdtify:operator}\ref{it:the:cohsdtify:operator:background}, also \(\cohsdts[\sot]=\cohsdts\).
Equation~\eqref{eq:definition:event} then guarantees that \(\event{\sosot}=\cohsdts\).
But then indeed,  \(\finfilterise(\fincohsdsbottom)=\cset{\event{\sosot}}{\sosot\Subset\fincohsdsbottom}=\set{\cohsdts}=\fineventfiltersbottom\).
For the tops, we find that
\begin{equation*}
\finfilterise(\fincohsdstop)
=\finfilterise(\setsofthings)
=\cset{\event{\sosot}}{\sosot\Subset\setsofthings}
=\finevents
=\fineventfilterstop.
\end{equation*}

For statement~\ref{it:the:finitary:representation:theorem:desirify:bounds}, consider the following chain of equivalences, for any~\(\sot\in\setsofthings\):
\begin{equation*}
\sot\in\findesirify(\fineventfiltersbottom)
\ifandonlyif\sot\in\findesirify(\set{\cohsdts})
\ifandonlyif\cohsdts[\sot]=\cohsdts
\ifandonlyif\sot\cap\beautifulthings\neq\emptyset
\ifandonlyif\sot\in\fincohsdsbottom,
\end{equation*}
where the third equivalence follows from Lemma~\ref{lem:the:cohsdtify:operator}\ref{it:the:cohsdtify:operator:background}.
For the tops, we find that
\begin{equation*}
\findesirify(\fineventfilterstop)
=\findesirify(\finevents)
=\cset{\sot\in\setsofthings}{\event{\set{\sot}}\in\finevents}
=\setsofthings
=\fincohsdstop.
\end{equation*}

Finally, we turn to statement~\ref{it:the:finitary:representation:theorem:completeness}.
Consider that the proper filter~\(\filter\) is prime if and only if for all~\(\sosot[1],\sosot[2]\Subset\setsofthings\),
\begin{equation}\label{eq:the:finitary:representation:theorem:completeness:filters}
\eventwithindex{1}\cup\eventwithindex{2}\in\filter
\then\group{\eventwithindex{1}\in\filter\text{ or }\eventwithindex{2}\in\filter}.
\end{equation}
Since we've already argued in the proof of Theorem~\ref{thm:basic:sets:constitute:a:bounded:lattice} that \(\eventwithindex{1}\cup\eventwithindex{2}=\event{\sosot}\), with \(\sosot\coloneqq\cset{\sot[1]\cup\sot[2]}{\sot[1]\in\sosot[1]\text{ and }\sot[2]\in\sosot[2]}\), it follows from Proposition~\ref{prop:it:does:not:matter:which}\ref{it:it:does:not:matter:which:equality} that the statement~\eqref{eq:the:finitary:representation:theorem:completeness:filters} is equivalent to
\begin{equation}\label{eq:the:finitary:representation:theorem:completeness:cohsdss}
\cset{\sot[1]\cup\sot[2]}{\sot[1]\in\sosot[1]\text{ and }\sot[2]\in\sosot[2]}\Subset\cohsds
\then\group{\sosot[1]\Subset\cohsds\text{ or }\sosot[2]\Subset\cohsds}.
\end{equation}
Assume now that \(\filter\) is prime, and apply the condition~\eqref{eq:the:finitary:representation:theorem:completeness:cohsdss} with \(\sosot[1]\coloneqq\set{\sot[1]}\) and \(\sosot[2]\coloneqq\set{\sot[2]}\) to find the completeness condition~\eqref{eq:complete:cohsds}.
Conversely, assume that~\(\cohsds\) is complete, and consider any~\(\sosot[1],\sosot[2]\Subset\setsofthings\).
To prove that condition~\eqref{eq:the:finitary:representation:theorem:completeness:cohsdss} is satisfied, we assume that \(\cset{\sot[1]\cup\sot[2]}{\sot[1]\in\sosot[1]\text{ and }\sot[2]\in\sosot[2]}\Subset\cohsds\) and at the same time~\(\sosot[1]\not\Subset\cohsds\), and we prove that then necessarily \(\sosot[2]\Subset\cohsds\).
It follows from the assumptions that there's some~\(\sot[o]\in\sosot[1]\) such that \(\sot[o]\notin\cohsds\), while at the same time \(\cset{\sot[o]\cup\sot[2]}{\sot[2]\in\sosot[2]}\Subset\cohsds\).
But then the completeness condition~\eqref{eq:complete:cohsds} guarantees that \(\sot[2]\in\cohsds\) for all~\(\sot[2]\in\sosot[2]\), whence, indeed, \(\sosot[2]\Subset\cohsds\).
\end{proof}

\begin{lemma}\label{lem:the:cohsdtify:operator}
Consider any~\(\sot,\sot[1],\sot[2]\in\setsofthings\).
Then the following statements hold:
\begin{enumerate}[label=\upshape(\roman*),leftmargin=*]
\item\label{it:the:cohsdtify:operator:increasing} if \(\sot[1]\subseteq\sot[2]\) then \(\cohsdts[{\sot[1]}]\subseteq\cohsdts[{\sot[2]}]\);
\item\label{it:the:cohsdtify:operator:forbidden} \(\cohsdts[\sot]=\cohsdts[{\sot\setminus\uglythings}]\);
\item\label{it:the:cohsdtify:operator:is:upset} \(\cohsdtUpset{\cohsdts[\sot]}=\cohsdts[\sot]\);
\item\label{it:the:cohsdtify:operator:empty} if \(\sot\subseteq\uglythings\) then \(\cohsdts[\sot]=\emptyset\);
\item\label{it:the:cohsdtify:operator:background} \(\sot\cap\beautifulthings\neq\emptyset\) if and only if \(\cohsdts[\sot]=\cohsdts\);
\end{enumerate}
\end{lemma}

\begin{proof}
The proofs of~\ref{it:the:cohsdtify:operator:increasing} and~\ref{it:the:cohsdtify:operator:is:upset} are trivial, looking at Equation~\eqref{eq:the:cohsdtify:operator}.

For~\ref{it:the:cohsdtify:operator:forbidden}, consider any~\(\cohsdt\in\cohsdts\), and the following sequence of equivalences:
\begin{equation*}
\cohsdt\in\cohsdts[{\sot\setminus\uglythings}]
\ifandonlyif\cohsdt\cap(\sot\setminus\uglythings)\neq\emptyset
\ifandonlyif\cohsdt\cap\sot\neq\emptyset
\ifandonlyif\cohsdt\in\cohsdts[\sot],
\end{equation*}
where the penultimate equivalence holds because \(\cohsdt\cap(\sot\setminus\uglythings)=(\cohsdt\setminus\uglythings)\cap\sot=\cohsdt\cap\sot\), since \(\cohsdt\setminus\uglythings=\cohsdt\).
This tells us that, indeed, \(\cohsdts[\sot]=\cohsdts[{\sot\setminus\uglythings}]\).

For~\ref{it:the:cohsdtify:operator:empty}, consider any~\(\cohsdt\in\cohsdts\).
Assume that \(\sot\subseteq\uglythings\), then also \(\sot\cap\cohsdt\subseteq\uglythings\cap\cohsdt=\emptyset\), so \(\cohsdt\notin\cohsdts[\sot]\).
Hence, indeed, \(\cohsdts[\sot]=\emptyset\).

For~\ref{it:the:cohsdtify:operator:background}, consider any~\(\cohsdt\in\cohsdts\).
Assume that \(\sot\cap\beautifulthings\neq\emptyset\), then also \(\sot\cap\cohsdt\neq\emptyset\) because \(\beautifulthings\subseteq\cohsdt\), so \(\cohsdt\in\cohsdts[\sot]\).
Hence, indeed, \(\cohsdts[\sot]=\cohsdts\).
Conversely, assume that \(\cohsdts[\sot]=\cohsdts\), then since \(\beautifulthings=\bigcap\cohsdts\in\cohsdts\), we find that in particular \(\beautifulthings\in\cohsdts[\sot]\), so indeed \(\beautifulthings\cap\sot\neq\emptyset\).
\end{proof}

\begin{lemma}\label{lem:characterisation:desirify:finitary}
Let \(\filter\) be any proper filter on~\(\structure{\finevents,\subseteq}\) and let~\(\sosot\Subset\setsofthings\) be any finite SDS, then \(\sosot\Subset\findesirify(\filter)\ifandonlyif\event{\sosot}\in\filter\).
\end{lemma}

\begin{proof}
Consider the following chain of equivalences, valid for any~\(\sosot\Subset\setsofthings\):
\begin{align*}
\sosot\Subset\findesirify(\filter)
\ifandonlyif\group{\forall\sot\in\sosot}\sot\in\findesirify(\filter)
\ifandonlyif\group{\forall\sot\in\sosot}\cohsdts[\sot]\in\filter
\ifandonlyif\event{\sosot}\in\filter,
\end{align*}
where the last equivalence follows from Equation~\eqref{eq:definition:event} and the fact that \(\filter\) satisfies~\ref{axiom:lattice:filters:intersections} and~\ref{axiom:lattice:filters:increasing} [observe, by the way, that \(\event{\emptyset}=\cohsdts\in\filter\) when \(\filter\) is a proper filter].
\end{proof}

\begin{proof}[Proof of Theorem~\ref{thm:conjunctive:representation:finite}]
\underline{\ref{it:conjunctive:representation:finite:consistency}}.
That \(\event{\sosot}=\cset{\cohsdt\in\cohsdts}{\sosot\Subset\cohsds[\cohsdt]}\) follows from Equation~\eqref{eq:connection:with:binary:model:finitary}.
For necessity, suppose that \(\cohsds\) is finitely consistent.
That means that there's some finitely coherent~\(\cohsds[o]\in\fincohsdss\) such that \(\cohsds\subseteq\cohsds[o]\).
But it then follows from Proposition~\ref{prop:filter:consistency} that \(\event{\sosot}\neq\emptyset\) for all~\(\sosot\Subset\cohsds[o]\), and therefore in particular also for all~\(\sosot\Subset\cohsds\).

For sufficiency, assume that \(\event{\sosot}\neq\emptyset\) for all~\(\sosot\Subset\cohsds\).
Then we infer from Lemma~\ref{lem:conjunctive:representation:finite}\ref{it:conjunctive:representation:finite:filterbase} that \(\filterbase[\cohsds]\coloneqq\cset{\event{\sosot}}{\sosot\Subset\cohsds}\) is a filter base.
Let's denote by \(\filter[\cohsds]\) the (smallest) proper filter that it generates; see Lemma~\ref{lem:conjunctive:representation:finite}\ref{it:conjunctive:representation:finite:filter}.
If we now let \(\cohsds'\coloneqq\findesirify(\filter[\cohsds])\), then it follows from Theorem~\ref{thm:the:finitary:representation:theorem}\ref{it:the:finitary:representation:theorem:if:filter:then:cohsds} that \(\cohsds'\) is a finitely coherent SDS: \(\cohsds'\in\fincohsdss\).
We're therefore done if we can prove that \(\cohsds\subseteq\cohsds'\).
To this end, consider any~\(\sot\in\cohsds\), then \(\set{\sot}\Subset\cohsds\) and therefore also \(\cohsdts[\sot]=\event{\set{\sot}}\in\filterbase[\cohsds]\).
But then \(\cohsdts[\sot]\in\filter[\cohsds]\), and therefore indeed also \(\sot\in\findesirify(\filter[\cohsds])=\cohsds'\).

\underline{\ref{it:conjunctive:representation:finite:closure}}.
If \(\cohsds\) is not finitely consistent then it follows from~\ref{it:conjunctive:representation:finite:consistency} that there's some~\(\sosot\Subset\cohsds\) for which \(\event{\sosot}=\cset{\cohsdt\in\cohsdts}{\sosot\Subset\cohsds[\cohsdt]}=\emptyset\) and therefore \(\bigcap_{\cohsdt\in\cohsdts\colon\sosot\Subset\sdsify{\cohsdt}}\sdsify{\cohsdt}=\bigcap_{\cohsdt\in\event{\sosot}}\sdsify{\cohsdt}=\setsofthings\), making sure that both the left-hand side and the right-hand side in~\ref{it:conjunctive:representation:finite:closure} are equal to~\(\setsofthings\).
We may therefore assume that \(\cohsds\) is finitely consistent.
If we borrow the notation from the argumentation above in the proof of~\ref{it:conjunctive:representation:finite:consistency}, and take into account Lemma~\ref{lem:conjunctive:representation:finite}\ref{it:conjunctive:representation:finite:liminf} and Equation~\eqref{eq:connection:with:binary:model:finitary}, then it's clear that we have to prove that
\[
\fincohsdsclosure(\cohsds)=\bigcup_{V\in\filter[\cohsds]}\bigcap_{\cohsdt\in V}\sdsify{\cohsdt}.
\]
If we take into account Equation~\eqref{eq:from:filter:to:cohsds:finitary}, then we see that \(\bigcup_{V\in\filter[\cohsds]}\bigcap_{\cohsdt\in V}\sdsify{\cohsdt}=\findesirify(\filter[\cohsds])=\cohsds'\), so we need to prove that \(\cohsds'=\fincohsdsclosure(\cohsds)\).
Since we already know from the proof of~\ref{it:conjunctive:representation:finite:consistency} that \(\cohsds'\) is finitely coherent and that \(\cohsds\subseteq\cohsds'\), we find that \(\fincohsdsclosure(\cohsds)\subseteq\cohsds'\), because \(\fincohsdsclosure(\cohsds)\) is the smallest finitely coherent SDS that the finitely consistent~\(\cohsds\) is included in.
To prove the converse inclusion, namely that \(\cohsds'\subseteq\fincohsdsclosure(\cohsds)\), it suffices to consider any~\(\cohsds''\in\fincohsdss\) such that \(\cohsds\subseteq\cohsds''\), and prove that then~\(\cohsds'\subseteq\cohsds''\).
Now consider any~\(E\in\filter[\cohsds]\), then by Lemma~\ref{lem:conjunctive:representation:finite}\ref{it:conjunctive:representation:finite:filter} there are~\(\sosot,\sosot'\Subset\setsofthings\) such that \(E=\event{\sosot'}\) and \(\sosot\Subset\cohsds\) and \(\event{\sosot}\subseteq\event{\sosot'}\).
But then also \(\sosot\Subset\cohsds''\), and therefore \(\sosot'\Subset\cohsds''\), by Proposition~\ref{prop:it:does:not:matter:which}.
This implies that also \(E=\event{\sosot'}\in\finfilterise(\cohsds'')\), so we can conclude that \(\filter[\cohsds]\subseteq\finfilterise(\cohsds'')\).
But then indeed also \(\cohsds'=\findesirify(\filter[\cohsds])\subseteq\findesirify(\finfilterise(\cohsds''))=\cohsds''\), where the inclusion follows from Theorem~\ref{thm:the:finitary:representation:theorem}\ref{it:the:finitary:representation:theorem:filter:order:preserving} and the last equality follows from Theorem~\ref{thm:the:finitary:representation:theorem}\ref{it:the:finitary:representation:theorem:cohsds:embedding}.

\underline{\ref{it:conjunctive:representation:finite:coherence}}.
The proof is now straightforward, given~\ref{it:conjunctive:representation:finite:closure}, since finitely coherent means finitely consistent and closed with respect to the \(\fincohsdsclosure\)-operator.
\end{proof}

\begin{lemma}\label{lem:conjunctive:representation:finite}
Consider any SDS~\(\cohsds\subseteq\setsofthings\) such that \(\event{\sosot}\neq\emptyset\) for all~\(\sosot\Subset\cohsds\), then the following statements hold:
\begin{enumerate}[label=\upshape(\roman*),leftmargin=*]
\item\label{it:conjunctive:representation:finite:filterbase} the set \(\filterbase[\cohsds]\coloneqq\cset{\event{\sosot}}{\sosot\Subset\cohsds}\) is a filter base;
\item\label{it:conjunctive:representation:finite:filter} the smallest proper filter~\(\filter[\cohsds]\) that includes~\(\filterbase[\cohsds]\) is given by
\begin{equation*}
\filter[\cohsds]
\coloneqq\fineventUpset{\filterbase[\cohsds]}
=\cset[\big]{\event{\sosot'}}{\sosot'\Subset\setsofthings\text{ and }\group{\exists\sosot\Subset\cohsds}\event{\sosot}\subseteq\event{\sosot'}};
\end{equation*}
\item\label{it:conjunctive:representation:finite:liminf} \(\bigcup_{V\in\filter[\cohsds]}\bigcap_{\cohsdt\in V}\sdsify{\cohsdt}=\bigcup_{B\in\filterbase[\cohsds]}\bigcap_{\cohsdt\in B}\sdsify{\cohsdt}=\bigcup_{\sosot\Subset\cohsds}\bigcap_{\cohsdt\in\event{\sosot}}\sdsify{\cohsdt}\).
\end{enumerate}
\end{lemma}

\begin{proof}
\underline{\ref{it:conjunctive:representation:finite:filterbase}}.
Any non-empty collection of non-empty sets that is closed under finite intersections is in particular a filter base; see the discussion in Section~\ref{sec:filters}, and in particular condition~\eqref{eq:directed:downwards}.
First of all, observe that the set~\(\filterbase[\cohsds]\) is non-empty because \(\emptyset\Subset\cohsds\), and contains no empty set by assumption.
To show that it's closed under finite intersections, and therefore directed downwards, consider any~\(\sosot[1],\sosot[2]\Subset\cohsds\), then we infer from Equation~\eqref{eq:intersection:of:possibles} that \(\event{\sosot[1]}\cap\event{\sosot[2]}=\event{\sosot[1]\cup\sosot[2]}\), with still \(\sosot[1]\cup\sosot[2]\Subset\cohsds\), and therefore, indeed, \(\event{\sosot[1]}\cap\event{\sosot[2]}\in\filterbase[\cohsds]\).

\underline{\ref{it:conjunctive:representation:finite:filter}}.
Trivially from the discussion of filter bases in Section~\ref{sec:filters}.

\underline{\ref{it:conjunctive:representation:finite:liminf}}.
The second equality is trivial given the definition of~\(\filterbase[\cohsds]\).
For the first inequality, it's clearly enough to prove that \(\bigcup_{V\in\filter[\cohsds]}\bigcap_{\cohsdt\in V}\sdsify{\cohsdt}\subseteq\bigcup_{B\in\filterbase[\cohsds]}\bigcap_{\cohsdt\in B}\sdsify{\cohsdt}\).
But this is immediate, since \ref{it:conjunctive:representation:finite:filter} tells us that for any~\(V\in\filter[\cohsds]\) there's some~\(B\in\filterbase[\cohsds]\) such that \(B\subseteq V\) and therefore also \(\bigcap_{\cohsdt\in V}\sdsify{\cohsdt}\subseteq\bigcap_{\cohsdt\in B}\sdsify{\cohsdt}\).
\end{proof}

\begin{proof}[Proof of Theorem~\ref{thm:prime:filter:representation:finitely:coherent}]
Consider any finitely consistent~\(\cohsds\).
For necessity, assume that \(\cohsds\) is finitely coherent, then it follows from Theorem~\ref{thm:the:finitary:representation:theorem}\ref{it:the:finitary:representation:theorem:if:cohsds:then:filter} that \(\finfilterise(\cohsds)\) is a proper filter.
The prime filter representation result in Theorem~\ref{thm:prime:filter:representation} then guarantees that
\[
\finfilterise(\cohsds)
=\bigcap\cset[\big]{\primefilter\in\fineventprimefilters}{\finfilterise(\cohsds)\subseteq\primefilter},
\]
and applying the map~\(\findesirify\) to both sides of this equality leads to
\begin{align*}
\cohsds
=\findesirify(\finfilterise(\cohsds))
&=\findesirify\group[\bigg]{\bigcap\cset[\big]{\primefilter\in\fineventprimefilters}{\finfilterise(\cohsds)\subseteq\primefilter}}\\
&=\bigcap\cset[\big]{\findesirify\group{\primefilter}}{\primefilter\in\fineventprimefilters\text{ and }\finfilterise(\cohsds)\subseteq\primefilter},
\end{align*}
where the first equality follows from Theorem~\ref{thm:the:finitary:representation:theorem}\ref{it:the:finitary:representation:theorem:cohsds:embedding} and the last one from the fact that \(\findesirify\) is an order isomorphism between complete lattices of sets, by Theorem~\ref{thm:the:finitary:representation:theorem}, and because we have seen in the finitary counterpart to Proposition~\ref{prop:coherent:sds:constitute:a:complete:lattice} and in Proposition~\ref{prop:filters:constititute:a:complete:lattice} that intersection plays the role of infimum in these complete lattices.
This then leads to
\begin{align*}
\cohsds
&=\bigcap\cset[\big]{\findesirify\group{\primefilter}}{\primefilter\in\fineventprimefilters\text{ and }\finfilterise(\cohsds)\subseteq\primefilter}\\
&=\bigcap\cset[\big]{\findesirify\group{\primefilter}}{\primefilter\in\fineventprimefilters\text{ and }\cohsds\subseteq\findesirify(\primefilter)}
=\bigcap\cset[\big]{\cohsds'\in\completefincohsdss}{\cohsds\subseteq\cohsds'},
\end{align*}
where the second equality follows from Theorem~\ref{thm:the:finitary:representation:theorem}\ref{it:the:finitary:representation:theorem:cohsds:embedding}\&\ref{it:the:finitary:representation:theorem:filter:embedding}\&\ref{it:the:finitary:representation:theorem:cohsds:order:preserving}\&\ref{it:the:finitary:representation:theorem:filter:order:preserving} and the last equality follows from the correspondence between prime filters and complete finitely coherent SDSes in Theorem~\ref{thm:the:finitary:representation:theorem}\ref{it:the:finitary:representation:theorem:completeness}.

For sufficiency, assume that \(\cohsds=\bigcap\cset[\big]{\cohsds'\in\completefincohsdss}{\cohsds\subseteq\cohsds'}\).
That \(\cohsds\) is finitely consistent implies that \(\cohsds\neq\setsofthings\), and therefore also \(\cset[\big]{\cohsds'\in\completefincohsdss}{\cohsds\subseteq\cohsds'}\neq\emptyset\).
\(\cohsds\) is therefore finitely coherent as the intersection of a non-empty set of finitely coherent SDSes.
\end{proof}

\begin{proof}[Proof of Proposition~\ref{prop:conjunctive:is:complete}]
First off, it follows from Proposition~\ref{prop:cohsdt:to:cohsds} that \(\sdsify{\cohsdt}\) is (finitely) coherent.
Now, consider any~\(\sot[1],\sot[2]\subseteq\things\) and assume that \(\sot[1]\cup\sot[2]\in\sdsify{\cohsdt}\), or equivalently, that
\begin{equation*}
\emptyset\subset(\sot[1]\cup\sot[2])\cap\cohsdt=(\sot[1]\cap\cohsdt)\cup(\sot[2]\cap\cohsdt),
\end{equation*}
which clearly implies that, indeed, \(\sot[1]\in\sdsify{\cohsdt}\) or \(\sot[2]\in\sdsify{\cohsdt}\).
\end{proof}

\subsection{Proofs of results in Section~\ref{sec:representation}}\label{sec:proofs:representation}

\begin{proof}[Proof of Theorem~\ref{thm:the:representation:theorem}]
For statement~\ref{it:the:representation:theorem:if:cohsds:then:filter}, assume that \(\cohsds\) is a coherent SDS.

We must show that \(\finfilterise(\cohsds)\) is a proper principal filter.
To see that \(\finfilterise(\cohsds)\) is non-empty, observe that always \(\emptyset\Subset\cohsds\), and therefore \(\cohsdts=\event{\emptyset}\in\finfilterise(\cohsds)\).
To show that \(\filterise(\cohsds)\) is increasing [satisfies~\ref{axiom:lattice:filters:increasing}], consider any~\(\sosot[1],\sosot[2]\subseteq\setsofthings\) such that \(\eventwithindex{1}\in\filterise(\cohsds)\) and \(\eventwithindex{1}\subseteq\eventwithindex{2}\).
Then the infinitary version of Proposition~\ref{prop:it:does:not:matter:which}\ref{it:it:does:not:matter:which:equality} implies that~\(\sosot[1]\subseteq\cohsds\), and the infinitary version of Proposition~\ref{prop:it:does:not:matter:which}\ref{it:it:does:not:matter:which:inclusion} guarantees that then~\(\sosot[2]\subseteq\cohsds\), whence, indeed, also \(\eventwithindex{2}\in\filterise(\cohsds)\).
To show that \(\filterise(\cohsds)\) is closed under arbitrary non-empty intersections [satisfies~\ref{axiom:lattice:filters:arbitrary:intersections}], consider any non-empty family of~\(\sosot[i]\subseteq\cohsds\), \(i\in I\).
Then we infer from Equation~\eqref{eq:intersection:of:possibles} that \(\bigcap_{i\in I}\eventwithindex{i}=\event{{\bigcup_{i\in I}\sosot[i]}}\), and since \(\bigcup_{i\in I}\sosot[i]\subseteq\cohsds\), this tells us that, indeed, \(\bigcap_{i\in I}\eventwithindex{i}\in\filterise(\cohsds)\).
To show that \(\filterise(\cohsds)\) is proper, simply observe that it follows from Proposition~\ref{prop:filter:consistency} that \(\event{\sosot}\) is non-empty for all~\(\sosot\subseteq\cohsds\).
This ensures that \(\emptyset\notin\filterise(\cohsds)\), so \(\filterise(\cohsds)\neq\events\) [recall that \(\emptyset\in\events\) by Theorem~\ref{thm:basic:sets:constitute:a:bounded:lattice}].
We conclude that \(\filterise(\cohsds)\) is a proper principal filter.

For statement~\ref{it:the:representation:theorem:if:filter:then:cohsds}, assume that \(\filter\) is a proper principal filter.
We check that the relevant conditions are satisfied for \(\cohsds\) to be a coherent SDS.

\underline{\ref{axiom:desirable:sets:consistency}}.
Since \(\cohsdts[\emptyset]=\emptyset\) and \(\filter\) is proper, we find that \(\cohsdts[\emptyset]\notin\filter\) [recall again that \(\emptyset\in\events\) by Theorem~\ref{thm:basic:sets:constitute:a:bounded:lattice}].
Hence, indeed, \(\emptyset\notin\desirify(\filter)\).

\underline{\ref{axiom:desirable:sets:increasing}}.
Consider any~\(\sot[1],\sot[2]\in\setsofthings\) with~\(\sot[1]\subseteq\sot[2]\).
Assume that \(\sot[1]\in\desirify(\filter)\), so \(\cohsdts[{\sot[1]}]\in\filter\).
That \(\sot[1]\subseteq\sot[2]\) implies, via Lemma~\ref{lem:the:cohsdtify:operator}\ref{it:the:cohsdtify:operator:increasing}, that \(\cohsdts[{\sot[1]}]\subseteq\cohsdts[{\sot[2]}]\).
This allows us to infer that also \(\cohsdts[{\sot[2]}]\in\filter\), using~\ref{axiom:lattice:filters:increasing}.
Hence, indeed, \(\sot[2]\in\desirify(\filter)\).

\underline{\ref{axiom:desirable:sets:forbidden}}.
Consider any~\(\sot\in\setsofthings\), and assume that~\(\sot\in\desirify(\filter)\).
Infer from Lemma~\ref{lem:the:cohsdtify:operator}\ref{it:the:cohsdtify:operator:forbidden} that \(\cohsdts[\sot]=\cohsdts[{\sot\setminus\uglythings}]\) and therefore also \(\cohsdts[\sot]\in\filter\ifandonlyif\cohsdts[{\sot\setminus\uglythings}]\in\filter\).
Hence, indeed, \(\sot\setminus\uglythings\in\desirify(\filter)\).

\underline{\ref{axiom:desirable:sets:background}}.
We may assume without loss of generality that \(\beautifulthings\neq\emptyset\), because otherwise this requirement is trivially satisfied. Consider then any~\(\thing[{+}]\in\beautifulthings\), then we must show that \(\set{\thing[{+}]}\in\desirify(\filter)\), or in other words that \(\cohsdts[{\set{\thing[{+}]}}]\in\filter\).
Now simply observe that
\begin{equation*}
\cohsdts[{\set{\thing[{+}]}}]
=\cset{\cohsdt\in\cohsdts}{\set{\thing[{+}]}\cap\cohsdt\neq\emptyset}
=\cohsdts
\in\filter,
\end{equation*}
where the second equality holds because the coherence of~\(\cohsdt\) implies that \(\beautifulthings\subseteq\cohsdt\), and the final statement holds because the smallest principal filter~\(\eventprincipalfiltersbottom=\set{\cohsdts}\) is included in all principal filters.

\underline{\ref{axiom:desirable:sets:production}}.
Consider any non-empty subset~\(\sosot\) of~\(\desirify(\filter)\) and any choice~\(\thing[\selection]\in\initialclosure(\selection(\sosot))\) for all the~\(\selection\in\selections\), then we must prove that \(\sot[o]\coloneqq\cset{\thing[\selection]}{\selection\in\selections}\in\desirify(\filter)\), or in other words that \(\cohsdts[{\sot[o]}]\in\filter\).
It follows from the assumptions that \(\cohsdts[\sot]\in\filter\) for all~\(\sot\in\sosot\), and therefore also, using~\ref{axiom:lattice:filters:arbitrary:intersections} and Equation~\eqref{eq:definition:event}, that \(\event{\sosot}=\bigcap_{\sot\in\sosot}\cohsdts[\sot]\in\filter\), so it's enough to prove that \(\event{\sosot}\subseteq\cohsdts[{\sot[o]}]\), because we'll then also find, by~\ref{axiom:lattice:filters:increasing}, that \(\cohsdts[{\sot[o]}]\in\filter\), as required.
This is what we now set out to do.
Consider any~\(\cohsdt\in\event{\sosot}\), then it follows from Proposition~\ref{prop:the:two:systems} and Equation~\eqref{eq:definition:production:event} that there's some selection map~\(\selection[o]\in\selections\) such that \(\initialclosure(\selection[o](\sosot))\subseteq\cohsdt\), which guarantees that \(\thing[{\selection[o]}]\in\cohsdt\), and therefore \(\cohsdt\cap\sot[o]\neq\emptyset\).
Hence, indeed, \(\cohsdt\in\cohsdts[{\sot[o]}]\).

For statement~\ref{it:the:representation:theorem:cohsds:embedding}, let \(\filter'\coloneqq\filterise(\cohsds)\) and let \(\cohsds'\coloneqq\desirify(\filter')\), then we have to prove that \(\cohsds=\cohsds'\).
We start with the following chain of equivalences:
\begin{equation}\label{eq:the:representation:theorem:cohsds:embedding}
\sot\in\cohsds'
\ifandonlyif\cohsdts[\sot]\in\filter'
\ifandonlyif\cohsdts[\sot]\in\filterise(\cohsds)
\ifandonlyif\group{\exists\sosot\subseteq\cohsds}
\group[\big]{\event{\sosot}=\event{\set{\sot}}}
\text{ for all~\(\sot\in\setsofthings\),}
\end{equation}
where the third equivalence follows from the fact that, by Equation~\eqref{eq:definition:event}, \(\event{\set{\sot}}=\cohsdts[\sot]\).
To prove that \(\cohsds\subseteq\cohsds'\), consider any~\(\sot\in\cohsds\), so \(\set{\sot}\subseteq\cohsds\).
Using Equation~\eqref{eq:the:representation:theorem:cohsds:embedding} with \(\sosot\coloneqq\set{\sot}\), this implies that, indeed, \(\sot\in\cohsds'\).
For the converse inclusion, consider any~\(\sot\in\cohsds'\), so there's, by Equation~\eqref{eq:the:representation:theorem:cohsds:embedding}, some~\(\sosot\subseteq\cohsds\) for which \(\event{\sosot}=\event{\set{\sot}}\).
The infinitary version of Proposition~\ref{prop:it:does:not:matter:which}\ref{it:it:does:not:matter:which:equality} then guarantees that also \(\set{\sot}\subseteq\cohsds\), or in other words, that \(\sot\in\cohsds\).

For statement~\ref{it:the:representation:theorem:filter:embedding}, let~\(\cohsds'\coloneqq\desirify(\filter)\) and let \(\filter'\coloneqq\filterise(\cohsds')\), then we have to prove that \(\filter'=\filter\), or in other words, if we consider any~\(\sosot\subseteq\setsofthings\), that \(\event{\sosot}\in\filter'\ifandonlyif\event{\sosot}\in\filter\).
Now, consider the following chain of equivalences, valid for any~\(\sosot\subseteq\setsofthings\):
\begin{align*}
\event{\sosot}\in\filter'
&\ifandonlyif\event{\sosot}\in\filterise(\cohsds')
\ifandonlyif\group{\exists\sosot'\subseteq\cohsds'}
\event{\sosot}=\event{\sosot'}\\
&\ifandonlyif\group{\exists\sosot'\subseteq\setsofthings}
\group[\big]{\sosot'\subseteq\desirify(\filter)\text{ and }\event{\sosot}=\event{\sosot'}}\\
&\ifandonlyif\group{\exists\sosot'\subseteq\setsofthings}
\group[\big]{\event{\sosot'}\in\filter\text{ and }\event{\sosot}=\event{\sosot'}}\\
&\ifandonlyif\event{\sosot}\in\filter,
\end{align*}
where the second equivalence follows from the definition of~\(\filterise(\cohsds')\), and the fourth equivalence follows from Lemma~\ref{lem:characterisation:desirify}.

Statements~\ref{it:the:representation:theorem:cohsds:order:preserving} and~\ref{it:the:representation:theorem:filter:order:preserving} follow readily from the definitions of the maps~\(\filterise\) and~\(\desirify\).

Before we turn to the proofs of statements~\ref{it:the:representation:theorem:filterise:bounds} and~\ref{it:the:representation:theorem:desirify:bounds}, it will be helpful to recall from Propositions~\ref{prop:coherent:sds:constitute:a:complete:lattice} and~\ref{prop:the:smallest:cohsds} that \(\cohsdsbottom=\cset{\sot\in\setsofthings}{\sot\cap\beautifulthings\neq\emptyset}\) and \(\cohsdstop=\setsofthings\), together with \(\eventprincipalfiltersbottom=\set{\cohsdts}\) and \(\eventprincipalfilterstop=\events\).

For statement~\ref{it:the:representation:theorem:filterise:bounds}, consider any~\(\sosot\subseteq\cohsdsbottom\) and any~\(\sot\in\sosot\),\footnote{The case that~\(\sosot=\emptyset\) is also covered, because \(\event{\emptyset}=\cohsdts\) as well; see the discussion after Equation~\eqref{eq:definition:event}.} then \(\sot\cap\beautifulthings\neq\emptyset\), and therefore, by Lemma~\ref{lem:the:cohsdtify:operator}\ref{it:the:cohsdtify:operator:background}, also \(\cohsdts[\sot]=\cohsdts\).
Equation~\eqref{eq:definition:event} then guarantees that \(\event{\sosot}=\cohsdts\).
But then indeed, \(\filterise(\cohsdsbottom)=\cset{\event{\sosot}}{\sosot\subseteq\cohsdsbottom}=\set{\cohsdts}=\eventprincipalfiltersbottom\).
For the tops, we find that
\begin{equation*}
\filterise(\cohsdstop)
=\filterise(\setsofthings)
=\cset{\event{\sosot}}{\sosot\subseteq\setsofthings}
=\events
=\eventprincipalfilterstop.
\end{equation*}

For statement~\ref{it:the:representation:theorem:desirify:bounds}, consider the following chain of equivalences, for any~\(\sot\in\setsofthings\):
\begin{equation*}
\sot\in\desirify(\eventprincipalfiltersbottom)
\ifandonlyif\sot\in\desirify(\set{\cohsdts})
\ifandonlyif\cohsdts[\sot]=\cohsdts
\ifandonlyif\sot\cap\beautifulthings\neq\emptyset
\ifandonlyif\sot\in\cohsdsbottom,
\end{equation*}
where the third equivalence follows from Lemma~\ref{lem:the:cohsdtify:operator}\ref{it:the:cohsdtify:operator:background}.
For the tops, we find that
\begin{equation*}
\desirify(\eventprincipalfilterstop)
=\desirify(\events)
=\cset{\sot\in\setsofthings}{\event{\set{\sot}}\in\events}
=\setsofthings
=\cohsdstop.
\qedhere
\end{equation*}
\end{proof}

\begin{lemma}\label{lem:characterisation:desirify}
Let \(\filter\) be any proper principal filter on~\(\structure{\events,\subseteq}\) and let~\(\emptyset\neq\sosot\subseteq\setsofthings\) be any non-empty SDS, then \(\sosot\subseteq\desirify(\filter)\ifandonlyif\event{\sosot}\in\filter\).
\end{lemma}

\begin{proof}
Consider the following chain of equivalences:
\begin{align*}
\sosot\subseteq\desirify(\filter)
\ifandonlyif\group{\forall\sot\in\sosot}\sot\in\desirify(\filter)
\ifandonlyif\group{\forall\sot\in\sosot}\cohsdts[\sot]\in\filter
\ifandonlyif\event{\sosot}\in\filter,
\end{align*}
where the last equivalence follows from Equation~\eqref{eq:definition:event} and the fact that \(\filter\) satisfies~\ref{axiom:lattice:filters:increasing} and~\ref{axiom:lattice:filters:arbitrary:intersections} [observe, by the way, that \(\event{\emptyset}=\cohsdts\in\filter\) when \(\filter\) is a proper principal filter].
\end{proof}

\begin{proof}[Proof of Theorem~\ref{thm:conjunctive:representation}]
\underline{\ref{it:conjunctive:representation:consistency}}.
The equality between the two sets follows from Equation~\eqref{eq:connection:with:binary:model}.

For necessity, suppose that \(\cohsds\) is consistent.
That means that there's some coherent~\(\cohsds[o]\in\cohsdss\) such that \(\cohsds\subseteq\cohsds[o]\).
But it then follows from Proposition~\ref{prop:filter:consistency} that \(\event{\sosot}\neq\emptyset\) for all~\(\sosot\subseteq\cohsds[o]\), and therefore in particular also for~\(\sosot=\cohsds\).

For sufficiency, assume that \(\event{\cohsds}\neq\emptyset\).
Denote by \(\filter[\cohsds]\coloneqq\eventupset{\event{\cohsds}}\) the (smallest) proper principal filter that it generates.
If we now let \(\cohsds'\coloneqq\desirify(\filter[\cohsds])\), then it follows from Theorem~\ref{thm:the:representation:theorem}\ref{it:the:representation:theorem:if:filter:then:cohsds} that \(\cohsds'\) is a coherent SDS: \(\cohsds'\in\cohsdss\).
We're therefore done if we can prove that \(\cohsds\subseteq\cohsds'\).
To this end, consider any~\(\sot\in\cohsds\), then \(\set{\sot}\subseteq\cohsds\) and therefore also \(\event{\cohsds}\subseteq\event{\set{\sot}}\).
But then \(\cohsdts[\sot]=\event{\set{\sot}}\in\filter[\cohsds]\), and therefore indeed also \(\sot\in\desirify(\filter[\cohsds])=\cohsds'\).

\underline{\ref{it:conjunctive:representation:closure}}.
If \(\cohsds\) is not consistent then it follows from~\ref{it:conjunctive:representation:consistency} that \(\event{\cohsds}=\emptyset\), and therefore also from Equation~\eqref{eq:connection:with:binary:model} that \(\bigcap_{\cohsdt\in\cohsdts\colon\cohsds\subseteq\sdsify{\cohsdt}}\sdsify{\cohsdt}=\bigcap_{\cohsdt\in\event{\cohsds}}\sdsify{\cohsdt}=\setsofthings\), making sure that both the left-hand side and the right-hand side in~\ref{it:conjunctive:representation:closure} are equal to~\(\setsofthings\).

We may therefore assume that \(\cohsds\) is consistent, and use the argumentation and notations above in the proof of~\ref{it:conjunctive:representation:consistency}, starting with the proper principal filter \(\filter[\cohsds]\coloneqq\eventupset{\event{\cohsds}}\)  and the corresponding coherent SDS~\(\cohsds'\coloneqq\desirify(\filter[\cohsds])\) for which we know that \(\cohsds\subseteq\cohsds'\), and, by Theorem~\ref{thm:the:representation:theorem}\ref{it:the:representation:theorem:filter:embedding}, that \(\filterise(\cohsds')=\filter[\cohsds]\).
Applying Equation~\eqref{eq:from:cohsds:to:filter} for the coherent~\(\cohsds'\) leads to
\[
\event{\cohsds'}
=\bigcap\filterise(\cohsds')
=\bigcap\filter[\cohsds]
=\event{\cohsds},
\]
so Equation~\eqref{eq:from:filter:to:cohsds} tells us that \(\cohsds'=\bigcap_{\cohsdt\in\cohsdts\colon\cohsds\subseteq\sdsify{\cohsdt}}\sdsify{\cohsdt}\).
We therefore need to prove that \(\cohsdsclosure(\cohsds)=\cohsds'\).
But since we already know that \(\cohsds\subseteq\cohsds'\), it's enough to consider any~\(\cohsds''\in\cohsdss\) such that \(\cohsds\subseteq\cohsds''\), and to prove that then~\(\cohsds'\subseteq\cohsds''\).
Now \(\cohsds\subseteq\cohsds''\) implies that \(\event{\cohsds''}\subseteq\event{\cohsds}\) and therefore that \(\filter[\cohsds]=\eventupset{\event{\cohsds}}\subseteq\eventupset{\event{\cohsds''}}=\filterise(\cohsds'')\), where the last equality follows from applying Equation~\eqref{eq:from:cohsds:to:filter} for the coherent~\(\cohsds''\).
But then indeed also, by Theorem~\ref{thm:the:representation:theorem}\ref{it:the:representation:theorem:filter:order:preserving}\&\ref{it:the:representation:theorem:cohsds:embedding}, \(\cohsds'=\desirify(\filter[\cohsds])\subseteq\desirify(\filterise(\cohsds''))=\cohsds''\).

\underline{\ref{it:conjunctive:representation:coherence}}.
The proof is now straightforward, given~\ref{it:conjunctive:representation:closure}, since coherent means consistent and closed with respect to the \(\cohsdsclosure\)-operator.
\end{proof}

\subsection{Proofs of results in Section~\ref{sec:finitary:models}}\label{sec:proofs:finitary:models}

\begin{proof}[Proof of Proposition~\ref{prop:finite:coherence:of:finitary:part}]
It's enough to consider any finitely coherent~\(\cohsds\in\fincohsdss\), and to check that \(\fintypart{\cohsds}\) satisfies the relevant axioms~\ref{axiom:desirable:sets:consistency}--\ref{axiom:desirable:sets:background} and~\ref{axiom:desirable:sets:finitary:production}.

\underline{\ref{axiom:desirable:sets:consistency}}.
It follows from the finite coherence of~\(\cohsds\) that \(\emptyset\notin\cohsds\) [use~\ref{axiom:desirable:sets:consistency}], and therefore that \(\emptyset\notin\cohsds\cap\finitesetsofthings\), whence also \(\emptyset\notin\sotUpset{\cohsds\cap\finitesetsofthings}=\fintypart{\cohsds}\).

\underline{\ref{axiom:desirable:sets:increasing}}.
Consider that
\begin{equation*}
\sotUpset{\fintypart{\cohsds}}
=\sotUpset[\big]{\sotUpset{\cohsds\cap\finitesetsofthings}}
=\sotUpset{\cohsds\cap\finitesetsofthings}
=\fintypart{\cohsds},
\end{equation*}
where the second equality follows because \(\sotUpset{\bolleke}\) is a closure operator [see the comments in the beginning of Section~\ref{sec:representation:lattices}].

\underline{\ref{axiom:desirable:sets:forbidden}}.
Consider any~\(\sot\in\fintypart{\cohsds}\), then there's some~\(\finsot\in\cohsds\cap\finitesetsofthings\) such that \(\finsot\Subset\sot\).
But then also \(\finsot\setminus\uglythings\subseteq\sot\setminus\uglythings\), and since \(\finsot\setminus\uglythings\in\cohsds\cap\finitesetsofthings\) by the finite coherence of~\(\cohsds\) [use~\ref{axiom:desirable:sets:forbidden}], we find that, indeed, \(\sot\setminus\uglythings\in\sotUpset{\cohsds\cap\finitesetsofthings}=\fintypart{\cohsds}\).

\underline{\ref{axiom:desirable:sets:background}}.
Consider any~\(\thing[+]\in\beautifulthings\).
Since \(\set{\thing[+]}\in\cohsds\) by the finite coherence of~\(\cohsds\) [use~\ref{axiom:desirable:sets:background}], and therefore also \(\set{\thing[+]}\in\cohsds\cap\finitesetsofthings\), we trivially find that \(\set{\thing[+]}\in\sotUpset{\cohsds\cap\finitesetsofthings}=\fintypart{\cohsds}\).

\underline{\ref{axiom:desirable:sets:finitary:production}}.
Fix any non-empty~\(\sosot\Subset\fintypart{\cohsds}\), and for all corresponding selections~\(\selection\in\selections\), some~\(\thing[\selection]\in\initialclosure(\selection(\sosot))\).
Then we must prove that \(\cset{\thing[\selection]}{\selection\in\selections[\sosot]}\in\fintypart{\cohsds}\).

For any~\(\sot\in\sosot\), there's some~\(\finsot\in\cohsds\) such that \(\finsot\Subset\sot\).
Let, with obvious notations, \(\finite{\sosot}\coloneqq\cset{\finsot}{\sot\in\sosot}\).
Then \(\finite{\sosot}\Subset\cohsds\cap\finitesetsofthings\), so \(\finite{\sosot}\) is a finite set of finite subsets of~\(\things\), which in turn implies that the set of selections~\(\selections[{\finite{\sosot}}]\) is finite as well.
For any~\(\finite{\selection}\in\selections[\finite{\sosot}]\), we consider the corresponding selection \(\selection_{\finite{\selection}}\in\selections\), defined by \(\selection_{\finite{\selection}}(\sot)\coloneqq\finite{\selection}(\finsot)\in\finsot\Subset\sot\), for all \(\sot\in\sosot\), again with obvious notations, implying that \(\selection_{\finite{\selection}}(\sosot)=\finite{\selection}(\finite{\sosot})\).
If we now let \(\thing[\finite{\selection}]\coloneqq\thing[{\selection_{\finite{\selection}}}]\), then it follows from the argumentation above that \(\thing[\finite{\selection}]\in\initialclosure(\finite{\selection}(\finite{\sosot}))\), simply because by assumption \(\thing[\selection_{\finite{\selection}}]\in\initialclosure(\selection_{\finite{\selection}}(\sosot))\).
Clearly,
\begin{equation}\label{eq:finite:coherence:of:finitary:part}
\cset{\thing[\finite{\selection}]}{\finite{\selection}\in\selections[{\finite{\sosot}}]}
=\cset{\thing[\selection_{\finite{\selection}}]}{\finite{\selection}\in\selections[{\finite{\sosot}}]}
\subseteq\cset{\thing[\selection]}{\selection\in\selections[\sosot]}.
\end{equation}
We infer from the finite coherence of~\(\cohsds\) that \(\cset{\thing[\finite{\selection}]}{\finite{\selection}\in\selections[{\finite{\sosot}}]}\in\cohsds\cap\finitesetsofthings\) [use~\ref{axiom:desirable:sets:finitary:production} and the finiteness of~\(\sofinsot\) and~\(\selections[{\finite{\sosot}}]\)].
But then Equation~\eqref{eq:finite:coherence:of:finitary:part} tells us that, indeed, \(\cset{\thing[\selection]}{\selection\in\selections[\sosot]}\in\fintypart{\cohsds}\).
\end{proof}

\begin{proof}[Proof of Proposition~\ref{prop:conjunctive:is:finitary}]
It follows from Proposition~\ref{prop:cohsdt:to:cohsds} that \(\sdsify{\cohsdt}\) is (finitely) coherent.
It then follows from~\ref{axiom:desirable:sets:increasing} that \(\fintypart{\sdsify{\cohsdt}}\subseteq\sdsify{\cohsdt}\).
For the converse inclusion, consider any~\(\sot\in\sdsify{\cohsdt}\), then \(\sot\cap\cohsdt\neq\emptyset\).
Consider, therefore, any~\(\thing\in\sot\cap\cohsdt\), then \(\set{t}\subseteq\sot\) and \(\set{t}\in\sdsify{\cohsdt}\cap\finitesetsofthings\), so \(\sot\in\fintypart{\sdsify{\cohsdt}}\), and therefore, indeed, \(\sdsify{\cohsdt}\subseteq\fintypart{\sdsify{\cohsdt}}\).
\end{proof}

\begin{proof}[Proof of Proposition~\ref{prop:complete:finitary:is:conjunctive}]
We give the proof for the second statement involving finite coherence.
The proof for the first statement involving coherence is similar, but slightly simpler.
That \(\fintypart{\sdsify{\cohsdt}}=\sdtify{\cohsds}\) follows from Proposition~\ref{prop:conjunctive:is:finitary}.
So, assume that the closure operator~\(\initialclosure\) is finitary, and consider any~\(\cohsds\in\completefincohsdss\).
It suffices to prove that \(\cohsds\cap\finitesetsofthings=\sdsify{\cohsdt}\cap\finitesetsofthings\) for some coherent \(\cohsdt\in\cohsdts\).

Let \(\cohsdt\coloneqq\sdtify{\cohsds}\), then \(\cohsdt\) is coherent and \(\sdsify{\cohsdt}\subseteq\cohsds\), by Proposition~\ref{prop:cohsds:to:cohsdt}, so it only remains to prove that \(\cohsds\cap\finitesetsofthings\subseteq\sdsify{\cohsdt}\cap\finitesetsofthings\).
So, consider any~\(\sot\in\cohsds\cap\finitesetsofthings\), then there are two possibilities.
If \(\sot\) is a singleton~\(\set{\thing}\), then necessarily \(\thing\in\sdtify{\cohsds}=\cohsdt\), and therefore \(\sot=\set{\thing}\in\sdsify{\cohsdt}\).
Otherwise, it follows from the completeness of~\(\cohsds\) that there's some strict subset~\(\sot'\) of~\(\sot\) such that \(\sot'\in\cohsds\), and therefore also \(\sot'\in\cohsds\cap\finitesetsofthings\).
Now replace \(\sot\) with \(\sot'\) in the argumentation above, and keep the recursion going until we do reach a singleton~\(\set{\thing}\) for which \(\thing\in\sdtify{\cohsds}=\cohsdt\), or equivalently \(\set{\thing}\in\cohsds\), which must happen after a finite number of steps, because \(\sot\) is finite.
But then necessarily also \(\thing\in\sot\), so \(\sot\cap\cohsdt\neq\emptyset\), and therefore indeed \(\sot\in\sdsify{\cohsdt}\cap\finitesetsofthings\).
\end{proof}

\begin{proof}[Proof of Theorem~\ref{thm:conjunctive:representation:finitary}]
The proof of sufficiency is fairly straightforward.
Assume that \(\cohsds=\bigcap\cset{\sdsify{\cohsdt}}{\cohsdt\in\cohsdts\text{ and }\cohsds\subseteq\sdsify{\cohsdt}}\).
Recall that it follows from Proposition~\ref{prop:cohsdt:to:cohsds} that any element of the set~\(\cset{\sdsify{\cohsdt}}{\cohsdt\in\cohsdts\text{ and }\cohsds\subseteq\sdsify{\cohsdt}}\neq\emptyset\) is finitely coherent, and therefore also strictly included in~\(\setsofthings\).
Since the assumed finite consistency of \(\cohsds\) implies that \(\cohsds\neq\setsofthings\), this guarantees that \(\cset{\sdsify{\cohsdt}}{\cohsdt\in\cohsdts\text{ and }\cohsds\subseteq\sdsify{\cohsdt}}\neq\emptyset\).
Hence, \(\cohsds\) is finitely coherent, as an intersection of a non-empty collection of finitely coherent SDSes.

We now turn to the proof of necessity, so assume that the SDS~\(\cohsds\) is finitary and finitely coherent.
Our proof will rely on the Prime Filter Representation Theorem for finitely coherent SDSes [Theorem~\ref{thm:prime:filter:representation:finitely:coherent}], which tells us that \(\cohsds=\bigcap\cset{\cohsds'\in\completefincohsdss}{\cohsds\subseteq\cohsds'}\).
Consider any~\(\cohsds'\in\completefincohsdss\) such that \(\cohsds\subseteq\cohsds'\) [it follows from the assumptions that these is always at least one such~\(\cohsds'\)].
Then also \(\finpart{\cohsds}\subseteq\finpart{\cohsds'}\), and therefore
\[
\cohsds
=\fintypart{\cohsds}
=\sotUpset{\finpart{\cohsds}}
\subseteq\sotUpset{\finpart{\cohsds'}}
=\fintypart{\cohsds'}
=\sdsify{\sdtify{\cohsds'}}
\subseteq\cohsds',
\]
where the first equality follows from the assumed finitary character of~\(\cohsds\), the last equality follows from the finitary character of the finitely coherent conjunctive~\(\sdsify{\sdtify{\cohsds'}}\) [see Proposition~\ref{prop:complete:finitary:is:conjunctive}], and the last inclusion follows from Proposition~\ref{prop:cohsds:to:cohsdt}.
We've thus proved that
\begin{equation*}
\cohsds\subseteq\cohsds'
\then
\cohsds\subseteq\sdsify{\sdtify{\cohsds'}}\subseteq\cohsds',
\text{ for all~\(\cohsds'\in\completefincohsdss\)}.
\end{equation*}
But, since \(\sdtify{\cohsds'}\in\cohsdts\) for all~\(\cohsds'\in\completefincohsdss\), by Proposition~\ref{prop:cohsds:to:cohsdt}, this implies that
\begin{equation*}
\cset{\sdsify{\sdtify{\cohsds'}}}{\cohsds'\in\completefincohsdss\text{ and }\cohsds\subseteq\cohsds'}
\subseteq\cset{\sdsify{\cohsdt}}{\cohsdt\in\cohsdts\text{ and }\cohsds\subseteq\sdsify{\cohsdt}},
\end{equation*}
and therefore
\begin{align*}
\cohsds
\subseteq\bigcap\cset{\sdsify{\cohsdt}}{\cohsdt\in\cohsdts\text{ and }\cohsds\subseteq\sdsify{\cohsdt}}
&\subseteq\bigcap\cset{\sdsify{\sdtify{\cohsds'}}}{\cohsds'\in\completefincohsdss\text{ and }\cohsds\subseteq\cohsds'}\\
&\subseteq\bigcap\cset{\cohsds'\in\completefincohsdss}{\cohsds\subseteq\cohsds'}
=\cohsds,
\end{align*}
where the third inclusion follows again from Proposition~\ref{prop:cohsds:to:cohsdt}, and the equality follows from Theorem~\ref{thm:prime:filter:representation:finitely:coherent}.
\end{proof}

\begin{proof}[Proof of Corollary~\ref{cor:finitary:and:coherent}]
For~\ref{it:finitary:and:coherent:finitely:coherent}, recall that all the conjunctive models have the form~\(\sdsify{\cohsdt}\) for \(\cohsdt\in\cohsdts\) by Proposition~\ref{prop:conjunctive:and:coherent:sds}, and are therefore coherent by Proposition~\ref{prop:cohsdt:to:cohsds}.
So is, therefore, any intersection of them.
Now recall Theorem~\ref{thm:conjunctive:representation:finitary}.

For~\ref{it:finitary:and:coherent:finitary:is:coherent}, assume that \(\cohsds\) is (finitely) coherent, then Proposition~\ref{prop:finite:coherence:of:finitary:part} guarantees that \(\fintypart{\cohsds}\) is finitely coherent.
But since \(\fintypart{\cohsds}\) is finitary, we infer from the first statement that its finite coherence implies its coherence.
\end{proof}

\subsection{Proofs of results in Section~\ref{sec:representation:finitary:finite}}\label{sec:proofs:representation:finitary:finite}

\begin{proof}[Proof of Theorem~\ref{thm:the:finitary:representation:theorem:finite}]
The proof is very similar to the proof of Theorem~\ref{thm:the:finitary:representation:theorem}, but we nevertheless include it here for the sake of completeness.

For statement~\ref{it:the:finitary:representation:theorem:finite:if:cohsds:then:filter}, assume that \(\cohsdfs\) is a finitely coherent SDFS.
We must show that \(\finfinfilterise(\cohsdfs)\) is a proper filter.
To see that \(\finfinfilterise(\cohsdfs)\) is non-empty, observe that always \(\emptyset\Subset\cohsdfs\), and therefore \(\cohsdts=\event{\emptyset}\in\finfinfilterise(\cohsdfs)\).
To show that \(\finfinfilterise(\cohsdfs)\) is increasing [satisfies~\ref{axiom:lattice:filters:increasing}], consider any~\(\sofinsot[1],\sofinsot[2]\Subset\finitesetsofthings\) such that \(\fineventwithindex{1}\in\finfinfilterise(\cohsdfs)\) and \(\fineventwithindex{1}\subseteq\fineventwithindex{2}\).
Then Proposition~\ref{prop:it:does:not:matter:which}\ref{it:it:does:not:matter:which:equality:finite} implies that~\(\sofinsot[1]\Subset\cohsdfs\), and Proposition~\ref{prop:it:does:not:matter:which}\ref{it:it:does:not:matter:which:inclusion:finite} then guarantees that also~\(\sofinsot[2]\Subset\cohsdfs\), whence, indeed, also \(\fineventwithindex{2}\in\finfinfilterise(\cohsdfs)\).
To show that \(\finfinfilterise(\cohsds)\) is closed under finite intersections [satisfies~\ref{axiom:lattice:filters:intersections}], consider arbitrary~\(\sofinsot[1],\sofinsot[2]\Subset\cohsdfs\).
Then we infer from Equation~\eqref{eq:intersection:of:possibles} that \(\fineventwithindex{1}\cap\fineventwithindex{2}=\event{{\sofinsot[1]\cup\sofinsot[2]}}\), and since \(\sofinsot[1]\cup\sofinsot[2]\Subset\cohsdfs\), this tells us that, indeed, \(\fineventwithindex{1}\cap\fineventwithindex{2}\in\finfinfilterise(\cohsdfs)\).
To show that the filter~\(\finfinfilterise(\cohsdfs)\) is proper, simply observe that it follows from Proposition~\ref{prop:filter:consistency} that \(\event{\sofinsot}\neq\emptyset\) for all~\(\sofinsot\Subset\cohsdfs\).
This ensures that \(\emptyset\notin\finfinfilterise(\cohsdfs)\), so \(\finfilterise(\cohsdfs)\neq\finfinevents\) [recall that \(\emptyset\in\finfinevents\) by Theorem~\ref{thm:basic:sets:constitute:a:bounded:lattice}].
We conclude that \(\finfilterise(\cohsdfs)\) is a proper filter.

For statement~\ref{it:the:finitary:representation:theorem:finite:if:filter:then:cohsds}, assume that \(\filter\) is a proper filter.
We check that the relevant conditions are satisfied for \(\cohsdfs\) to be a finitely coherent SDFS.

\underline{\ref{axiom:desirable:finite:sets:consistency}}.
Since \(\cohsdts[\emptyset]=\emptyset\) and \(\filter\) is proper, we find that \(\cohsdts[\emptyset]\notin\filter\) [recall again that \(\emptyset\in\finfinevents\) by Theorem~\ref{thm:basic:sets:constitute:a:bounded:lattice}].
Hence, indeed, \(\emptyset\notin\finfindesirify(\filter)\).

\underline{\ref{axiom:desirable:finite:sets:increasing}}.
Consider any~\(\finsot[1],\finsot[2]\in\finitesetsofthings\) with~\(\finsot[1]\subseteq\finsot[2]\).
Assume that \(\finsot[1]\in\finfindesirify(\filter)\), so \(\cohsdts[{\finsot[1]}]\in\filter\).
That \(\finsot[1]\subseteq\finsot[2]\) implies, via Lemma~\ref{lem:the:cohsdtify:operator}\ref{it:the:cohsdtify:operator:increasing}, that \(\cohsdts[{\finsot[1]}]\subseteq\cohsdts[{\finsot[2]}]\).
This allows us to infer that also \(\cohsdts[{\finsot[2]}]\in\filter\), using~\ref{axiom:lattice:filters:increasing}.
Hence, indeed, \(\finsot[2]\in\finfindesirify(\filter)\).

\underline{\ref{axiom:desirable:finite:sets:forbidden}}.
Consider any~\(\finsot\in\finitesetsofthings\), and assume that~\(\finsot\in\finfindesirify(\filter)\).
Infer from Lemma~\ref{lem:the:cohsdtify:operator}\ref{it:the:cohsdtify:operator:forbidden} that \(\cohsdts[\finsot]=\cohsdts[{\finsot\setminus\uglythings}]\) and therefore also \(\cohsdts[\finsot]\in\filter\ifandonlyif\cohsdts[{\finsot\setminus\uglythings}]\in\filter\).
Hence, indeed, \(\finsot\setminus\uglythings\in\finfindesirify(\filter)\).

\underline{\ref{axiom:desirable:finite:sets:background}}.
We may assume without loss of generality that \(\beautifulthings\neq\emptyset\), because otherwise this requirement is trivially satisfied.
Consider then any~\(\thing[{+}]\in\beautifulthings\), then we must show that \(\set{\thing[{+}]}\in\finfindesirify(\filter)\), or in other words that \(\cohsdts[{\set{\thing[{+}]}}]\in\filter\).
Now simply observe that
\begin{equation*}
\cohsdts[{\set{\thing[{+}]}}]
=\cset{\cohsdt\in\cohsdts}{\set{\thing[{+}]}\cap\cohsdt\neq\emptyset}
=\cohsdts
\in\filter,
\end{equation*}
where the second equality holds because the coherence of~\(\cohsdt\) implies that \(\beautifulthings\subseteq\cohsdt\), and the final statement holds because the smallest filter~\(\finfineventfiltersbottom=\set{\cohsdts}\) is included in all filters.

\underline{\ref{axiom:desirable:finite:sets:production}}.
Consider any non-empty~\(\sofinsot\Subset\finfindesirify(\filter)\) and any choice~\(\thing[\selection]\in\initialclosure(\selection(\sofinsot))\) for all the~\(\selection\in\finselections\), then we must prove that \(\finsot[o]\coloneqq\cset{\thing[\selection]}{\selection\in\finselections}\in\finfindesirify(\filter)\), or in other words that \(\cohsdts[{\finsot[o]}]\in\filter\).
It follows from the assumptions that \(\cohsdts[\finsot]\in\filter\) for all~\(\finsot\in\sofinsot\), and therefore also, by~\ref{axiom:lattice:filters:intersections} and the finiteness of~\(\sofinsot\), that \(\event{\sofinsot}=\bigcap_{\finsot\in\sofinsot}\cohsdts[\finsot]\in\filter\), so it's enough to prove that \(\event{\sofinsot}\subseteq\cohsdts[{\finsot[o]}]\), because then also, by~\ref{axiom:lattice:filters:increasing}, \(\cohsdts[{\finsot[o]}]\in\filter\), as required.
Consider, to this end, any~\(\cohsdt\in\event{\sofinsot}\), then it follows from Proposition~\ref{prop:the:two:systems} and Equation~\eqref{eq:definition:production:event} that there's some selection map~\(\selection[o]\in\finselections\) such that \(\initialclosure(\selection[o](\sofinsot))\subseteq\cohsdt\).
This guarantees that \(\thing[{\selection[o]}]\in\cohsdt\), and therefore \(\cohsdt\cap\finsot[o]\neq\emptyset\).
Hence, indeed, \(\cohsdt\in\cohsdts[{\finsot[o]}]\).

For statement~\ref{it:the:finitary:representation:theorem:finite:cohsds:embedding}, let \(\filter'\coloneqq\finfinfilterise(\cohsdfs)\) and let \(\cohsdfs'\coloneqq\finfindesirify(\filter')\), then we have to prove that \(\cohsdfs=\cohsdfs'\).
We start with the following chain of equivalences:
\begin{equation}\label{eq:the:finitary:representation:theorem:finite:cohsds:embedding}
\finsot\in\cohsdfs'
\ifandonlyif\cohsdts[\finsot]\in\filter'
\ifandonlyif\cohsdts[\finsot]\in\finfinfilterise(\cohsdfs)
\ifandonlyif\group{\exists\sofinsot\Subset\cohsdfs}
\event{\sofinsot}=\event{\set{\finsot}},
\text{ for all~\(\finsot\in\finitesetsofthings\),}
\end{equation}
where the third equivalence follows from the fact that, by Equation~\eqref{eq:definition:event}, \(\event{\set{\finsot}}=\cohsdts[\finsot]\).
To prove that \(\cohsdfs\subseteq\cohsdfs'\), consider any~\(\finsot\in\cohsdfs\), then \(\set{\finsot}\Subset\cohsdfs\).
Using Equation~\eqref{eq:the:finitary:representation:theorem:finite:cohsds:embedding} with \(\sofinsot\coloneqq\set{\finsot}\), this implies that, indeed, \(\finsot\in\cohsdfs'\).
For the converse inclusion, consider any~\(\finsot\in\cohsdfs'\), so there's, by Equation~\eqref{eq:the:finitary:representation:theorem:finite:cohsds:embedding}, some~\(\sofinsot\Subset\cohsdfs\) for which \(\event{\sofinsot}=\event{\set{\finsot}}\).
Proposition~\ref{prop:it:does:not:matter:which}\ref{it:it:does:not:matter:which:equality:finite} then guarantees that also \(\set{\finsot}\Subset\cohsdfs\), or in other words, that, indeed, \(\finsot\in\cohsdfs\).

For statement~\ref{it:the:finitary:representation:theorem:finite:filter:embedding}, let~\(\cohsdfs'\coloneqq\finfindesirify(\filter)\) and let \(\filter'\coloneqq\finfinfilterise(\cohsdfs')\), then we have to prove that \(\filter'=\filter\), or in other words, that \(\event{\sofinsot}\in\filter'\ifandonlyif\event{\sofinsot}\in\filter\) for all~\(\sofinsot\Subset\finitesetsofthings\).
Now, consider the following chain of equivalences, valid for any~\(\sofinsot\Subset\finitesetsofthings\):
\begin{align*}
\event{\sofinsot}\in\filter'
&\ifandonlyif\event{\sofinsot}\in\finfinfilterise(\cohsdfs')
\ifandonlyif\group{\exists\sofinsot'\Subset\cohsdfs'}\event{\sofinsot}=\event{\sofinsot'}\\
&\ifandonlyif\group{\exists\sofinsot'\Subset\finitesetsofthings}
\group[\big]{\sofinsot'\Subset\finfindesirify(\filter)\text{ and }\event{\sofinsot}=\event{\sofinsot'}}\\
&\ifandonlyif\group{\exists\sofinsot'\Subset\finitesetsofthings}
\group[\big]{\event{\sofinsot'}\in\filter\text{ and }\event{\sofinsot}=\event{\sofinsot'}}\\
&\ifandonlyif\event{\sofinsot}\in\filter,
\end{align*}
where the second equivalence follows from the definition of~\(\finfilterise(\cohsds')\) and the fourth equivalence follows from Lemma~\ref{lem:characterisation:desirify:finitary}.

Statements~\ref{it:the:finitary:representation:theorem:finite:cohsds:order:preserving} and~\ref{it:the:finitary:representation:theorem:finite:filter:order:preserving} follow readily from the definitions of the maps~\(\finfinfilterise\) and~\(\finfindesirify\).

Before we turn to the proofs of statements~\ref{it:the:finitary:representation:theorem:finite:filterise:bounds} and~\ref{it:the:finitary:representation:theorem:finite:desirify:bounds}, it will be helpful to recall from  Propositions~\ref{prop:coherent:sdfs:constitute:a:complete:lattice} and~\ref{prop:the:smallest:cohsdfs} that \(\cohsdfsbottom=\cset{\finsot\in\finitesetsofthings}{\finsot\cap\beautifulthings\neq\emptyset}\) and \(\cohsdfstop=\finitesetsofthings\), and from Proposition~\ref{prop:filters:constititute:a:complete:lattice} and Theorem~\ref{thm:basic:sets:constitute:a:bounded:lattice} that \(\finfineventfiltersbottom=\set{\cohsdts}\) and \(\finfineventfilterstop=\finfinevents\).

For statement~\ref{it:the:finitary:representation:theorem:finite:filterise:bounds}, consider any~\(\sofinsot\Subset\cohsdfsbottom\) and~\(\finsot\in\sofinsot\),\footnote{The case that~\(\sofinsot=\emptyset\) is also covered, because \(\event{\emptyset}=\cohsdts\) as well; see the discussion after Equation~\eqref{eq:definition:event}.} then \(\finsot\cap\beautifulthings\neq\emptyset\), and therefore, by Lemma~\ref{lem:the:cohsdtify:operator}\ref{it:the:cohsdtify:operator:background}, also \(\cohsdts[\finsot]=\cohsdts\).
Equation~\eqref{eq:definition:event} then guarantees that \(\event{\sofinsot}=\cohsdts\).
But then indeed,  \(\finfinfilterise(\cohsdfsbottom)=\cset{\event{\sofinsot}}{\sofinsot\Subset\cohsdfsbottom}=\set{\cohsdts}=\finfineventfiltersbottom\).
For the tops, we find that
\begin{equation*}
\finfinfilterise(\cohsdfstop)
=\finfinfilterise(\finitesetsofthings)
=\cset{\event{\sofinsot}}{\sofinsot\Subset\finitesetsofthings}
=\finfinevents
=\finfineventfilterstop.
\end{equation*}

For statement~\ref{it:the:finitary:representation:theorem:finite:desirify:bounds}, consider the following chain of equivalences, for any~\(\finsot\in\finitesetsofthings\):
\begin{equation*}
\finsot\in\finfindesirify(\finfineventfiltersbottom)
\ifandonlyif\finsot\in\finfindesirify(\set{\cohsdts})
\ifandonlyif\cohsdts[\finsot]=\cohsdts
\ifandonlyif\finsot\cap\beautifulthings\neq\emptyset
\ifandonlyif\finsot\in\cohsdfsbottom,
\end{equation*}
where the third equivalence follows from Lemma~\ref{lem:the:cohsdtify:operator}\ref{it:the:cohsdtify:operator:background}.
For the tops, we find that
\begin{equation*}
\finfindesirify(\finfineventfilterstop)
=\finfindesirify(\finfinevents)
=\cset{\finsot\in\finitesetsofthings}{\event{\set{\finsot}}\in\finfinevents}
=\finitesetsofthings
=\cohsdfstop.
\end{equation*}

Finally, we turn to statement~\ref{it:the:finitary:representation:theorem:finite:completeness}.
Consider that the proper filter~\(\filter\) is prime if and only if for all~\(\sofinsot[1],\sofinsot[2]\Subset\finitesetsofthings\),
\begin{equation}\label{eq:the:finitary:representation:theorem:finite:completeness:filters}
\fineventwithindex{1}\cup\fineventwithindex{2}\in\filter
\then\group{\fineventwithindex{1}\in\filter\text{ or }\fineventwithindex{2}\in\filter}.
\end{equation}
Since we've already argued in the proof of Theorem~\ref{thm:basic:sets:constitute:a:bounded:lattice} that \(\fineventwithindex{1}\cup\fineventwithindex{2}=\event{\sofinsot}\), with \(\sofinsot\coloneqq\cset{\finsot[1]\cup\finsot[2]}{\finsot[1]\in\sofinsot[1]\text{ and }\finsot[2]\in\sofinsot[2]}\), it follows from Proposition~\ref{prop:it:does:not:matter:which}\ref{it:it:does:not:matter:which:equality:finite} that the statement~\eqref{eq:the:finitary:representation:theorem:finite:completeness:filters} is equivalent to
\begin{equation}\label{eq:the:finitary:representation:theorem:finite:completeness:cohsdss}
\cset{\finsot[1]\cup\finsot[2]}{\finsot[1]\in\sofinsot[1]\text{ and }\finsot[2]\in\sofinsot[2]}\Subset\cohsdfs
\then\group{\sofinsot[1]\Subset\cohsdfs\text{ or }\sofinsot[2]\Subset\cohsdfs}.
\end{equation}
Assume now that \(\filter\) is prime, and apply the condition~\eqref{eq:the:finitary:representation:theorem:finite:completeness:cohsdss} with \(\sofinsot[1]\coloneqq\set{\finsot[1]}\) and \(\sofinsot[2]\coloneqq\set{\finsot[2]}\) to find the completeness condition~\eqref{eq:complete:cohsdfs}.
Conversely, assume that~\(\cohsdfs\) is complete, and consider any~\(\sofinsot[1],\sofinsot[2]\Subset\finitesetsofthings\).
To prove that condition~\eqref{eq:the:finitary:representation:theorem:finite:completeness:cohsdss} is satisfied, we assume that \(\cset{\finsot[1]\cup\finsot[2]}{\finsot[1]\in\sofinsot[1]\text{ and }\finsot[2]\in\sofinsot[2]}\Subset\cohsdfs\) and at the same time~\(\sofinsot[1]\not\Subset\cohsdfs\), and we prove that then necessarily \(\sofinsot[2]\Subset\cohsdfs\).
It follows from the assumptions that there's some~\(\finsot[o]\in\sofinsot[1]\) such that \(\finsot[o]\notin\cohsdfs\), while at the same time \(\cset{\finsot[o]\cup\finsot[2]}{\finsot[2]\in\sofinsot[2]}\Subset\cohsdfs\).
But then the completeness condition~\eqref{eq:complete:cohsds} guarantees that \(\finsot[2]\in\cohsdfs\) for all~\(\finsot[2]\in\sofinsot[2]\), whence, indeed, \(\sofinsot[2]\Subset\cohsdfs\).
\end{proof}

\begin{lemma}\label{lem:characterisation:finfindesirify}
Let \(\filter\) be any proper filter on~\(\structure{\finfinevents,\subseteq}\) and let~\(\sofinsot\Subset\finitesetsofthings\) be any finite SDFS, then \(\sofinsot\Subset\finfindesirify(\filter)\ifandonlyif\event{\sofinsot}\in\filter\).
\end{lemma}

\begin{proof}
The proof is very similar to the proof of Lemma~\ref{lem:characterisation:desirify:finitary}, but we nevertheless include it here for the sake of completeness.

Consider the following chain of equivalences, valid for any~\(\sofinsot\Subset\finitesetsofthings\):
\begin{align*}
\sofinsot\Subset\finfindesirify(\filter)
\ifandonlyif\group{\forall\finsot\in\sofinsot}\finsot\in\finfindesirify(\filter)
\ifandonlyif\group{\forall\finsot\in\sofinsot}\cohsdts[\finsot]\in\filter
\ifandonlyif\event{\sofinsot}\in\filter,
\end{align*}
where the last equivalence follows from Equation~\eqref{eq:definition:event} and the fact that \(\filter\) satisfies~\ref{axiom:lattice:filters:intersections} and~\ref{axiom:lattice:filters:increasing} [observe, by the way, that \(\event{\emptyset}=\cohsdts\in\filter\) when \(\filter\) is a proper filter].
\end{proof}

\begin{proof}[Proof of Theorem~\ref{thm:conjunctive:representation:finite:finite}]
The proof is very similar to the proof of Theorem~\ref{thm:conjunctive:representation:finite}, but we nevertheless include it here for the sake of completeness.

\underline{\ref{it:conjunctive:representation:finite:finite:consistency}}.
That \(\event{\sofinsot}=\cset{\cohsdt\in\cohsdts}{\sofinsot\Subset\cohsdfs[\cohsdt]}\) follows from Equation~\eqref{eq:connection:with:binary:model:finitary:finite}.
For necessity, suppose that \(\cohsdfs\) is finitely consistent.
That means that there's some finitely coherent~\(\cohsdfs[o]\in\cohsdfss\) such that \(\cohsdfs\subseteq\cohsdfs[o]\).
But it then follows from Proposition~\ref{prop:filter:consistency} that \(\event{\sofinsot}\neq\emptyset\) for all~\(\sofinsot\Subset\cohsdfs[o]\), and therefore in particular also for all~\(\sofinsot\Subset\cohsdfs\).

For sufficiency, assume that \(\event{\sofinsot}\neq\emptyset\) for all~\(\sofinsot\Subset\cohsdfs\).
Then we infer from Lemma~\ref{lem:conjunctive:representation:finite:finite}\ref{it:conjunctive:representation:finite:finite:filterbase} that \(\filterbase[\cohsdfs]\coloneqq\cset{\event{\sofinsot}}{\sofinsot\Subset\cohsdfs}\) is a filter base.
Let's denote by \(\filter[\cohsdfs]\) the (smallest) proper filter that it generates; see Lemma~\ref{lem:conjunctive:representation:finite:finite}\ref{it:conjunctive:representation:finite:finite:filter}.
If we now let \(\cohsdfs'\coloneqq\finfindesirify(\filter[\cohsdfs])\), then it follows from Theorem~\ref{thm:the:finitary:representation:theorem:finite}\ref{it:the:finitary:representation:theorem:finite:if:filter:then:cohsds} that \(\cohsdfs'\) is a finitely coherent SDFS: \(\cohsdfs'\in\cohsdfss\).
We're therefore done if we can prove that \(\cohsdfs\subseteq\cohsdfs'\).
To this end, consider any~\(\finsot\in\cohsdfs\), then \(\set{\finsot}\Subset\cohsdfs\) and therefore also \(\cohsdts[\finsot]=\event{\set{\finsot}}\in\filterbase[\cohsdfs]\).
But then \(\cohsdts[\finsot]\in\filter[\cohsdfs]\), and therefore indeed also \(\finsot\in\finfindesirify(\filter[\cohsdfs])=\cohsdfs'\).

\underline{\ref{it:conjunctive:representation:finite:finite:closure}}.
If \(\cohsdfs\) is not finitely consistent then it follows from~\ref{it:conjunctive:representation:finite:finite:consistency} that there's some~\(\sofinsot\Subset\cohsdfs\) for which \(\event{\sofinsot}=\cset{\cohsdt\in\cohsdts}{\sofinsot\Subset\cohsdfs[\cohsdt]}=\emptyset\) and therefore \(\bigcap_{\cohsdt\in\cohsdts\colon\sofinsot\Subset\sdfsify{\cohsdt}}\sdfsify{\cohsdt}=\bigcap_{\cohsdt\in\event{\sofinsot}}\sdfsify{\cohsdt}=\finitesetsofthings\), making sure that both the left-hand side and the right-hand side in~\ref{it:conjunctive:representation:finite:finite:closure} are equal to~\(\finitesetsofthings\).
We may therefore assume that \(\cohsdfs\) is finitely consistent.
If we borrow the notation from the argumentation above in the proof of~\ref{it:conjunctive:representation:finite:finite:consistency}, and take into account Lemma~\ref{lem:conjunctive:representation:finite:finite}\ref{it:conjunctive:representation:finite:finite:liminf} and Equation~\eqref{eq:connection:with:binary:model:finitary:finite}, then it's clear that we have to prove that \[
\cohsdfsclosure(\cohsdfs)
=\bigcup_{V\in\filter[\cohsdfs]}\bigcap_{\cohsdt\in V}\sdfsify{\cohsdt}.
\]
If we take into account Equation~\eqref{eq:from:filter:to:cohsdfs}, then we see that \(\bigcup_{V\in\filter[\cohsdfs]}\bigcap_{\cohsdt\in V}\sdfsify{\cohsdt}=\finfindesirify(\filter[\cohsdfs])=\cohsdfs'\), so we need to prove that \(\cohsdfs'=\cohsdfsclosure(\cohsdfs)\).
Since we already know from the proof of~\ref{it:conjunctive:representation:finite:finite:consistency} that \(\cohsdfs'\) is finitely coherent and that \(\cohsdfs\subseteq\cohsdfs'\), we find that \(\cohsdfsclosure(\cohsdfs)\subseteq\cohsdfs'\), because \(\cohsdfsclosure(\cohsdfs)\) is the smallest finitely coherent SDFS that the finitely consistent~\(\cohsdfs\) is included in.
To prove the converse inclusion, namely that \(\cohsdfs'\subseteq\cohsdfsclosure(\cohsdfs)\), it suffices to consider any~\(\cohsdfs''\in\cohsdfss\) such that \(\cohsdfs\subseteq\cohsdfs''\), and prove that then~\(\cohsdfs'\subseteq\cohsdfs''\).
Now consider any~\(E\in\filter[\cohsdfs]\), then by Lemma~\ref{lem:conjunctive:representation:finite:finite}\ref{it:conjunctive:representation:finite:finite:filter} there are~\(\sofinsot,\sofinsot'\Subset\finitesetsofthings\) such that \(E=\event{\sofinsot'}\) and \(\sofinsot\Subset\cohsdfs\) and \(\event{\sofinsot}\subseteq\event{\sofinsot'}\).
But then also \(\sofinsot\Subset\cohsdfs''\), and therefore \(\sofinsot'\Subset\cohsdfs''\), by Proposition~\ref{prop:it:does:not:matter:which}.
This implies that also \(E=\event{\sofinsot'}\in\finfinfilterise(\cohsdfs'')\), so we can conclude that \(\filter[\cohsdfs]\subseteq\finfinfilterise(\cohsdfs'')\).
But then indeed also \(\cohsdfs'=\finfindesirify(\filter[\cohsdfs])\subseteq\finfindesirify(\finfinfilterise(\cohsdfs''))=\cohsdfs''\), where the inclusion follows from Theorem~\ref{thm:the:finitary:representation:theorem:finite}\ref{it:the:finitary:representation:theorem:finite:filter:order:preserving} and the last equality follows from Theorem~\ref{thm:the:finitary:representation:theorem:finite}\ref{it:the:finitary:representation:theorem:finite:cohsds:embedding}.

\underline{\ref{it:conjunctive:representation:finite:finite:coherence}}.
The proof is now straightforward, given~\ref{it:conjunctive:representation:finite:finite:closure}, since finitely coherent means finitely consistent and closed with respect to the \(\cohsdfsclosure\)-operator.
\end{proof}

\begin{lemma}\label{lem:conjunctive:representation:finite:finite}
Consider any SDFS~\(\cohsdfs\subseteq\finitesetsofthings\) such that \(\event{\sofinsot}\neq\emptyset\) for all~\(\sofinsot\Subset\cohsdfs\), then the following statements hold:
\begin{enumerate}[label=\upshape(\roman*),leftmargin=*]
\item\label{it:conjunctive:representation:finite:finite:filterbase} the set \(\filterbase[\cohsdfs]\coloneqq\cset{\event{\sofinsot}}{\sofinsot\Subset\cohsdfs}\) is a filter base;
\item\label{it:conjunctive:representation:finite:finite:filter} the smallest proper filter~\(\filter[\cohsdfs]\) that includes~\(\filterbase[\cohsdfs]\) is given by
\begin{equation*}
\filter[\cohsdfs]
\coloneqq\finfineventUpset{\filterbase[\cohsdfs]}
=\cset[\big]{\event{\sofinsot[o]}}{\sofinsot[o]\Subset\finitesetsofthings\text{ and }\group{\exists\sofinsot\Subset\cohsdfs}\event{\sofinsot}\subseteq\event{\sofinsot[o]}};
\end{equation*}
\item\label{it:conjunctive:representation:finite:finite:liminf} \(\bigcup_{V\in\filter[\cohsdfs]}\bigcap_{\cohsdt\in V}\sdfsify{\cohsdt}=\bigcup_{B\in\filterbase[\cohsdfs]}\bigcap_{\cohsdt\in B}\sdfsify{\cohsdt}=\bigcup_{\sofinsot\Subset\cohsdfs}\bigcap_{\cohsdt\in\event{\sofinsot}}\sdfsify{\cohsdt}\).
\end{enumerate}
\end{lemma}

\begin{proof}
The proof is very similar to the proof of Lemma~\ref{lem:conjunctive:representation:finite}, but we nevertheless include it here for the sake of completeness.

\underline{\ref{it:conjunctive:representation:finite:finite:filterbase}}.
Any non-empty collection of non-empty sets that is closed under finite intersections is in particular a filter base; see the discussion in Section~\ref{sec:filters}, and in particular condition~\eqref{eq:directed:downwards}.
First of all, observe that the set~\(\filterbase[\cohsdfs]\) is non-empty because \(\emptyset\Subset\cohsdfs\), and that it contains no empty set by assumption.
To show that it's closed under finite intersections, and therefore directed downwards, consider any~\(\sofinsot[1],\sofinsot[2]\Subset\cohsdfs\), then we infer from Equation~\eqref{eq:intersection:of:possibles} that \(\event{\sofinsot[1]}\cap\event{\sofinsot[2]}=\event{\sofinsot[1]\cup\sofinsot[2]}\), with still \(\sofinsot[1]\cup\sofinsot[2]\Subset\cohsdfs\), and therefore, indeed, \(\event{\sofinsot[1]}\cap\event{\sofinsot[2]}\in\filterbase[\cohsdfs]\).

\underline{\ref{it:conjunctive:representation:finite:finite:filter}}.
Trivially from the discussion of filter bases in Section~\ref{sec:filters}.

\underline{\ref{it:conjunctive:representation:finite:finite:liminf}}.
The second equality is trivial given the definition of~\(\filterbase[\cohsdfs]\).
For the first inequality, it's clearly enough to prove that \(\bigcup_{V\in\filter[\cohsdfs]}\bigcap_{\cohsdt\in V}\sdfsify{\cohsdt}\subseteq\bigcup_{B\in\filterbase[\cohsdfs]}\bigcap_{\cohsdt\in B}\sdfsify{\cohsdt}\).
But this is immediate, since \ref{it:conjunctive:representation:finite:finite:filter} tells us that for any~\(V\in\filter[\cohsdfs]\) there's some~\(B\in\filterbase[\cohsdfs]\) such that \(B\subseteq V\) and therefore also \(\bigcap_{\cohsdt\in V}\sdfsify{\cohsdt}\subseteq\bigcap_{\cohsdt\in B}\sdfsify{\cohsdt}\).
\end{proof}

\begin{proof}[Proof of Theorem~\ref{thm:prime:filter:representation:finitely:coherent:finite}]
The proof is very similar to the proof of Theorem~\ref{thm:prime:filter:representation:finitely:coherent}, but we nevertheless include it here for the sake of completeness.

Consider any finitely consistent SFDS~\(\cohsdfs\).
For necessity, assume that \(\cohsdfs\) is finitely coherent, then it follows from Theorem~\ref{thm:the:finitary:representation:theorem:finite}\ref{it:the:finitary:representation:theorem:finite:if:cohsds:then:filter} that \(\finfinfilterise(\cohsdfs)\) is a proper filter.
The prime filter representation result in Theorem~\ref{thm:prime:filter:representation} then guarantees that
\[
\finfilterise(\cohsdfs)
=\bigcap\cset[\big]{\primefilter\in\finfineventprimefilters}{\finfinfilterise(\cohsdfs)\subseteq\primefilter},
\]
and applying the map~\(\finfindesirify\) to both sides of this equality leads to
\begin{align*}
\cohsdfs
=\finfindesirify(\finfinfilterise(\cohsdfs))
&=\finfindesirify\group[\bigg]{\bigcap\cset[\big]{\primefilter\in\finfineventprimefilters}{\finfinfilterise(\cohsdfs)\subseteq\primefilter}}\\
&=\bigcap\cset[\big]{\finfindesirify\group{\primefilter}}{\primefilter\in\finfineventprimefilters\text{ and }\finfinfilterise(\cohsdfs)\subseteq\primefilter},
\end{align*}
where the first equality follows from Theorem~\ref{thm:the:finitary:representation:theorem:finite}\ref{it:the:finitary:representation:theorem:finite:cohsds:embedding} and the last one from the fact that \(\finfindesirify\) is an order isomorphism between complete lattices of sets, by Theorem~\ref{thm:the:finitary:representation:theorem:finite}, and because we have seen in Propositions~\ref{prop:coherent:sdfs:constitute:a:complete:lattice} and~\ref{prop:filters:constititute:a:complete:lattice} that intersection plays the role of infimum in these complete lattices.
This then leads to
\begin{align*}
\cohsdfs
&=\bigcap\cset[\big]{\finfindesirify\group{\primefilter}}{\primefilter\in\finfineventprimefilters\text{ and }\finfinfilterise(\cohsdfs)\subseteq\primefilter}\\
&=\bigcap\cset[\big]{\finfindesirify\group{\primefilter}}{\primefilter\in\finfineventprimefilters\text{ and }\cohsdfs\subseteq\finfindesirify(\primefilter)}
=\bigcap\cset[\big]{\cohsdfs'\in\completecohsdfss}{\cohsdfs\subseteq\cohsdfs'},
\end{align*}
where the second equality follows from Theorem~\ref{thm:the:finitary:representation:theorem:finite}\ref{it:the:finitary:representation:theorem:finite:cohsds:embedding}\&\ref{it:the:finitary:representation:theorem:finite:filter:embedding}\&\ref{it:the:finitary:representation:theorem:finite:cohsds:order:preserving}\&\ref{it:the:finitary:representation:theorem:finite:filter:order:preserving} and the last equality follows from the correspondence between prime filters and complete finitely coherent SDSes in Theorem~\ref{thm:the:finitary:representation:theorem:finite}\ref{it:the:finitary:representation:theorem:finite:completeness}.

For sufficiency, assume that \(\cohsdfs=\bigcap\cset[\big]{\cohsdfs'\in\completecohsdfss}{\cohsdfs\subseteq\cohsdfs'}\).
That \(\cohsdfs\) is finitely consistent implies that \(\cohsdfs\neq\finitesetsofthings\), and therefore also \(\cset[\big]{\cohsdfs'\in\completecohsdfss}{\cohsdfs\subseteq\cohsdfs'}\neq\emptyset\).
\(\cohsdfs\) is therefore finitely coherent as the intersection of a non-empty set of finitely coherent SDFSes.
\end{proof}

\begin{proof}[Proof of Proposition~\ref{prop:complete:and:conjunctive:finite}]
The proof of statement~\ref{it:complete:and:conjunctive:finite:sdt:to:sdfs} is very similar to the proof of Theorem~\ref{prop:conjunctive:is:complete}, but we nevertheless include it here for the sake of completeness.
First off, it follows from Proposition~\ref{prop:cohsdt:to:cohsdfs} that \(\sdfsify{\cohsdt}\) is finitely coherent.
Now, consider any~\(\finsot[1],\finsot[2]\in\finitesetsofthings\) and assume that \(\finsot[1]\cup\finsot[2]\in\sdfsify{\cohsdt}\), or equivalently, that
\begin{equation*}
\emptyset\subset(\finsot[1]\cup\finsot[2])\cap\cohsdt=(\finsot[1]\cap\cohsdt)\cup(\finsot[2]\cap\cohsdt),
\end{equation*}
which clearly implies that, indeed, \(\finsot[1]\in\sdfsify{\cohsdt}\) or \(\finsot[2]\in\sdfsify{\cohsdt}\).

The proof of statement~\ref{it:complete:and:conjunctive:finite:sdfs:is:conjunctive} is, in its turn, very similar to the proof of Proposition~\ref{prop:complete:finitary:is:conjunctive}, but marginally simpler.
We also include it here for the sake of completeness.
Observe that \(\sdfsify{\sdtify{\cohsdfs}}\subseteq\cohsdfs\) by Proposition~\ref{prop:cohsdfs:to:cohsdt}, so it remains to prove that \(\cohsdfs\subseteq\sdfsify{\sdtify{\cohsdfs}}\).
So, consider any~\(\finsot\in\cohsdfs\), then there are two possibilities.
If \(\finsot\) is a singleton~\(\set{\thing}\), then necessarily \(\thing\in\sdtify{\cohsdfs}\), and therefore \(\finsot=\set{\thing}\in\sdfsify{\sdtify{\cohsdfs}}\).
Otherwise, it follows from the completeness of~\(\cohsdfs\) that there's some strict subset~\(\finsot'\) of~\(\finsot\) such that \(\finsot'\in\cohsdfs\), and therefore also \(\finsot'\in\cohsdfs\).
Now replace \(\finsot\) with \(\finsot'\) in the argumentation above, and keep the recursion going until we do reach a singleton~\(\set{\thing}\) for which \(\thing\in\sdtify{\cohsdfs}\), or equivalently \(\set{\thing}\in\cohsdfs\), which must happen after a finite number of steps, because \(\finsot\) is finite.
But then necessarily also \(\thing\in\finsot\), so \(\finsot\cap\sdtify{\cohsdfs}\neq\emptyset\), and therefore indeed \(\finsot\in\sdfsify{\sdtify{\cohsdfs}}\).

For the final statements, assume that \(\initialclosure\) is finitary, and consider any finitely coherent SDFS~\(\cohsdfs\).
If~\(\cohsdfs\) is complete, then \ref{it:complete:and:conjunctive:finite:sdfs:is:conjunctive} guarantees that it is also conjunctive.
Conversely, if \(\cohsdfs\) is conjunctive, then Proposition~\ref{prop:cohsdfs:to:cohsdt} guarantees that \(\cohsdt[\cohsdfs]\) is a coherent SDT, and Proposition~\ref{prop:conjunctive:and:coherent:sdfs} that \(\cohsdfs=\cohsdfs[{\cohsdt[\cohsdfs]}]\), and then~\ref{it:complete:and:conjunctive:finite:sdt:to:sdfs} tells us that \(\cohsdfs\) is indeed complete.
The rest of the proof is now immediate.
\end{proof}

\begin{proof}[Proof of Theorem~\ref{thm:finitely:coherent:isomorphism}]
For the proof of statement~\ref{it:finitely:coherent:isomorphism:sdfs:to:sds}, we check that \(\sotUpset\cohsdfs\) satisfies the finite coherence conditions~\ref{axiom:desirable:sets:consistency}--\ref{axiom:desirable:sets:finitary:production}.

\underline{\ref{axiom:desirable:sets:consistency}}.
Assume towards contradiction that \(\emptyset\in\sotUpset\cohsdfs\), then there must be some~\(\finsot\in\cohsdfs\) such that \(\finsot\Subset\emptyset\), and therefore \(\finsot=\emptyset\), contradicting the assumption that \(\cohsdfs\) satisfies~\ref{axiom:desirable:finite:sets:consistency}.

\underline{\ref{axiom:desirable:sets:increasing}}.
Consider any~\(\sot[1]\in\sotUpset\cohsdfs\) and~\(\sot[2]\supseteq\sot[1]\).
Then there is some~\(\finsot[1]\in\cohsdfs\) such that \(\finsot[1]\Subset\sot[1]\), and therefore also \(\finsot[1]\Subset\sot[2]\).
This tells us that, indeed, \(\sot[2]\in\sotUpset\cohsdfs\).

\underline{\ref{axiom:desirable:sets:forbidden}}.
Consider any~\(\sot\in\sotUpset\cohsdfs\).
Then there is some~\(\finsot\in\cohsdfs\) such that \(\finsot\Subset\sot\), and therefore also \(\finsot\setminus\uglythings\Subset\sot\setminus\uglythings\).
Since it follows from~\ref{axiom:desirable:finite:sets:forbidden} that \(\finsot\setminus\uglythings\in\cohsdfs\), this tells us that, indeed, \(\sot\setminus\uglythings\in\sotUpset\cohsdfs\).

\underline{\ref{axiom:desirable:sets:background}}.
Consider any~\(\thing[+]\in\beautifulthings\), then by assumption \(\set{\thing[+]}\in\cohsdfs\), and therefore indeed also \(\set{\thing[+]}\in\sotUpset\cohsdfs\).

\underline{\ref{axiom:desirable:sets:finitary:production}}.
Fix any non-empty~\(\sosot\Subset\sotUpset\cohsdfs\), and for all corresponding selections~\(\selection\in\selections\), some~\(\thing[\selection]\in\initialclosure(\selection(\sosot))\).
Then we must prove that \(\cset{\thing[\selection]}{\selection\in\selections[\sosot]}\in\sotUpset\cohsdfs\).

For any~\(\sot\in\sosot\), there's some~\(\finsot\in\cohsdfs\) such that \(\finsot\Subset\sot\).
Let, with obvious notations, \(\finite{\sosot}\coloneqq\cset{\finsot}{\sot\in\sosot}\).
Then \(\finite{\sosot}\Subset\cohsdfs\), so \(\finite{\sosot}\) is a finite set of finite subsets of~\(\things\), which in turn implies that the set of selections~\(\selections[{\finite{\sosot}}]\) is finite as well.
For any~\(\finite{\selection}\in\selections[\finite{\sosot}]\), we consider the corresponding selection \(\selection_{\finite{\selection}}\in\selections\), defined by \(\selection_{\finite{\selection}}(\sot)\coloneqq\finite{\selection}(\finsot)\in\finsot\Subset\sot\), for all \(\sot\in\sosot\), again with obvious notations, implying that \(\selection_{\finite{\selection}}(\sosot)=\finite{\selection}(\finite{\sosot})\).
If we now let \(\thing[\finite{\selection}]\coloneqq\thing[{\selection_{\finite{\selection}}}]\), then it follows from the argumentation above that \(\thing[\finite{\selection}]\in\initialclosure(\finite{\selection}(\finite{\sosot}))\), simply because by assumption \(\thing[\selection_{\finite{\selection}}]\in\initialclosure(\selection_{\finite{\selection}}(\sosot))\).
Clearly,
\begin{equation}\label{eq:finite:coherence:of:upset}
\cset{\thing[\finite{\selection}]}{\finite{\selection}\in\selections[{\finite{\sosot}}]}
=\cset{\thing[\selection_{\finite{\selection}}]}{\finite{\selection}\in\selections[{\finite{\sosot}}]}
\subseteq\cset{\thing[\selection]}{\selection\in\selections[\sosot]}.
\end{equation}
We infer from the finite coherence of~\(\cohsdfs\) that \(\cset{\thing[\finite{\selection}]}{\finite{\selection}\in\selections[{\finite{\sosot}}]}\in\cohsdfs\) [use~\ref{axiom:desirable:finite:sets:production} and the finiteness of~\(\sofinsot\) and~\(\selections[{\finite{\sosot}}]\)].
But then Equation~\eqref{eq:finite:coherence:of:upset} tells us that, indeed, \(\cset{\thing[\selection]}{\selection\in\selections[\sosot]}\in\sotUpset{\cohsdfs}\).

On to the proof of statement~\ref{it:finitely:coherent:isomorphism:sds:to:sdfs}.
It is trivially verified that \(\finpart{\cohsds}\) satisfies the conditions~\ref{axiom:desirable:finite:sets:consistency}--\ref{axiom:desirable:finite:sets:background}, so we concentrate on~\ref{axiom:desirable:finite:sets:production}.
Consider any non-empty~\(\sofinsot\Subset\finpart{\cohsds}\), and for all~\(\selection\in\finselections\), any choice~\(\thing[\selection]\in\initialclosure(\selection(\sofinsot))\).
Since clearly also \(\sofinsot\Subset\cohsds\), we infer from~\ref{axiom:desirable:sets:finitary:production} [or from~~\ref{axiom:desirable:sets:production}] that \(\cset{\thing[\selection]}{\selection\in\finselections}\in\cohsds\).
Since \(\finselections\) is a subset of the set of all maps from the finite set \(\sofinsot\) to the finite set \(\bigcup\sofinsot\), it follows that also \(\cset{\thing[\selection]}{\selection\in\finselections}\in\finitesetsofthings\).
We conclude that, indeed, \(\cset{\thing[\selection]}{\selection\in\finselections}\in\finpart{\cohsds}\).

For the proof of statement~\ref{it:finitely:coherent:isomorphism:up:after:fin}, observe that this is essentially Proposition~\ref{prop:finitary:equal:to:finitary:part}.

And finally, for the proof of statement~\ref{it:finitely:coherent:isomorphism:fin:after:up}, observe that, trivially, \(\cohsdfs\subseteq\finpart{\sotUpset\cohsdfs}\).
Conversely, consider any~\(\finsot\in\finpart{\sotUpset\cohsdfs}\), then \(\finsot\in\finitesetsofthings\) and \(\finsot\in\sotupset\cohsdfs\), so there is some~\(\finsot'\in\cohsdfs\) such that \(\finsot'\Subset\finsot\).
But then \ref{axiom:desirable:finite:sets:increasing} guarantees that also \(\finsot\in\cohsdfs\).
\end{proof}

\subsection{Proofs of results in Section~\ref{sec:propositional:logic}}\label{sec:proofs:propositional:logic}

\begin{proof}[Proof of Proposition~\ref{prop:propositional:basic}]
First, fix any~\(\cohsdt\in\cohsdts\) such that \(\cohsdfs\subseteq\sdfsify{\cohsdt}\).
Consider any~\(\finsot\in\cohsdfs\), then clearly also \(\finsot\cap\cohsdt\neq\emptyset\).
Consider any~\(\thing[\finsot]\in\finsot\cap\cohsdt\) [which is always possible], then \(\thing[\finsot]\entails\bigvee\finsot\), and therefore also \(\bigvee\finsot\in\cohsdt\), because \(\cohsdt\) is deductively closed.
This tells us that \(\cohsdt(\cohsdfs)\subseteq\cohsdt\).
It therefore suffices to prove~\ref{it:propositional:basic:is:coherent} in order to also prove~\ref{it:propositional:basic:is:smallest}.
And \ref{it:propositional:basic:is:smallest} then readily leads to~\ref{it:propositional:basic:representation}, taking into account Corollary~\ref{cor:conjunctive:representation:finitely:coherent:finite}.

So, let's concentrate on the proof of~\ref{it:propositional:basic:is:coherent}.
First, we prove that \(\cohsdt(\cohsdfs)\) is deductively closed.
So consider any~\(\finsot\in\cohsdfs\) and any~\(\thing\in\things\) such that \(\bigvee\finsot\entails\thing\).
Then we have to prove that also \(\thing\in\cohsdt(\cohsdfs)\).
To this end, let \(\finsot'\coloneqq\finsot\cup\set{\thing}\) then still \(\finsot'\in\cohsdfs\) by coherence [\ref{axiom:desirable:finite:sets:increasing}], and therefore also \(\bigvee\finsot'\in\cohsdt(\cohsdfs)\).
But we infer from~\(\bigvee\finsot\entails\thing\) that \(\bigvee\finsot'=\thing\), and therefore indeed \(\thing\in\cohsdt(\cohsdfs)\).

It now only remains to prove that \(\cohsdt(\cohsdfs)\) is consistent.
Assume towards contradiction that it isn't, so \(\cohsdt(\cohsdfs)\cap\uglythings\neq\emptyset\).
Consider any~\(\thing[-]\in\cohsdt(\cohsdfs)\cap\uglythings\), so there's some~\(\finsot\in\finitesetsofthings\) for which \(\thing[-]=\bigvee\finsot\).
This can only happen if all wffs in~\(\finsot\) are contradictions, so \(\finsot\subseteq\uglythings\).
But then~\ref{axiom:desirable:sets:forbidden} implies that~\(\emptyset\in\cohsdfs\), contradicting~\ref{axiom:desirable:finite:sets:consistency}.
\end{proof}

\section{Notation}\label{sec:symbols}
In this appendix, we provide a list of the most commonly used and most important notation, and where it is defined or first introduced.

\begin{center}
\begin{longtable}{lp{0.6\linewidth}l}
\textbf{notation}&\textbf{meaning}&\textbf{introduced where}\\
\(\things\) & set of all things & Section~\ref{sec:desirable:things}\\
\(\powerset\group{\things}\) & set of all sets of things & Section~\ref{sec:desirable:things}\\
SDT & abbreviation for `\SDT' & Section~\ref{sec:desirable:things}\\
\(\sot\), \(\sot[i]\) & generic sets of things & Section~\ref{sec:desirable:things}\\
\(\closure\) & closure operator & Section~\ref{sec:desirable:things}\\
\(\initialclosure\) & closure operator associated with \(\things\) & Section~\ref{sec:desirable:things}\\
\(\uglythings\) & set of all forbidden things & Section~\ref{sec:desirable:things}\\
\(\beautifulthings\) & the smallest closed SDT & Section~\ref{sec:desirable:things}\\
\(\cohsdt\), \(\cohsdt[i]\) & generic coherent SDTs & Section~\ref{sec:desirable:things}\\
\(\toppedcohsdts\) & set of all closed SDTs & Section~\ref{sec:desirable:things}\\
\(\cohsdts\) & set of all coherent SDTs & Equation~\eqref{eq:definition:cohsdts}\\
\(\cohsdttop\) & top of the complete lattice \(\structure{\toppedcohsdts,\subseteq}\) & Section~\ref{sec:desirable:things}\\
\(\cohsdtbottom\) & bottom of the complete lattice \(\structure{\toppedcohsdts,\subseteq}\) & Section~\ref{sec:desirable:things}\\
SDS & abbreviation for `\SDS' & Section~\ref{sec:desirability:sds}\\
\(\sosot\), \(\sosot[i]\) & generic sets of sets of things & Section~\ref{sec:desirability:sds}\\
\(\selections\) & set of all selection maps on \(\sosot\) & Section~\ref{sec:desirability:sds}\\
\(\selection\), \(\selection[o]\) & generic selection maps & Section~\ref{sec:desirability:sds}\\
\(\cohsds\), \(\cohsds[i]\) & generic coherent SDSes & Section~\ref{sec:desirability:sds}\\
\(\cohsdss\) & set of all coherent SDSes & Section~\ref{sec:desirability:sds}\\
\(\toppedcohsdss\) & the complete lattice \(\toppedcohsdss\coloneqq\cohsdss\cup\set{\powerset\group{\things}}\) & Section~\ref{sec:desirability:sds}\\
\(\cohsdsclosure\) & closure operator associated with \(\toppedcohsdss\) & Section~\ref{sec:desirability:sds}\\
\(\cohsdstop\) & top of the complete lattice \(\structure{\toppedcohsdss,\subseteq}\) & Section~\ref{sec:desirability:sds}\\
\(\cohsdsbottom\) & bottom of the complete lattice \(\structure{\toppedcohsdss,\subseteq}\) & Section~\ref{sec:desirability:sds}\\
\(\fincohsdss\) & set of all finitely coherent SDSes & Section~\ref{sec:desirability:sds:finitary}\\
\(\toppedfincohsdss\) & the complete lattice \(\toppedfincohsdss\coloneqq\fincohsdss\cup\set{\powerset\group{\things}}\) & Section~\ref{sec:desirability:sds:finitary}\\
\(\fincohsdsclosure\) & closure operator associated with \(\toppedfincohsdss\) & Section~\ref{sec:desirability:sds:finitary}\\
\(\sdtify{\sosot}\) & SDT associated with the SDS \(\sosot\) & Equation~\eqref{eq:sos:to:sot:and:sot:to:sos}\\
\(\sdsify{\sot}\) & SDS associated with the SDT \(\sot\) & Equation~\eqref{eq:sos:to:sot:and:sot:to:sos}\\
\(\sdtify{\bolleke}\) & map turning an SDS into an SDT & Section~\ref{sec:desirability:sds:finitary}\\
\(\sdsify{\bolleke}\) & map turning an SDT into an SDS & Section~\ref{sec:desirability:sds:finitary}\\
\(\sdsify{\sdtify{\cohsds}}\) & conjunctive part of the coherent SDS \(\cohsds\) & Section~\ref{sec:desirability:sds:finitary}\\
\(\finitesetsofthings\) & set of all finite sets of things & Section~\ref{sec:desirability:sdfs}\\
SDFS & abbreviation for `\SDFS' & Section~\ref{sec:desirability:sdfs}\\
\(\cohsdfss\) & set of all finitely coherent SDFSes & Section~\ref{sec:desirability:sdfs}\\
\(\toppedcohsdfss\) & the complete lattice \(\toppedcohsdfss\coloneqq\cohsdfss\cup\set{\finitesetsofthings}\) & Section~\ref{sec:desirability:sdfs}\\
\(\cohsdfsclosure\) & closure operator associated with \(\toppedcohsdfss\) & Section~\ref{sec:desirability:sdfs}\\
\(\cohsdfstop\) & top of the complete lattice~\(\structure{\toppedcohsdfss,\subseteq}\) & Section~\ref{sec:desirability:sdfs}\\
\(\cohsdfsbottom\) & bottom of the complete lattice~\(\structure{\toppedcohsdfss,\subseteq}\) & Section~\ref{sec:desirability:sdfs}\\
\(\sdfsify{\bolleke}\) & map turning an SDT into an SDFS & Equation~\eqref{eq:sofinsot:to:finsot:and:finsot:to:sofinsot}\\
\(\sdfsify{\sdtify{\cohsdfs}}\) & conjunctive part of the coherent SDFS \(\cohsdfs\) & Section~\ref{sec:desirability:sdfs}\\
\(\filter\) & filter & Section~\ref{sec:towards:representation}\\
\(\idealcohsdt\) & set of things that are actually desirable & Section~\ref{sec:towards:representation}\\
\(\cohsdts[\sot]\) & event corresponding to `the set of things~\(\sot\) is desirable' & Equation~\eqref{eq:the:cohsdtify:operator}\\
\(\event{\sosot}\) & event corresponding to `the set of sets of things~\(\sosot\) is desirable' & Equation~\eqref{eq:definition:event}\\
\(\posetupset{a}\) & set of elements dominating \(a\) in the poset \(\lattice\) & Equation~\eqref{eq:up:operators}\\
\(\posetUpset{A}\) & set of elements dominating some \(a\in A\) in the poset \(\lattice\) & Equation~\eqref{eq:up:operators}\\
\(\basicevent{\sosot}\) & event based on the set of sets of things \(\sosot\) & Equation~\eqref{eq:definition:production:event}\\
\(\events\) & the set of events representing sets of sets of things & Section~\ref{sec:representation:lattices}\\
\(\finevents\) & the set of events representing finite sets of sets of things & Section~\ref{sec:representation:lattices}\\
\(\finfinevents\) & the set of events representing finite sets of finite sets of things & Section~\ref{sec:representation:lattices}\\
\(\structure{\lattice,\leq}\) & generic bounded lattice & Section~\ref{sec:filters}\\
\(\meet\) & meet of a generic bounded lattice & Section~\ref{sec:filters}\\
\(\join\) & join of a generic bounded lattice & Section~\ref{sec:filters}\\
\(\latticetop\) & top of a generic bounded lattice & Section~\ref{sec:filters}\\
\(\latticebottom\) & bottom of a generic bounded lattice & Section~\ref{sec:filters}\\
\(\filters\group{\lattice}\) & set of all filters on~\(\structure{\lattice,\leq}\) & Section~\ref{sec:filters}\\
\(\properfilters\group{\lattice}\) & set of all proper filters on~\(\structure{\lattice,\leq}\) & Section~\ref{sec:filters}\\
\(\filterbase\) & filter base & Section~\ref{sec:filters}\\
\(\filterclosure\) & closure operator associated with~\(\properfilters\group{\lattice}\) & Section~\ref{sec:filters}\\
\(\primefilters\group{\lattice}\) & set of all prime filters on~\(\structure{\lattice,\leq}\) & Section~\ref{sec:filters}\\
\(\fincohsdstop\) & top of the complete lattice \(\structure{\toppedfincohsdss,\subseteq}\) & Section~\ref{sec:representation:finitary}\\
\(\fincohsdsbottom\) & bottom of the complete lattice \(\structure{\toppedfincohsdss,\subseteq}\) & Section~\ref{sec:representation:finitary}\\
\(\finfilterise\) & map turning a finitely coherent SDS into a proper filter on~\(\structure{\finevents,\subseteq}\) & Section~\ref{sec:representation:finitary}\\
\(\findesirify\) & map turning a proper filter on~\(\structure{\finevents,\subseteq}\) into a finitely coherent SDS & Section~\ref{sec:representation:finitary}\\
\(\completesdss\) & set of all complete and coherent SDSes & Section~\ref{sec:representation:finitary}\\
\(\completefincohsdss\) & set of all complete and finitely coherent SDSes & Section~\ref{sec:representation:finitary}\\
\(\eventprincipalfilters\) & set of all principal filters on~\(\structure{\events,\subseteq}\) & Section~\ref{sec:representation}\\
\(\propereventprincipalfilters\) & set of all proper principal filters on~\(\structure{\events,\subseteq}\) & Section~\ref{sec:representation}\\
\(\filterise\) & map turning a coherent SDS into a proper principal filter on~\(\structure{\events,\subseteq}\) & Section~\ref{sec:representation}\\
\(\desirify\) & map turning a proper principal filter on~\(\structure{\events,\subseteq}\) into a coherent SDS & Section~\ref{sec:representation}\\
\(\finpart{\sosot}\) & finite part of the SDS~\(\sosot\) & Section~\ref{sec:finitary:models}\\
\(\fintypart{\sosot}\) & finitary part of the SDS~\(\sosot\) & Section~\ref{sec:finitary:models}\\
\(\finfinfilterise\) & map turning a finitely coherent SDFS into a proper filter on~\(\structure{\finfinevents,\subseteq}\) & Section~\ref{sec:representation:finitary:finite}\\
\(\finfindesirify\) &  map turning a proper filter on~\(\structure{\finfinevents,\subseteq}\) into a finitely coherent SDFS & Section~\ref{sec:representation:finitary:finite}\\
\(\lindenbaum\) & Lindenbaum algebra & Section~\ref{sec:propositional:logic}\\
\(\gbls\) & set of all gambles & Section~\ref{sec:choice:functions}\\
\(\posgbls\) & set of all gambles \(h\gblgt0\) & Section~\ref{sec:choice:functions}\\
\(\neggbls\) & set of all gambles \(h\gbllt0\) & Section~\ref{sec:choice:functions}\\
\(\nonposopts\) & set of all gambles \(h\gbllt0\) or \(h=0\) & Section~\ref{sec:choice:functions}\\
\(E_p\) & expectation operator associated with the probability mass function~\(p\) & Section~\ref{sec:choice:functions}\\
\(\simplex\) & set of all probability mass functions on~\(\states\) & Section~\ref{sec:choice:functions}\\
\(\solp\) & credal set & Section~\ref{sec:choice:functions}
\end{longtable}
\end{center}

\end{document}